\documentclass[10pt]{article}
\usepackage{amsfonts}
\usepackage{amsmath}
\usepackage{amsthm}
\usepackage{amssymb}
\usepackage{latexsym}
\usepackage{graphicx}
\usepackage{epsfig}
\usepackage{url}
\usepackage{ulem}
\usepackage{multicol}
\usepackage{makeidx}
\makeindex
\usepackage{enumerate}
\usepackage{tabularx}
\newtheorem{thm}{Theorem}[section]
\newtheorem{lemma}[thm]{Lemma}

\newtheorem{corollary}[thm]{Corollary}
\newtheorem{prop}[thm]{Proposition}

\theoremstyle{remark}
\newtheorem{definition}[thm]{Definition}

\newtheorem{exa}[thm]{Example}

\def\RR{\mathbb{R}}
\def\CC{\mathbb{C}}
\def\FF{\mathbb{F}}

\def\x{\mathcal{X}}

\def\cM{\mathcal M}

\newcommand{\re}{Re}
\newcommand{\im}{Im}
\newcommand{\nchess}{NCHes}

\newcommand{\Hm}{\mathrm{H}}
\newcommand{\Proj}{P}

\newcommand{\RRx}{\RR\langle x \rangle}
\newcommand{\CCx}{\CC\langle x \rangle}
\newcommand{\FFx}{\FF\langle x \rangle}
\newcommand{\RRNx}{\RR\langle x, x^T \rangle}
\newcommand{\CCNx}{\CC\langle x, x^T \rangle}
\newcommand{\FFNx}{\FF\langle x, x^T
\rangle}
\newcommand{\summ}[2]{\underset{#1}{\overset{#2}{\sum}}}
\newcommand{\eeq}{\notag}
\newcommand{\ncm}{\,}
\newcommand{\gtupn}{(\RR^{n\times n}_{sym})^g}
\newcommand{\h}{\gamma}
\newcommand{\dird}{D}

\newcommand{\Symm}{\mathrm{Symm}}

\newcommand{\Comm}{\mathrm{Comm}}

\newcommand{\Subs}{\mathrm{Subs}}
\newcommand{\NS}{\mathrm{NS}}
\newcommand{\num}{m}

\newcommand{\lap}{\mathrm{Lap}}

\numberwithin{equation}{subsection}

\begin{document}

\title{Noncommutative Harmonic and Subharmonic Polynomials and other Noncommutative Partial Differential Equations}
\author{Christopher S. Nelson}

\maketitle

\begin{abstract}
Solutions to Laplace's equation are called harmonic functions.
Harmonic functions arise in many applications, such as physics and the theory
of stochastic processes.  Of interest classically are harmonic polynomials,
which have a simple classification.
Further, the work of Reznick, building on the work of others, namely Sylvester,
Clifford, Rosanes, Gundelfinger, Cartan, Maass and Helgason, has led
to a classification of all polynomial solutions to a differential equation of
the form
\[q\left(\frac{\partial}{\partial x_1}, \ldots, \frac{\partial}{\partial
x_g}\right)y = 0,\]
where $q$ is a homogeneous polynomial over an algebraically closed field.

The definition of harmonicity can be extended to the space of
polynomials in free variables using the concept of a noncommutative Laplacian.
Given a positive integer $\ell$, the $\ell$-Laplacian of a noncommutative
(abbreviated NC) polynomial $p$ in the direction $h$ is defined to be
$$ \lap_{\ell}[p,h] := \sum_{i=1}^g \frac{d^{\ell}}{dt^{\ell}} p(x_1, \ldots,
x_i + th, \ldots, x_g).$$
A NC polynomial $p$ is said to be $\ell$-harmonic if $\lap_{\ell}[p,h] = 0$.
When $\ell = 2$, then the $\ell$-Laplacian is simply called the Laplacian and a
$\ell$-harmonic NC polynomial is simply called harmonic.
More generally, the concept of a constant coefficient differential equation can
be extended to the space of NC polynomials via the
NC directional derivative.

The main contribution of this paper is to show that a NC polynomial is
$\ell$-harmonic if and only if it is a linear
combination of NC polynomials of the form
\begin{align}
\notag
\prod_{i=1}^d \sum_{j=1}^g a_{ij}x_i,
\end{align}
such that for any $\ell$ integers $1 \leq k_1 < k_2 < \ldots < k_{\ell} \leq d$,
the coefficients $a_{ij}$ satisfy
\[\sum_{j=1}^G \prod_{i=1}^{\ell} a_{k_i j} = 0.\]
This result is analogous to the classification of polynomial solutions to a
differential equation of the form $q(\partial/\partial x_1, \ldots,
\partial/\partial x_g)y = 0$ given by Reznick.
Additionally, this paper proves new results about NC ``subharmonic''
polynomials.

This work extends results of Helton, McAllaster, and Hernandez
on NC harmonic and subharmonic polynomials, which classified
solutions for the $\ell = 2$ case in two noncommuting variables $x_1$ and $x_2$.

\end{abstract}

\newpage

\section{Introduction}

The introduction is divided into three subsections.  $\S$ \ref{sect:definitions} presents the basic definitions for the theory of non-commutative polynomials.  $\S$ \ref{sect:classification1} presents the main results of this paper.  $\S$ \ref{sect:topics} discusses the motivation for the concepts studied in this paper.

\subsection{Definitions}
\label{sect:definitions}

\subsubsection{Noncommutative Polynomials}
\label{subsub:NCPolys}
A noncommutative monomial $m$ of degree $d$\index{degree} on the free variables
$x = (x_1,\,\ldots,\,x_g)$ is a product $x_{a_1}x_{a_2}\cdots x_{a_d}$
corresponding to a unique sequence $\{a_i\}_{i=1}^d$ of
nonnegative integers, $1\le a_i\le g$.  The set of all
monomials in $(x_1,\,\ldots,\,x_g)$ is denoted as $\cM$\index{M@$\cM$}.
The length of a monomial $m$ is denoted as $|m|$.

The space of noncommutative, or NC, polynomials \index{noncommutative polynomial} $p(x) = p(x_1,\ldots,x_g)$
with coefficients in a field $\FF$ is denoted $\FFx$. \index{Fax@$\FFx$}\index{Rax@$\RRx$|see{$\FFx$}}
\index{Cax@$\CCx$|see{$\FFx$}}  Throughout this paper, $\FF$ will be thought of as either $\RR$ or $\CC$. An element $p(x) \in \FFx$ is expressed as
\begin{equation}
p(x)  = \summ{m\,\in\, \cM}{}A_{m}\,m
\eeq\end{equation}
for scalars $A_m \in \FF$.
An example of a NC polynomial is \
$$p(x) = p(x_1,x_2)
 = x_1^2\ncm x_2\ncm x_1 + x_1\ncm x_2\ncm x_1^2 + x_1
\ncm x_2 - x_2\ncm x_1 + 7.$$ \
In commutative variables, this would be equivalent to $2x_1^{3}x_2 + 7 \in \FF[x]$.
Special care will be taken throughout
this paper to
distinguish between spaces
of NC polynomials $\FFx$ and spaces of polynomials in commutative variables $\FF[x]$\index{Fbx@$\FF[x]$}
\index{Cbx@$\mathbb{C}[x]$|see{$\FF[x]$}}
\index{Rbx@$\mathbb{R}[x]$|see{$\FF[x]$}}.

\subsubsection{Degree of a NC Polynomial}

A NC monomial
\[m = x_{a_1}\ldots x_{a_D}\]
 is said to have
\textbf{degree $d$ in $x_i$} \index{degree!in a variable $x_i$}
 if precisely $d$ of the $x_{a_j}$ are equal to $x_i$. This is denoted
$\deg_i(m) = d_i$.\index{degip@$\deg_i(p)$|see{degree in a variable $x_i$}}
For a NC polynomial
\[p(x_1, \ldots, x_d) = \sum_{|m| \leq D} A_{m}m,\]
$\deg_i(p)$ is defined as
\[\deg_i(p) := \max_{A_m \neq 0} \{ \deg_i (m) \}.\]
A NC polynomial $p(x)$ is said to be \textbf{homogeneous of degree $d$ in $x_i$} if
 $p$ is a sum of terms with degree $d$ in $x_i$.

\subsubsection{Transpose of a NC Polynomial}

The \textbf{transpose}\index{transpose!of a polynomial in $\RRx$} of a polynomial $p \in \FFx$ is defined so that it is compatible with matrix transposition.
This paper for the
most part will consider
 polynomials in \textbf{symmetric variables}. \index{variable!symmetric}
 This means that each variable $x_i$
satisfies $x_i^T = x_i$.
The space $\FFNx$ will denote the space where $x_i \neq x_i^T$ and will be studied in $\S$ \ref{sect:ncvars}.
Transposes of monomials in symmetric variables are defined to be
\[(x_{a_1}\ncm \ldots\ncm x_{a_{d}})^T
:=
x_{a_{d}}\ncm\ldots\ncm x_{a_1}.\]
The transpose of a polynomial $p$, denoted $p^T$, is therefore defined by
$p(x)  = \summ{m\,\in\, \cM}{}A_{m}\,m^T,$
and has the following properties:
\begin{quote}
(1) $(p^T)^T = p$ \\
(2) $(p_1 + p_2)^T = p_1^T + p_2^T$ \\
(3) $(\alpha p)^T = \alpha p^T$ \qquad\qquad ($\alpha \in
\RR$)\\
(4) $(p_1\ncm p_2)^T = p_2^T\ncm p_1^T$.
\end{quote}
\textbf{Symmetric (or self-adjoint) polynomials} \index{symmetric!polynomial}
are those that are equal to their
transposes, i.e. those which satisfy $p(x) = p^T(x).$

\subsubsection{Evaluating Noncommutative Polynomials}
\label{sec:openG2}

Let $\gtupn$ denote the set of $g$-tuples
$(X_1,\ldots,X_g)$ of real symmetric $n\times n$ matrices.
One can evaluate a
polynomial $p(x)=p(x_1,\ldots,x_g) \in \RRx$
at a
tuple $X=(X_1,\dots,X_g)\in(\RR^{n\times n}_{sym})^g$.
In this
case, $p(X)$ is also an $n\times n$ matrix and
the involution on $\RRx$
that was introduced earlier is compatible
with matrix transposition, i.e.,
$$
p^T(X)=p(X)^T,
$$
where $p(X)^T$ denotes the transpose of the  matrix  $ p(X)$.
When $X\in \gtupn$ is substituted into $p$,
the constant term  $p(0)$ of $p(x)$  becomes $p(0) I_n$.
For example,
   $$p(x)= 3+x_1^2+5x_2^3  \  \ \implies \ \  p(X)= 3 I_n+X_1^2+5X_2^3.$$

A symmetric polynomial $p \in \RRx$ is {\bf matrix positive}
if $p(X)$ is a positive-semidefinite matrix for each tuple
$X=(X_1,\dots,X_g)\in (\RR^{n\times n}_{sym})^g$. A symmetric polynomial
$p$ is matrix positive if and only if
$p$ is a finite sum of squares
\[ p = \sum_{i} p_i^Tp_i,\]
for some polynomials $p_i \in \RRx$ (See \cite{H02}).
This paper continuously emphasizes that, unless otherwise noted,
$x_1, x_2, \ldots, x_n$ stand for variables and
$X_1, X_2, \ldots, X_n$ stand for matrices (usually symmetric).

\subsubsection{Noncommutative Differentiation}

Define $h$ to be an indeterminate direction parameter.
For some variable $x_i$, define $\dird[p(x_1,\ldots,x_g),x_i,h]$
to be
\begin{equation}
\dird[p(x_1,\ldots,x_g), x_i,h]:= \frac{d}{dt}[p(x_1, \ldots, (x_i+th),
\ldots, x_g)]_{|_{t=0}}.
\end{equation}
This polynomial in $x = (x_1, \ldots, x_g)$ and $h$ is called the \textbf{directional derivative} \index{directional derviative}
\index{D@$\dird[p,q,h]$|see{directional derivative}}
of
$p = p(x_1,\ldots,x_g)$
in $x_i$ in the direction $h$.
Note that the directional derivative
operator is linear in  $h$.
When all variables commute, this directional
derivative is equal to
$h\displaystyle\frac{\partial p}{\partial x_i}(x)$.
For a detailed formal definition
of this directional derivative see \cite{HMV06},
and
for more examples see \cite{CHSY03}.
\begin{exa} The directional derivative of $p = x_1^2\ncm x_2$ is

\begin{tabularx}{\linewidth}{X}
$\dird\left[x_1^2\ncm x_2,x_1,h\right] = \frac{d}{dt}\left[(x_1+th)^2x_2\right]_{|_{t=0}}$\\
\quad\qquad\qquad\qquad$= \frac{d}{dt}\left[x_1^2\ncm x_2+th\ncm x_1\ncm
x_2+tx_1\ncm h\ncm x_2+t^2h^2\ncm x_2\right]_{|_{t=0}}$\\
\quad\qquad\qquad\qquad$
= \left[h\ncm x_1\ncm x_2 + x_1\ncm h\ncm x_2 + 2th^2\ncm
x_2\right]_{|_{t=0}}$\\
\quad\qquad\qquad\qquad$= h\ncm x_1\ncm x_2 + x_1\ncm h\ncm x_2$.\\
\end{tabularx}
\end{exa}\qed

\

As in this example, in general the directional derivative of $p$ on $x_i$
in the direction $h$ is simply the sum of all possible terms produced from $p$ by replacing
one instance of \mbox{$x_i$ with $h$.}

\begin{lemma}
The directional derivative
of  NC polynomials is linear,
\begin{equation}
\dird[a\,p + b\,q,x_i,h]
= a\,\dird[p,x_i,h] + b\,\dird[q,x_i,h]
\eeq\end{equation}
and respects transposes
\begin{equation}
\dird\left[p^T,x_i,h\right]
= \dird[p,x_i,h]^T.
\eeq\end{equation}
\end{lemma}

\proof \ Straighforward.

\qed

Define $\dird[p, ax_i + bx_j, h]$ to be
\[\dird[p, ax_i + bx_j, h] = a\dird[p, x_i, h] + b\dird[p, x_j, h],\]
and define $\dird[p, x_{a_1}\ldots x_{a_k}, h]$ to be
the $k^{th}$ derivative
\[\dird[p, x_i^k, h] = \overbrace{\dird[\dird[\ldots\dird}^{k\ \mathrm{times}}[p, x_{a_1}, h], x_{a_2}, h], \ldots, x_{a_k}, h].\]
For any commutative polynomial $q \in \FF[x]$, the derivative
$\dird[p, q, h]$ is defined in the obvious way. The polynomial
$q$ here is called
the \textbf{full symbol} \index{full symbol} and is analogous to the full symbol
in the commutative case.
\begin{exa}If $p = x_1^3 + x_1x_2 \in \FFx$ and $q = x_1^2 + 2x_2 \in \FF[x]$,
 then
 \begin{align}
 \dird[p, q, h] &=  \dird\left[x_1^3 + x_1 x_2,x_1^2 + 2x_2, h\right]\notag \\
 &= \dird\left[\dird\left[x_1^3 + x_1x_2, x_1,h\right],x_1,h\right] + 2\dird\left[x_1^3 + x_1x_2, x_2,h\right]\notag\\
 &=2x_1h^2 + 2hx_1h + 2h^2x_1 + 2x_1h.\notag
 \end{align}
 \qed
\end{exa}
If $q$ is homogeneous of degree $\ell$, then
when all variables commute,
$\dird[p,q,h]$ is simply equal to
\begin{align}
\label{eq:RefOnCommCollapseDeriv}
\dird[p,q,h] & = h^{\ell}q\left(\frac{\partial}{\partial x_1}, \ldots
\frac{\partial}{\partial x_g}\right)p = h^{\ell} q(\nabla)p,
\end{align}
\index{$\nabla$}
\index{qgradient@$q(\nabla)$}
where $q(\nabla)$ is the differential operator given by evaluating $q$ at the
$g$-tuple $(\partial/\partial x_1, \ldots, \partial/\partial x_g)$.
If $q \in \FF[x]$ is any polynomial, then a differential operator $\dird[\cdot, q, h]$
is called a \textbf{constant-coefficient} NC differential operator.
The \textbf{order} \index{order of a differential operator}
of $\dird[\cdot, q, h]$ is the degree of
$q$, and $\dird[\cdot, q, h]$ is of \textbf{homogeneous order}
if $q$ is homogeneous.

\subsection{Main Results}
\label{sect:classification1}

\subsubsection{Non-Commutative Harmonic and Subharmonic Polynomials}

This paper is primarily interested in the concept of NC harmonic and subharmonic polynomials.
The notion of a NC Laplacian, NC
harmonic, and NC subharmonic polynomials was first introduced
in \cite{HMH09}.
The \textbf{Laplacian}\index{Laplacian} of a NC polynomial $p$ is
\begin{align}
\lap[p, h] &:= \dird\left[p, \summ{i=1}{g} x_i^2, h\right] \\
&=
\summ{i=1}{g}\frac{d^2}{dt^2}
[p(x_1,\ldots,(x_i+th),\ldots,x_g)]_{|_{t=0}}.
\end{align}
When the variables commute,
$\lap[p,h]$ is
$h^2
\Delta \bigl[p \bigr]
$, the standard Laplacian on $\RR^n$
\[
\Delta \bigl[p \bigr] := \summ{i=1}{g}\frac{\partial^2 p}{\partial x_i^2}.
\]

A NC  polynomial is called {\bf harmonic} \index{harmonic}
if its Laplacian is zero. A NC polynomial is called
 {\bf subharmonic} \index{subharmonic}
if its Laplacian\index{Laplacian} is matrix positive---or equivalently, if
its Laplacian is a finite sum of squares $\sum_i P_i^T P_i$ (see \cite{H02}). A
subharmonic
polynomial is called
{\bf purely subharmonic} \index{subharmonic!purely subharmonic}
 if it is not harmonic, i.e.,
if its Laplacian is nonzero and
matrix positive.

For a discussion of the motivation behind studying NC harmonic and subharmonic
polynomials, see $\S$ \ref{sect:topics}.

\subsubsection{$\ell$-Harmonic Polynomials}
\label{subsub:ClassIntro}
The
\textbf{$\ell$-Laplacian},\index{lLaplacian@$\ell$-Laplacian}
\index{Lapl@$\lap_{\ell}$|see{$\ell$-Laplacian}} denoted $\lap_{\ell}[p,h]$ is
defined by
\[ \lap_{\ell}[p,h] = \dird\left[p, \sum_{i=1}^g x_i^{\ell}, h\right].\]
A NC polynomial $p \in \FFx$ is
\textbf{$\ell$-harmonic}\index{lharmonic@$\ell$-harmonic} if $\lap_{\ell}[p,h] =
0.$
In the
special case that $\ell = 2$, a
NC $2$-harmonic polynomial is simply said to be
NC harmonic\index{harmonic}, and the NC $2$-Laplacian is the NC
Laplacian.\index{Laplacian} \index{Lap@$\lap$|see{Laplacian}}

When the variables commute,
$\lap_{\ell}[p,h]$ is
$h^{\ell}
\Delta_{\ell} \bigl[p \bigr]
$, the \textbf{$\ell$-Laplacian} on
$\RR^n$\index{lLaplacian@$\ell$-Laplacian!in
$\RR [x]$}\index{1cl@$\Delta_{\ell}$|see{$\ell$-Laplacian in $\RR [x]$}}
\[
\Delta_{\ell} \bigl[p \bigr] := \summ{i=1}{g}\frac{\partial^{\ell} p}{\partial
x_i^2}.
\]
In the commutative case, all homogeneous degree $d$, $\ell$-harmonic polynomials
in $\CC[x]$ are sums of  polynomials of the form
\begin{equation}
\label{eq:prevprovedcomm}
(a_1 x_1 + \ldots + a_g x_g)^d \end{equation}
where either $d < \ell$, or $d \geq \ell$ and $a_1^{\ell} + \ldots + a_g^{\ell}
= 0$.
This result follows from a paper of Reznick (see \cite{R94}), which extends
the work of others, namely Sylvester, Clifford, Rosanes, Gundelfinger, Cartan,
Maass and Helgason.  Note that the $\ell$-harmonic case is only a special case
of what Reznick proved.

The main theorem of this dissertation, Theorem \ref{thm:main}, classifies all
$\ell$-harmonic polynomials in a similar way to (\ref{eq:prevprovedcomm}) with
the
exception that instead of using pure powers of linear forms, the noncommutative
classification uses permutations of independent products of powers of linear
forms.

\begin{thm}
\label{thm:main}
Let $x = (x_1, \ldots, x_g)$ and let $\ell$ be a positive integer.
\begin{enumerate}
\item
\label{item:lowdegmain}
Every polynomial in $\CCx$ of degree less than $\ell$ is $\ell$-harmonic.
\item \label{item:middledegmain}If $g \geq d $, then every NC,
homogeneous degree $d$, $\ell$-harmonic polynomial in $\CCx$ is a sum of
polynomials of the form
\begin{equation}
\label{eq:mainthm}
\prod_{i=1}^d \sum_{j=1}^g a_{ij}x_j,
\end{equation}
such that $\sum_{j=1}^g a_{k_1 j} \ldots a_{k_{\ell} j} = 0$ for each $1 \leq
k_1 < \ldots < k_{\ell} \leq d$.

\item \label{item:highdegmain}If $d > \max\{\ell,g\}$, then
the space $\CCx$ can be viewed as a subspace of $\CC \langle x_1, \ldots, x_d
\rangle$.  Therefore,
every NC, homogeneous degree $d$, $\ell$-harmonic polynomial in $\CCx \subset
\CC\langle x_1, \ldots, x_d\rangle$ is a sum of polynomials of the form
\[\prod_{i=1}^d \sum_{j=1}^d a_{ij}x_j, \]
such that $\sum_{j=1}^d a_{k_1 j} \ldots a_{k_{\ell} j} = 0$ for each $1 \leq
k_1 \leq \ldots \leq k_{\ell} \leq d$.
\end{enumerate}
\end{thm}

\proof \
This proof requires several results and will be completed
in $\S$
\ref{sect:classification}.

The following are a few examples illustrating the kinds of polynomials
which are
of the form (\ref{eq:mainthm}).

\begin{exa}
\label{exa:prod}

The noncommutative polynomials
\[q = (x_1+ix_2)^2
\qquad
and
\qquad
 = (x_3 + ix_4)^2\]
are harmonic since
$1^2 + i^2 = 0.$
Further,
\[qr = (x_1+ix_2)^2(x_3 + ix_4)^2 = \prod_{i=1}^4 \sum_{j=1}^4 a_{ij}x_j\]
is harmonic since for each $1 \leq k_1, k_2 \leq 4$,
\[\sum_{j=1}^4 a_{k_1j}a_{k_2j} = 0,\]
where
\[
a_{11} = a_{21} = 1 \quad \mathrm{and} \quad a_{31} = a_{41} = 0,
\]\[
a_{12} = a_{22} = i \quad \mathrm{and} \quad a_{32}=a_{42} = 0,
\]\[
a_{13} = a_{23} = 0 \quad \mathrm{and} \quad a_{33} = a_{43} = 1,
\]\[
a_{14} = a_{24} = 0 \quad \mathrm{and} \quad a_{34} = a_{44} = i.
\]
\qed
\end{exa}

\begin{exa}
\label{exa:dummy1}
Consider the NC polynomial
\[r(x_1,x_2,x_3) = x_1^2x_2^2 - x_2^2x_1^2 + x_2^2x_3^3 - x_3^2x_2^2 +
x_3^2x_1^2 - x_1^2x_3^2.\]
This polynomial is harmonic. At first glance it may not be apparent why $r$ is
of the same form as (\ref{eq:mainthm}). One even notices
that when the variables commute, $r = 0$.

By considering $r$ as an element of $\CC\langle x_1,\ldots, x_4\rangle$, one
sees
\begin{align}
r(x) =&\ \frac{1}{2}(x_2^2 - x_1^2)(x_3^2 - x_4^2) - \frac{1}{2}(x_3^2 -
x_4^2)(x_2^2 - x_1^2)\\
&
+ \frac{1}{2}(x_1^2 - x_4^2)(x_2^2 - x_3^2) - \frac{1}{2}(x_2^2 - x_3^2)(x_1^2 -
x_4^2) \notag\\
&+ \frac{1}{2}(x_3^2 - x_1^2)(x_4^2 - x_2^2) - \frac{1}{2}(x_4^2 - x_2^2)(x_3^2
- x_1^2)
\notag\end{align}
Each $(x_i^2 - x_j^2)$ is equal to
\begin{align}
x_i^2 - x_j^2 &= \mathrm{Re}(x_i + ix_j)^2\\
\notag &= \frac{1}{2}(x_i + ix_j)^2
+ \frac{1}{2}(x_i - ix_j)^2.
\end{align}

This polynomial $r$
belongs to a special
class of polynomials which is defined and studied in $\S$
\ref{sect:fullydeg2}.\qed
\end{exa}

\subsubsection{Polynomials in Nonsymmetric Variables}
A polynomial $p \in \FFNx$\index{Faxs@$\FFNx$}
\index{Raxs@$\RRNx$|see{$\FFNx$}}
\index{Caxs@$\CCNx$|see{$\FFNx$}} is called a polynomial in nonsymmetric
variables.  In this case, $x_i \neq x_i^T$\index{variable!nonsymmetric} for each
$x_i$.  One can define notions of NC harmonic, $\ell$-harmonic, and subharmonic
in this setting.  In $\S$ \ref{sect:ncvars}, Theorem \ref{thm:main} is extended
to the nonsymmetric variable case.

\subsubsection{Subharmonic Polynomials}

\label{sub:othertopics}
\index{subharmonic}
It is proven in \cite{HMH09} that
all homogeneous degree $2g$ subharmonic polynomials $p \in \RRx$ are of the form
\begin{align}
\label{eq:subharmclassif}
p = \sum_i^{finite} c_i \h_i(x)^T\h_i(x) + H
\end{align}
for some $c_i \in \RR$, for some NC, harmonic, homogeneous degree $d$
polynomials $\h_i$, and for some NC, harmonic, homogeneous degree $2d$
polynomial $H \in \RRx$.  All
such polynomials $p$ of the form (\ref{eq:subharmclassif}) are not necessarily
subharmonic (for example, take all of the $c_i$ to be negative).  When the $c_i$
are all positive, then $p$ is a sum of squares of harmonic polynomials plus a
harmonic polynomial.  Proposition \ref{prop:2varsub} proves that all NC,
homogeneous, subharmonic polynomials in two variables are of this form.  Example
\ref{exa:countersubSOS} proves, however, that there exist NC subharmonic
polynomials in more than two variables which are not finite sums of squares of
harmonic polynomials plus a harmonic polynomial.

A NC polynomial $p$ is \textbf{bounded below}\index{bounded below} if there
exists a constant $C$ such that $p + C$ is matrix positive.

\begin{thm}\label{thm:bddbelow}Let $d$ be a positive integer.  Let $p \in \RRx$
or $\RRNx$ be a NC, homogeneous degree $2d$, subharmonic polynomial which is
bounded below. Then $p$ is of the form
\[p = \sum_{i}^{finite} \h_i^T\h_i, \]
where each $\h_i \in \RRx$ is a homogeneous degree $d$ harmonic polynomial.
\end{thm}

\begin{proof}
The proof is given in \S \ref{subsub:pfOfBddBelow}.
\end{proof}

\subsection{Related Topics and Motivation}
\label{sect:topics}

This paper follows a burdeogening line of work  
on noncommutative polynomials and rational functions
(see \cite{H02},  \cite{HM04}, \cite{HMV06},  \cite{BB08}, 
\cite{BK10}, \cite{KV09}).
More specifically, this paper extends the results of Helton, McAllister, and
Hernandez
in \cite{HMH09}.  Related to this topic is the work of Popescu on free
holomorphic
and free pluriharmonic functions (see \cite{P05}, \cite{P06}, \cite{P08a},
\cite{P08b},
\cite{P09}).  Popescu's work uses a different definition of derivative, hence
his notion of a noncommutative harmonic
function is different from the one presented here.

This article would come
under the general heading of ``free analysis",
since the setting is a noncommutative algebra whose generators
are ``free" of relations. 
The largest and oldest (initiated by Voiculescu)
 branch of the subject
is free probability. The interested reader is referred to the web
site \cite{SV06}
of American Institute of Mathematics, which in particular gives the
findings of the AIM workshop in 2006 on free analysis.

\bigskip

\noindent
{\bf Noncommutative Convexity} \
The noncommutative Hessian is defined as:
\begin{equation*}
\nchess[p(x_1,\ldots,x_g), \{ x_1,\eta_1\},\ldots, \{ x_g,\eta_g\}] :=
\frac{d^2}{dt^2}[p(x_1+t\eta_1,\ldots,x_g+t_g\eta_g)]_{|_{t=0}}. \\
\end{equation*}
Note that this is composed of several independent direction parameters
$\eta_i$. If $p$ is a polynomial, then its Hessian is a polynomial
in $x$ and $\eta$ which is homogeneous of degree 2 in $\eta$.

A NC polynomial is considered {\bf convex} wherever  its
Hessian is matrix positive.  A NC polynomial $p = p(x_1,\ldots,x_d)$ is {\bf
geometrically convex}
 if and only if, for every $X, Y \in \gtupn  $, the polynomial
\begin{equation}
\frac{1}{2}\bigl(p(X) + p(Y)\bigr) - p\biggl(\frac{X + Y}{2}\biggr)
\eeq\end{equation}
is positive semidefinite.
It is proved in \cite{HM98}
that convexity is equivalent to geometric convexity.
A crucial fact regarding these polynomials
is that they are
all of degree two or less (see \cite{HM04}).
Some excellent papers on noncommutative convexity are
\cite{HT06} and \cite{Han97}.

The commutative analog of this ``directional''
Hessian is the quadratic function
\begin{equation}
H\bigl(p\bigr)
\begin{pmatrix} \eta_1 \\ \vdots \\ \eta_g
\end{pmatrix} \cdot
\begin{pmatrix} \eta_1 \\ \vdots \\ \eta_g
\end{pmatrix}
\end{equation}
where $H\bigl(p \bigr)$ is the Hessian matrix:
\begin{equation}
\begin{pmatrix}
\frac{\partial^2p}{\partial x_1x_1} &\cdots &\frac{\partial^2p}{\partial x_1x_g}
\\
\vdots & \ddots & \vdots \\
\frac{\partial^2p}{\partial x_gx_1} &\cdots &\frac{\partial^2p}{\partial x_gx_g}
\end{pmatrix}.
\end{equation}
If this Hessian is positive
semidefinite  for all points
 $(x_1,\ldots,x_g)$,
then $f$ is said to be convex.

Classically the Laplacian
%
 is the trace of the Hessian.  When the Laplacian is nonnegative, $f$ is subharmonic, so all convex polynomials must also be subharmonic.
Similarly, all NC convex polynomials 
must also NC subharmonic.  

\bigskip

\noindent
{\bf Noncommutative Algebra in Engineering} \
Inequalities involving polynomials in matrices and their
inverses and other associated optimization problems have become
important in engineering, for an exposition see \cite{dOHMP}.
When such polynomials are matrix convex, any
local minima are automatically global minima.
Inequalities involving such matrix polynomials
may also be analyzed well using interior point numerical methods.
In the last few
years, proposed approaches in the field of
optimization and control theory based on linear matrix
inequalities and semidefinite programming have become
important and promising since they present
a general framework which can be
can be used for
a large set of problems.
When they are convex, matrix inequality
optimization problems are well behaved and interior point
methods provide efficient algorithms which are effective on
moderate-sized problems.
Unfortunately, the class of matrix-convex
NC polynomials  is very small. As already mentioned,
matrix convex polynomials are all of degree two or less \cite{HM04}.

The original reason for studying noncommutative polynomial
solutions to partial differential inequalities
 was to analyze conditions similar to
convexity which are not as restrictive, in the hopes of finding much
broader classes of polynomials which still retained "nice" properties (i.e. subharmonic polynomials).

\section{Basic Facts about Derivatives and Permutations of Noncommutative Polynomials}
\label{sect:basicfacts}
This section proves some useful results on differentiation of NC polynomials.  These results will be used later to prove Theorem \ref{thm:main}.

\subsection{Permutations of Polynomials}

\subsubsection{Definition}
\label{sub:actionsofSd}

For a positive integer $d$, let \index{symmetric!group of degree $d$}\index{Sd@$S_d$|see{symmetric group of degree $d$}}$S_d$ denote the \textbf{symmetric group
of degree $d$} (that is, the group of permutations on $d$ elements).  An element $\sigma \in S_d$
is a bijection of $\{1, \ldots, d\}$ onto itself, mapping $1, 2, \ldots, d$ onto $\sigma(1), \sigma(2), \ldots, \sigma(d)$ respectively.
This
may be presented as
\begin{align}
\sigma = \left(\begin{array}{cccc}
1&2&\ldots&d\\
\sigma(1)&\sigma(2)&\ldots&\sigma(d)
\end{array}\right).
\end{align}
Alternately, a permutation $\sigma \in S_d$ may
be represented as a product of
cycles
\[\tau = (a_1\ a_2\ \ldots a_k), \]
where the cycle $\tau$ is defined by
\begin{align}
\label{eq:cycle}
\tau(a_1) &= a_2\\
\notag \tau(a_2) &=a_3\\
\notag \vdots\\
\notag \tau(a_{d-1}) & = a_d\\
\notag \tau(a_d) & = a_1
\end{align}
and $\tau(k) = k$ for all other $k$.

Given a monomial $m=x_{a_{1}}
 x_{a_{2}}\ldots x_{a_{d}}\in \FFx$ of degree $d$
and a permutation $\sigma \in S_d$ define $\sigma[m]$ to be
\begin{equation}
\label{def:action}
 \sigma \left[x_{\alpha_{1}}
 x_{\alpha_{2}}\ldots x_{\alpha_{d}}\right] :=
x_{\alpha_{\sigma(1)}}x_{\alpha_{\sigma(2)}}\ldots x_{\alpha_{\sigma(d)}}.
\end{equation}
For a general homogeneous degree $d$ polynomial $p$ of the form
\[ p  = \sum_{|m|=d}A_{m}\,m\]
define $\sigma[p]$ to be
\[\sigma[p] := \sum_{|m|=d}A_{m}\,\sigma[m].\]
 The polynomial $\sigma[p]$ is referred to as a
 \textbf{permutation of $p$} \index{permutation of a polynomial}
 \index{$\sigma[p]$|see{permutation of a polynomial}}.

For clarification, note that later on---for example, in $\S$
\ref{sect:fullydeg2}---we use a permuation $\tau \in S_d$ to look at monomials
\[ x_{\tau(1)}^{\ell} \ldots x_{\tau(d)}^{\ell}.\]
The permutation action defined in (\ref{def:action}) applies a permutation to
the subscripts of the $a_1, \ldots, a_d$, not to the subscripts of the
variables $x_1, \ldots, x_g$.

\begin{exa}Let $\sigma \in S_4$ be
\[\sigma = \left(\begin{array}{cccc}
1&2&3&4\\
4&2&3&1
\end{array}\right) = (1\ 4),\]
and let $p \in \FFx$ be
\[p = x_1x_2x_3x_4 + x_1^4.\]
One sees that $\sigma[p]$
is equal to \
$
\sigma[p] = x_{4}x_2x_3x_1 + x_1^4.
$
\qed
\end{exa}

For any $d$ and any $\sigma \in S_d$, define $\sigma[0]$ to be $0$.   With this convention, permutation as defined in (\ref{def:action}) defines an action of $S_d$ on the subspace of $\FFx$ spanned by all NC, homogeneous
degree $d$ polynomials.

\begin{prop}
 \label{prop:permOfProdLin}
Let $p \in \FFx$ be a homogeneous degree $d$ polynomial
of the form
\[ p = \prod_{i=1}^d \sum_{j=1}^g a_{ij} x_j. \]
Let $\sigma \in S_d$.
Then $\sigma[p]$ is equal to
\[ \sigma[p] = \prod_{i=1}^d \sum_{j=1}^g a_{\sigma(i)j} x_j.\]
\end{prop}

\begin{proof}
 We have
\begin{align}
\notag
  p &= \sum_{j_1=1}^g \ldots \sum_{j_d = 1}^g \left[
(a_{1j_1}x_{j_1})(a_{2j_2}x_{j_2}) \ldots (a_{dj_d}x_{j_g})
\right]\\
\notag
&= \sum_{j_1=1}^g \ldots \sum_{j_d = 1}^g \left(a_{1j_1} \ldots a_{dj_d}\right)
x_{j_1} \ldots x_{j_d}.
\end{align}
By definition, $\sigma[p]$ is equal to
\begin{align}
 \notag \sigma[p]
& = \sum_{j_1=1}^g \ldots \sum_{j_d = 1}^g \left(a_{1j_1} \ldots a_{dj_d}\right)
x_{j_{\sigma(1)}} \ldots x_{j_{\sigma(d)}}.
\end{align}
The product of scalars
$ a_{1j_1} \ldots a_{dj_d}$
is equal to
\[ a_{1j_1} \ldots a_{dj_d} = a_{\sigma(1)j_{\sigma(1)}} \ldots
a_{\sigma(d)j_{\sigma(d)}}.\]
Making the substitution $j_{\sigma(i)} = k_i$ gives
\begin{align}
 \notag \sigma[p]
& = \sum_{j_1=1}^g \ldots \sum_{j_d = 1}^g \left(a_{\sigma(1)j_{\sigma(1)}}
\ldots a_{\sigma(d)j_{\sigma(d)}}\right)
x_{j_{\sigma(1)}} \ldots x_{j_{\sigma(d)}}\\
\notag
&= \sum_{k_1=1}^g \ldots \sum_{k_d = 1}^g
\left(a_{\sigma(1)k_1}
\ldots a_{\sigma(d)k_d}\right)
x_{k_1} \ldots x_{k_d}
\\
\notag
&= \prod_{i=1}^d \sum_{j=1}^g a_{\sigma(i)j} x_j.
\end{align}
\end{proof}

\subsubsection{Symmetrization}
  For each homogeneous degree $d$
  polynomial $p \in \FFx$, define the \textbf{symmetrization}\index{symmetrization}
  \index{Symm@$\Symm$|see{symmetrization}} of $p$ by
  \begin{equation}
  \Symm[p] := \frac{1}{d!} \sum_{\sigma \in S_d} \sigma [p].
  \end{equation}
  Define $\Symm[0]$ to be $0$.
  For a general polynomial $p$ equal to
  \[p = \sum_{i=0}^{d} p_i(x) \in \FFx,\]
  where each $p_i$ is either zero or homogeneous of degree $i$,
  define $\Symm[p]$ to be
  \[\Symm[p] := \sum_{i=0}^{d} \Symm[p_i].\]
  A polynomial $p \in \FFx$ is said to be \textbf{symmetrized}\index{symmetrization!symmetrized polynomial} if $p = \Symm[p].$

  \begin{exa}If $p(x_1, x_2) = x_1x_2 + x_2x_1$ and $q(x_1,x_2) = x_1x_1x_2x_2$ then
  \[\Symm[p(x_1,x_2)] = p(x_1, x_2)\]
  so $p(x_1,x_2)$ is symmetrized. On the other hand,
  \begin{align}\Symm[q(x_1, x_2)] =& \frac{1}{6}(x_1x_1x_2x_2 + x_1x_2x_1x_2 + x_2x_1x_1x_2\\
  \notag &+ x_1x_2x_2x_1 + x_2x_1x_2x_1 + x_2x_2x_1x_1)\end{align}
  which shows $q(x_1, x_2)$ is not symmetrized.
  \qed\end{exa}

  Note further that for any homogeneous degree $d$ polynomial $r \in \FFx$,
  \begin{align}
  \Symm\left[\Symm[r]\right] &= \frac{1}{(d!)^2}\sum_{\sigma \in S_d} \sigma \left[\sum_{\tau \in S_d} \tau[r]\right]\notag\\
  & =
  \frac{1}{(d!)^2} \sum_{\sigma \in S_d}\sum_{\tau \in S_d} \sigma\tau[r]. \notag \end{align}
  For each $\sigma$, let $\tau = \sigma^{-1}\omega$. Then
  \begin{align}
   \Symm\left[\Symm[r]\right]&= \frac{1}{(d!)^2} \sum_{\sigma \in S_d}\sum_{\omega \in S_d} \sigma\sigma^{-1}\omega[r] \\
    &= \frac{d!}{(d!)^2}\sum_{\omega \in S_d} \omega[r] = \Symm[r]. \notag
    \end{align} In other words,
  the symmetrization of a polynomial is symmetrized.

Here are some propositions which prove other properties of $\Symm[]$.

\begin{prop}
 \label{prop:symmOfSymm}
A homogeneous degree $d$ NC polynomial $p \in \FFx$ is symmetrized if and only
if $\sigma[p] = p$ for
each $\sigma \in S_d$.
\end{prop}

\begin{proof}
 If $\sigma[p] = p$ for each $\sigma \in S_d$, then
\[ \Symm[p] = \frac{1}{d!} \sum_{\sigma \in S_d} \sigma[p] = \frac{1}{d!}
\sum_{\sigma \in S_d} p = p.\]
If $p$ is symmetrized, then $p = \Symm[p]$ and for each $\sigma \in S_d$,
\begin{align}
 \notag
\sigma[p] &= \sigma[\Symm[p]] = \sigma \left[ \frac{1}{d!} \sum_{\tau \in S_d}
\tau[p] \right]\\
\notag
&= \frac{1}{d!} \sum_{\tau \in S_d} \sigma\tau[p]
\ = \ \frac{1}{d!} \sum_{\omega \in S_d} \omega[p]
\\
\notag
&= \Symm[p] = p.
\end{align}

\end{proof}

\begin{prop}
\label{prop:symmOfPower}
Let $p = (b_1x_1 + \ldots + b_dx_d)^d \in \FFx$ be a homogeneous degree $d$
polynomial.
Then $\Symm[p] = p$.
\end{prop}

\begin{proof}
 We can express $p$ as
\[ p = \prod_{i=1}^d \sum_{j=1}^g a_{ij} x_j,\]
where $a_{ij} = b_j$ for all $i$.
By Proposition \ref{prop:permOfProdLin}, for each $\sigma \in S_d$, we have
\[ \sigma[p] = \prod_{i=1}^d \sum_{j=1}^g a_{\sigma(i)j} x_j = \prod_{i=1}^d
\sum_{j=1}^g a_{ij} x_j,\]
since $a_{\sigma(i)j} = a_{ij} = b_j$.
\end{proof}

  Define the \textbf{commutative collapse}\index{commutative collapse}
  \index{Comm@$\Comm$|see{commutative collapse}} of $p \in \FFx$, denoted $\Comm[p]$, as the
  projection of $p$ onto $\FF[x]$ given by allowing the variables
  $x_i$ to commute.

  \begin{exa}If $p(x_1, x_2) = x_1x_2 + x_2x_1 \in \FFx$, then
  $\Comm[p] = 2x_1x_2 \in \FF[x].$\qed
  \end{exa}

\subsubsection{Independent Products}

   Let $p_1, p_2, \ldots, p_n \in \FFx$. The product $p_1 p_2 \ldots p_n$ is an
   \textbf{independent product of $p_1, \ldots, p_n$}
  \index{independent product}
  if $p_1, p_2, \ldots, p_n$ all depend on different variables.  Formally, $p_1, p_2, \ldots, p_n$ form an independent product if whenever $\deg_i(p_j) > 0$ for some $x_i$ and $p_j$, then $\deg_i(p_k) = 0$ for all $k \neq j$.

  As a convention, if a polynomial $p$ written as
  \[p = p_1p_2 \ldots p_n\]
  is called an independent product, what is meant is that it is an independent product of $p_1, p_2, \ldots, p_n$. Further, $p = p_1$ can be considered as an independent product of $p_1$ itself.

   \begin{exa}
  Let $p_1 = x_1^2 - x_2^2$ and $p_2 = x_3$.
  Then $p_1p_2 = (x_1^2 - x_2^2)x_3$ is an independent product since $p_1$ only depends on the variables $x_1$ and $x_2$, and $p_2$ only depends on the variable $x_3$.\qed
  \end{exa}

\subsection{Product Rules for NC Derivatives}
This subsection proves some product rules for directional derivatives
of NC polynomials.
\subsubsection{Product Rule for First Derivatives}
The following appears as Lemma 2.1 in \cite{HMH09}, but
the proof is included here for the convenience of the reader.

\begin{lemma}
\label{lem:prodrule}
The product rule for the  directional derivative of
NC polynomials
is
\begin{equation}
\label{eq:DProd}
\dird[p_1\ncm {p_2},x_i,h]
= \dird[{p_1},x_i,h]\ncm {p_2}\, +\, {p_1}\ncm \dird[{p_2},x_i,h].
\eeq\end{equation}
\end{lemma}

\proof \ \
The directional derivative $\dird[m,x_i,h]$
of a product $m=m_1m_2$ of noncommutative monomials
$m_1$ and $m_2$ is the sum of terms produced by
replacing one instance of $x_i$ in $m$ by $h$.
This sum can be divided into two parts:

$\mu_1$, the sum of terms whose $h$ lie in the first $|m_1|$ letters,
\ {i.e.} \ \ $ \dird[m_1,x_i,h]m_2$

$\mu_2$, the sum of terms whose $h$ lie in the last $|m_2 |$ letters,
\ {i.e.} \
$\ m_1 \dird[m_2,x_i,h].
$

\noindent
Therefore
\begin{equation}
\dird[m_1m_2,x_i,h] = \dird[m_1,x_i,h]m_2 + m_1\dird[m_2,x_i,h].
\eeq\end{equation}

This product rule extends to the product
of any two NC polynomials $p_1$ and $p_2$ as follows.
\begin{align}
\dird[p_1p_2, x_i,h] &=\dird\left[\left(\sum_{m_1 \in \mathcal{M}} A_{m_1}m_1\right)\left(\sum_{m_2 \in \mathcal{M}} A_{m_2}m_2 \right), x_i, h \right] \notag\\
&= \sum_{m_1 \in \mathcal{M}}\sum_{m_2 \in \mathcal{M}} A_{m_1}A_{m_2} \dird[m_1m_2, x_i, h]\notag\\
&= \sum_{m_1 \in \mathcal{M}}\sum_{m_2 \in \mathcal{M}} A_{m_1}A_{m_2} \dird[m_1, x_i, h]m_2 \notag\\
&+ \sum_{m_1 \in \mathcal{M}}\sum_{m_2 \in \mathcal{M}} A_{m_1}A_{m_2} m_1\dird[m_2, x_i, h]
\notag\\
&=\left(\sum_{m_1 \in \mathcal{M}} \dird[A_{m_1}m_1, x_i, h]\right)\left(\sum_{m_2 \in \mathcal{M}}A_{m_2}m_2\right)\\
 \notag &+ \left(\sum_{m_1 \in \mathcal{M}}A_{m_1}m_1\right)\left(\sum_{m_2 \in \mathcal{M}}\dird[A_{m_2}m_2, x_i, h]\right)
\notag\\
&=\dird[p_1,x_i,h]p_2 + p_1\dird[p_2,x_i,h].
\notag
\end{align}\qed

\subsubsection{Chain Rule for Noncommutative Polynomials}

\begin{definition}
Let $A \in \mathbb{F}^{g\times g}$ be a $g \times g$ matrix.  Define $Ax$ to be
\[Ax =  \left(\sum_{j=1}^g a_{1j}x_j,
\sum_{j=2}^g a_{2j}x_j,
\ldots, \sum_{j=1}^g a_{gj}x_j\right).\]
\end{definition}

The following proposition proves the chain rule for noncommutative
polynomials in the special case that $y = Ax,$ where $A$ is a
linear transformation.

\begin{prop}
\label{thm:chainrule}
Let $p \in \FFx$ and $A \in \FF^{g \times g}$.
Then,
\[\dird\left[p(Ax), x_j, h\right](x) =  \dird\left[p, \sum_{i=1}^g a_{ij}x_i, h\right](Ax),\]
where $\dird[p, x_i, h](Ax)$ denotes plugging $Ax$ into the NC polynomial
$\dird[p, x_i, h](x)$.
\end{prop}

\begin{proof}Consider the case where
 $p = \displaystyle\sum_{i=1}^g b_i x_i + c$.

In this case, $\dird[p, x_i, h] = b_i$ for each $i$. Further, $p(Ax)$ is equal to
\[p(Ax) = \sum_{i=1}^g b_i\left(\sum_{k=1}^g a_{ik}x_k\right) + c = \sum_{k=1}^g \left(\sum_{i=1}^g b_ia_{ik}\right)x_k + c.\]
Differentiating gives
\[\dird\left[p(Ax), x_j, h\right](x) = \sum_{i=1}^g b_ia_{ij} =  \dird\left[p, \sum_{i=1}^g a_{ij} x_i, h\right](Ax).\]

Next, consider the case where $p, q \in \FFx$ are any two polynomials for which the proposition holds. Then,
\begin{align}
\dird\left[p(Ax)q(Ax), x_j, h\right](x) &= \dird\left[p(Ax), x_j, h\right](x)q(Ax) + p(Ax)\dird\left[q(Ax), x_j, h\right](x)
\notag\\
&= \dird\left[p, \sum_{i=1}^g a_{ij}x_i, h\right](Ax)q(Ax) \notag\\
&+ p(Ax)\dird\left[q(x), \sum_{i=1}^g a_{ij}x_i, h\right](Ax)\notag \\
&= \sum_{i=1}^g a_{ij}\dird\left[p, x_i, h\right](Ax)q(Ax) \notag\\
&+ \sum_{i=1}^g a_{ij}p(Ax)\dird[q(x), x_i, h](Ax)\notag \\
&= \dird\left[pq, \sum_{i=1}^g a_{ij} x_i, h\right](Ax).\notag
\end{align}

The proposition holds for all polynomials degree $d \leq 1$,
and
any polynomial of degree $d > 1$
is simply a sum of products of polynomials
in smaller degree. The result therefore follows by induction and by linearity of the NC directional derivative.
\end{proof}

\begin{corollary}
\label{cor:chainrule}
Let $A \in \mathbb{F}^{g\times g}$ be a $g \times g$ matrix,
let $p \in \mathbb{F}\langle x \rangle$ be a NC polynomial, and let $q \in \mathbb{F}[x]$ be a commutative polynomial.
Then
\[\dird[p(Ax), q(x), h](x) = \dird\left[p, q\left(A^T x\right), h\right](Ax).\]
\end{corollary}

\begin{proof}For $q$ of degree $\ell = 1$ this follows
from Proposition \ref{thm:chainrule}.
Inductively assume the corollary is true for $\deg(q) < \ell$.
By linearity, it suffices to consider monomials $q = x_{b_1} \ldots x_{b_{\ell}}$.
Applying Proposition \ref{thm:chainrule} to $p(Ax)$ gives
\begin{align}
\dird[p(Ax), q, h] &= \dird\left[\dird[p(Ax), x_{b_1}\ldots x_{b_{\ell-1}}, h], x_{b_{\ell}}, h\right]
\\&= \dird\left[\dird\left[p, \left(x_{b_1}\ldots x_{b_{\ell-1}}\right)\left(A^Tx\right), h\right](Ax), x_{b_{\ell}}, h\right](x)
\notag \\&= \dird\left[\dird\left[p, \left(x_{b_1}\ldots x_{b_{{\ell}-1}}\right)\left(A^Tx\right), h\right](x), x_{b_{\ell}}\left(A^Tx\right), h\right](Ax)
\notag\\
\notag
&=\dird\left[p, \left(x_{b_1}\ldots x_{b_{\ell}}\right)\left(A^{T}x\right), h\right](Ax)\\&= \dird\left[p, q\left(A^Tx\right), h\right](Ax).\notag
\end{align}
\end{proof}

\subsubsection{Product Rule for Independent Products}

\begin{prop}
\label{prop:ProdSuff}
Let $x = (x_1, \ldots, x_g)$, let $\ell$ be a positive
integer, and let $p_1, \ldots, p_k \in \FFx$
be NC polynomials such that $p_1 \ldots p_k$
is an independent product.

\begin{enumerate}
\item
\label{item:ellprodrule}
Fix $x_i$.
Taking the derivative with respect to the full symbol $x_i^{\ell}$ may be done by the following product rule:
\begin{align}
\label{eq:productell}
\dird\left[p_1p_2\ldots p_k, x_i^{\ell}, h\right]
&= \dird\left[p_1, x_i^{\ell},h\right]p_2\ldots p_k \\
\notag &+p_1\dird\left[p_2,x_i^{\ell},h\right]\ldots p_k\\
&\vdots
\notag\\
\notag
&+ p_1p_2\ldots \dird\left[p_k, x_i^{\ell},h\right].
\end{align}

\item
\label{item:elllapprodrule}Taking the $\ell$-Laplacian may be done by the following product rule:
\begin{align}
\label{eq:productelllap}
\lap_{\ell}\left[p_1p_2\ldots p_k, h\right]
&= \lap_{\ell}[p_1, h]p_2\ldots p_k \\
\notag &+p_1\lap_{\ell}[p_2,h]\ldots p_k\\
&\vdots
\notag\\
\notag
&+ p_1p_2\ldots \lap_{\ell}[p_k, h].
\end{align}

\item
\label{item:ellcor}
If each $p_i$ is $\ell$-harmonic, then so is $p_1p_2 \ldots p_k$.
\end{enumerate}
\end{prop}

\begin{proof}Items (\ref{item:elllapprodrule}) and (\ref{item:ellcor}) are straightforward corollaries to (\ref{item:ellprodrule}).

To prove (\ref{item:ellprodrule}), assume without loss of generality that $k=2$,
that $p_2$ doesn't depend on the variables $x_1, \ldots, x_r$, and that $p_1$ doesn't depend on the variables $x_{r+1}, \ldots, x_g$.
Proceed by induction on $\ell$. The case where $\ell = 1$
follows by Lemma \ref{lem:prodrule}.
Next assume the proposition for order less than $\ell$.
Then,
\[\dird\left[p_1p_2, x_i^{\ell}, h\right] = \dird\left[\dird\left[p_1, x_i^{\ell-1}, h\right]p_2 + p_1\dird\left[p_2, x_i^{\ell-1}, h\right], x_i, h\right].\]
Since $p_2$ doesn't depend on the variables $x_1, \ldots, x_r$,
 it follows that for all $i \leq r$ and for all $L$,
\[\dird\left[\dird\left[p_2,x_i^L,h\right], x_i, h\right] = 0.\]
Similarly, for $i > k$ and for all $L$,
\[\dird\left[\dird\left[p_1, x_i^L, h\right], x_i, h\right] = 0.\]
If $i \leq r$, then
\begin{align}
\dird\left[p_1p_2,  x_i^{\ell}, h\right] &= \dird\left[\dird\left[p_1, x_i^{\ell-1}, h\right]p_2 + p_1\dird\left[p_2, x_i^{\ell-1}, h\right], x_i, h\right]\\
&= \dird\left[\dird\left[p_1, x_i^{\ell-1}, h\right]p_2, x_i, h\right] \notag\\
& = \dird\left[\dird\left[p_1, x_i^{\ell - 1}, h\right], x_i, h\right]p_2 + \dird\left[p_1, x_i^{\ell-1}\right]\dird\left[p_2, x_i, h\right]\notag\\
&=\dird\left[p_1, x_i^{\ell},h\right]p_2
\\ \notag
&= \dird\left[p_1,  x_i^{\ell}, h\right]p_2
+ p_1\dird\left[p_2,  x_i^{\ell}, h\right]\notag.
\end{align}
A similar expression holds for $i > r$.
\end{proof}

\subsection{Derivatives of Permutations}
\label{subsect:derivation}

The NC directional derivative is well behaved under the permutation defined in (\ref{def:action}).

\subsubsection{Properties of $\Comm$ and $\Symm$}
\begin{prop}
\label{thm:commiffsymm}
If $p \in \FFx$, then $\Comm[p] = 0$ if and only if $\Symm[p] = 0$.
\end{prop}

\begin{proof}
Each NC monomial $m = x_{\alpha_{1}}x_{\alpha_{2}}\ldots x_{\alpha_{d}}$ has the same commutative
collapse as another NC monomial $x_{\beta_{1}}x_{\beta_{2}}\ldots x_{\beta_{d}}$ if and only if
$x_{\beta_{1}}x_{\beta_{2}}\ldots x_{\beta_{d}} = \sigma [x_{\alpha_{1}}x_{\alpha_{2}}\ldots x_{\alpha_{d}}]$
for some $\sigma \in S_d$.
This is because the set of commutative
polynomials
can be thought of as the set of noncommutative polynomials modulo permutation.

First consider a polynomial $p$ of the form
\begin{align}
\label{eq:combmonomials}
p = \sum_{\sigma \in S_d} A_{\sigma}\sigma [m]
\end{align}
where $m$ is some monomial of degree $d$.
Suppose $p$ satisfies
\[\Comm[p] = \left(\sum_{\sigma \in S_d} A_{\sigma}\right)\Comm[m] = 0.\]
Since $\Comm[m] \neq 0$, this is true if and only if $\displaystyle\sum_{\sigma \in S_d} A_{\sigma} = 0.$
Also,
\begin{align}
\label{eq:sigmaNtau}
d! \Symm[p] &= \sum_{\tau \in S_d} \tau \left[\sum_{\sigma \in S_d} A_{\sigma} \sigma [m]\right]\\
&= \sum_{\sigma \in S_d}\sum_{\tau \in S_d} A_{\sigma} \tau \sigma[m]. \notag
\end{align}
For each $\sigma$, set $\tau = \omega \sigma^{-1}$ in (\ref{eq:sigmaNtau}) so that $A_{\sigma}\tau\sigma[m] = A_{\sigma}\omega[m]$.
Then,
\begin{align}
d! \Symm[p]&= \sum_{\sigma \in S_d}\sum_{\omega \in S_d} A_{\sigma} \omega[m]
\\&= \left(\sum_{\sigma \in S_d} A_{\sigma}\right) \left(\sum_{\omega \in S_d}  \omega [m]\right).
\label{eq:thisiszeroiff}
\end{align}
The
expression (\ref{eq:thisiszeroiff}) is zero if
and only if $\sum_{\sigma \in S_d} A_{\sigma} = 0$, that is, if and only if $\Comm[p] = 0$.

In general, all polynomials are
are simply sums of polynomials of the form of (\ref{eq:combmonomials}), so the
proposition follows by linearity of the derivative.
\end{proof}

\begin{prop}
\label{prop:commcommsymm}
If $p \in \FFx$, then $\Comm[p] = \Comm[\Symm[p]]$
\end{prop}

\begin{proof}
First,
\[\Symm[p - \Symm[p]] = \Symm[p] - \Symm[\Symm[p]] = \Symm[p] - \Symm[p] = 0.\]
Therefore Proposition \ref{thm:commiffsymm} implies that
\[\Comm[p - \Symm[p]] = 0.\]
Adding $\Comm[\Symm[p]]$ to both sides shows that
\[\Comm[p] = \Comm[\Symm[p]].\]
\end{proof}

With these results in mind, for any commutative polynomial $p \in \FF[x]$, define
the \textbf{symmetrization of a polynomial in $\FF[x]$}, or
$\Symm[p] \in \FFx$,
 \index{symmetrization!of a polynomial in $\FF[x]$} as the unique symmetrized NC polynomial for which $\Comm[\Symm[p]] = p.$

\begin{exa}If $p \in \FF[x]$ is \
$p(x_1,x_2) = x_1^4 + 6x_1^2x_2^2$,\
then
\begin{align}
\label{eq:symmexa}\Symm[p] &= x_1x_1x_1x_1 + x_1x_1x_2x_2 + x_1x_2x_1x_2 + x_2x_1x_1x_2\\
&+ x_1x_2x_2x_1 + x_2x_1x_2x_1 + x_2x_2x_1x_1 \in \FFx.\notag
\end{align}
Note also that for the NC polynomial
$q = x_1^4 + 6x_1^2x_2^2 \in \FFx,$
one has $\Symm[p] = \Symm[q]$. In general,
the notation $\Symm[p]$ does not depend on whether
$p$, as written, means a polynomial in $\FF[x]$ or a polynomial in $\FFx$. \qed
\end{exa}

\subsubsection{Derivatives of Permutations}
The following proposition shows that the NC directional derivative operator and the permutation operation defined in (\ref{def:action}) commute.

\begin{prop}
\label{thm:basicactionprop}
Let $p \in \FFx$ be a NC homogeneous degree $d$ polynomial, let $q \in \FF[x]$ be a commutative polynomial, and let $\sigma \in S_d$.
Then,
\[\dird[\sigma[p], q,h] = \sigma[\dird[p,q,h]]. \]
\end{prop}

\begin{proof}It suffices to prove this proposition in the case that $p$ is a monomial
since the directional derivative and the permutation are both linear in $p$.  Also, it suffices to prove this proposition in the case that $q = x_i$. If it is known that for each $p \in \FFx$ and each $x_i$ that
\[ \dird[\sigma[p],x_i,h] = \sigma[\dird[p,x_i,h]],\]
then by repeatedly applying the derivative and by linearity, the result follows for arbitrary $q \in \FF[x]$.

For each monomial $p = x_{\alpha_1} \ldots x_{\alpha_d}$, define
$\Subs[p,h,j]$ to be the monomial produced by substituting $h$ for the $j^{th}$ entry of $p$.
Define $\delta_{ij}[p]$ to be
\begin{align}
\delta_{ij}[x_{\alpha_1}\ldots x_{\alpha_d}] = \left\{ \begin{array}{cc} 1& {\alpha_j} = i\\
0& \mathrm{otherwise}
\end{array}
\right.
\end{align}
The directional derivative with respect to $x_i$ is the sum of terms produced by
replacing one instance of $x_i$ in $p$ by $h$.  Therefore,
\[\dird[p, x_i, h] =
\sum_{j = 1}^d \delta_{ij}[p] \Subs[p, h, j].\]
Now $\dird[\sigma[p], x_i, h]$ is the sum of terms where each instance of $x_i$ in $\sigma[p]$ is replaced by $h$.
Therefore,
\[\dird[\sigma[p], x_i, h] =
\sum_{j = 1}^d \delta_{ij}[\sigma[p]] \Subs[\sigma[p], h, j].\]

The $j^{th}$ entry of $p$ is $x_i$ if and only if the $\sigma(j)^{th}$ entry of $\sigma[p]$ is $x_i$.  Turning this around, the $j^{th}$ entry of $\sigma[p]$ is $x_i$ if and only if the $\sigma^{-1}(j)^{th}$ entry of $p$ is $x_i$.  Therefore
\[\delta_{ij}[\sigma[p]] = \delta_{i\sigma^{-1}(j)}[p].\]

Further, the polynomial $\Subs[\sigma[p],h,j]$ is the monomial.
\[\Subs[\sigma[p],h,j] = x_{\sigma(1)}x_{\sigma(2)}\ldots x_{\sigma(j-1)}hx_{\sigma(j+1)}\ldots x_{\sigma(d)}. \]
This is equal to $\sigma[\Subs[p,h,\sigma^{-1}(j)]]$.

Putting all of this together gives,
\begin{align}
\dird[\sigma[p],x_i,h] &= \sum_{j=1}^d \delta_{i\sigma^{-1}(j)}[p]\sigma[\Subs[p,h,\sigma^{-1}(j)]]
\notag\\
\notag &= \sigma \left[\sum_{k=1}^d \delta_{ik}[p]\Subs[p,h,k]\right]
\\ \notag
&=\sigma[\dird[p,x_i,h]].
\end{align}
\end{proof}

\begin{prop}
\label{prop:DsymEqualsSymD} Let $p \in \FFx$ be a NC homogeneous
degree $d$ polynomial, and $q \in \FF[x]$
be a commutative homogeneous degree $\ell$ polynomial. Then
\begin{equation}\dird[\Symm[p],q, h] = \Symm[\dird[p, q, h]]
= \Symm[h^{\ell}q(\nabla)\Comm[p]].
\end{equation}
\end{prop}

\begin{proof}
By Propostion \ref{thm:basicactionprop}, \begin{align}
\notag \dird[\Symm[p],q,h] &= \dird\left[\sum_{\sigma \in S_d}\sigma[p],q,h\right] = \sum_{\sigma \in S_d}\sigma[\dird[p,q,h]] = \Symm[\dird[p,q,h]].
\end{align}
Further, by (\ref{eq:RefOnCommCollapseDeriv}),
\[\Comm[\dird[p,q,h]] = h^{\ell}q(\nabla)\Comm[p]. \]
Therefore by Proposition \ref{prop:commcommsymm}
\[\Symm[\dird[p,q,h]] = \Symm[\Comm[\dird[p,q,h]] = \Symm[h^{\ell}q(\nabla)\Comm[p]]. \]
\end{proof}

\begin{prop}Let $p$ and $q$ be as in
Proposition \ref{prop:DsymEqualsSymD},
and let $\sigma \in S_d$ be
a permutation.
Then
$\dird[p, q, h] = 0$ if and only if $\dird[\sigma[p], q, h] = 0$.
\end{prop}
\begin{proof}If $\dird[p,q,h] = 0$, then
\[\dird[p, q, h] = 0 = \sigma[0] =  \sigma[\dird[p, q, h]] = \dird[\sigma [p], q, h].\]
The other direction follows by applying $\sigma^{-1}$ in the same way.
\end{proof}

\subsubsection{Correspondence Between Noncommutative Symmetrized Polynomial Solutions and Commutative Polynomial Solutions of Partial Differential Equations}

One special class of solutions to a NC partial differential equation of homogeneous order
are those solutions which are symmetrized polynomials. The following theorem
shows that in this case there exists a natural one-to-one correspondence
between NC symmetrized polynomial solutions of a partial differential
equation and commutative polynomial solutions to the same
partial differential equation.

\begin{thm}
\label{thm:parta}
Let $p \in \FFx$ be a NC, homogeneous degree $d$ polynomial
and $q \in \FF[x]$ be a commutative homogeneous
degree $\ell$ polynomial.
Then the NC polynomial $\Symm[p]$ satisfies
\[\dird[\Symm[p], q, h] = 0, \]
if and only if
commutative polynomial $\Comm[p]$ satisfies
\[q\left(\frac{\partial}{\partial x_1}, \ldots, \frac{\partial}{\partial x_g}\right)\Comm[p] = q(\nabla)\Comm[p] = 0.\]
\end{thm}

\begin{proof}
Proposition \ref{thm:commiffsymm} implies that
\[\dird[\Symm[p], q, h]=0 \quad \Longleftrightarrow \quad
\Comm[\dird[p, q, h]] = h^{\ell}q(\nabla)\Comm[p]=0.\]
Since $h \neq 0$, this implies that
\[D[\Symm[p], q, h]=0 \quad \Longleftrightarrow \quad
q(\nabla)\Comm[p] = 0.\]

\end{proof}

\begin{exa}The commutative polynomial $p$
\[p (x_1, x_2, x_3) =  x_1^2x_2^2 - x_1^2x_3^2 - x_2^2x_3^2 +
\frac{1}{3}x_3^4 \]
is harmonic.
The NC polynomial
\begin{align}
\Symm[p] = &\ \frac{1}{6}(x_1x_1x_2x_2
+ x_1x_2x_1x_2
+ x_2x_1x_1x_2
+ x_1x_2x_2x_1
+ x_2x_1x_2x_1 \notag\\
&+ x_2x_2x_1x_1
+ x_1x_1x_3x_3
+ x_1x_3x_1x_3
+ x_3x_1x_1x_3
+ x_1x_3x_3x_1\notag\\
&+ x_3x_1x_3x_1
+ x_3x_3x_1x_1
+ x_2x_2x_3x_3
+ x_2x_3x_2x_3
+ x_3x_2x_2x_3\notag\\
&+ x_2x_3x_3x_2
+ x_3x_2x_3x_2
+ x_3x_3x_2x_2 + 2x_3^2)\notag
\end{align}
is therefore harmonic by
Theorem \ref{thm:parta}.\qed
\end{exa}

\section{Classification of Noncommutative $\ell$-Harmonic Polynomials}
\label{sect:classification}

Recall that $p \in \FFx$ is  called \textbf{$\ell$-harmonic}\index{lharmonic@$\ell$-harmonic} if \index{lLaplacian@$\ell$-Laplacian}
\[
\lap_{\ell}[p,h] := \dird\left[p, \sum_{i=1}^g x_i^{\ell}, h\right] =  0.
\]
In the
special case that $\ell = 2$, such a polynomial is simply said to be
harmonic\index{harmonic}, and the differential
operator $\dird[p, \sum_{i=1}^g x_i^2, h] = \lap[p, h]$
is called the noncommutative Laplacian\index{Laplacian}.
The main result of this section is the proof of Theorem \ref{thm:main}, which classifies
all NC $\ell$-harmonic polynomials in $\CCx$.

\subsection{An Alternate Classification of the Set of Commutative $\ell$-Harmonic Polynomials}

As noted in $\S$ \ref{subsub:ClassIntro}, the set of $\ell$-harmonic polynomials
in $\CC[x]$ is spanned by the set of  polynomials of the form
\[(a_1 x_1 + \ldots + a_g x_g)^d \]
where $a_1^{\ell} + \ldots + a_g^{\ell} = 0$ (see \cite{R94}).
The following
proposition gives an alternate characterization of the set of $\ell$-harmonic commutative polynomials.

\begin{prop}
\label{thm:classificationcomm}
Let $x = (x_1, \ldots, x_g)$,
and let $p \in \FF[x]$
be a degree $d$ commutative polynomial.
Fix a variable $x_i$.
Express $p$ as
\[p = \sum_{r=0}^{d} x_i^{r} p_r \]
where for each $r$ either $p_r = 0$ or $\deg_i(p_r) = 0$.
The polynomial $p \in \FF[x]$ is $\ell$-harmonic
if and only if it equals
\begin{align}\label{eq:comm}
p = \sum_{r=0}^{\ell - 1} \sum_{k=0}^{\lfloor \frac{d}{\ell} \rfloor} \frac{(-1)^k}{(\ell k + r)!} x_i^{\ell k + r} \left(\Delta_{\ell}- \frac{\partial^{\ell}}{\partial x_i^{\ell}}\right)^k [p_{r}].
\end{align}
\end{prop}

\begin{proof}

To prove necessity, suppose $p$ is $\ell$-harmonic.
For $r > d$, define $p_r = 0$.  With this convention, $p$ is equal to
\[p = \sum_{r=0}^{\infty} x_i^r p_r.\]
Applying $\Delta_{\ell}$ to $p$ gives \begin{align}
\Delta_{\ell} [p] &= \Delta_{\ell}\left[\sum_{r=0}^{\infty} x_i^r p_r\right]
\\
\notag
&=
\sum_{r=1}^{\infty} r(r-1) \ldots (r - \ell + 1) x_i^{r-\ell} [p_{r}]
+ \sum_{r=0}^{\infty} x_i^r\left(\Delta_{\ell}- \frac{\partial^{\ell}}{\partial x_i^{\ell}}\right)[p_r] \\
&= \sum_{r=0}^{\infty} x_i^r\left( (r+\ell)\ldots (r + 1)p_{r+ \ell} +   \left(\Delta_{\ell}- \frac{\partial^{\ell}}{\partial x_i^{\ell}}\right)[p_r]\right)= 0.\notag
\end{align}
This produces a recursion relation
\[p_{r + \ell} = \frac{-1}{(r+\ell)\ldots (r + 1)}\left(\Delta_{\ell}- \frac{\partial^{\ell}}{\partial x_i^{\ell}}\right)p_r.\]
Given some polynomials
$p_{0}, \ldots, p_{\ell - 1}$, the remaining $p_{\ell k +r}$ are given by
\[p_{\ell k + r} = \frac{(-1)^k}{(\ell k + r)!}\left(\Delta_{\ell}- \frac{\partial^{\ell}}{\partial x_i^{\ell}}\right)^k p_{r}\]
for each $k$ and for $0 \leq r < \ell.$
Therefore $p$ is equal to
\[p = \sum_{r=0}^{\ell - 1} \sum_{k=0}^{\infty} \frac{(-1)^k}{(\ell k + r)!} x_i^{\ell k + r} \left(\Delta_{\ell}- \frac{\partial^{\ell}}{\partial x_i^{\ell}}\right)^k p_r.\]
From this (\ref{eq:comm}) follows since $\left(\Delta_{\ell}- \displaystyle\frac{\partial^{\ell}}{\partial x_i^{\ell}}\right)^k p_r$
equals zero for $\ell k > d \geq \deg(p_r)$.

To prove sufficiency, consider $p$ as defined in (\ref{eq:comm}).
When $k = 0$ and $r < \ell$,  the derivative $\displaystyle\frac{\partial^{\ell}}{\partial x_i^{\ell}}x_i^{\ell k + r}$ is equal to
\[\frac{\partial^{\ell}}{\partial x_i^{\ell}}x_i^{r}= 0.\]
 When $k > 0$ and $r < \ell$, the derivative $\displaystyle\frac{\partial^{\ell}p}{\partial x_i^{\ell}}x_i^{\ell k + r}$ is equal to
\[\frac{\partial^{\ell}}{\partial x_i^{\ell}}x_i^{\ell k + r} = (\ell k + r) \ldots (\ell[k-1] + r + 1) x_i^{\ell(k-1) + r}.\]
Therefore $\displaystyle\frac{\partial^{\ell}p}{\partial x_i^{\ell}}$ is equal to
\begin{align}
\label{eq:DerivOnlyXm}
\frac{\partial^{\ell}p}{\partial x_i^{\ell}} &= \frac{\partial^{\ell}}{\partial x_i^{\ell}}\sum_{r=0}^{\ell - 1} \sum_{k=0}^{\infty} \frac{(-1)^k}{(\ell k + r)!} x_i^{\ell k + r} \left(\Delta_{\ell}- \frac{\partial^{\ell}}{\partial x_i^{\ell}}\right)^k [p_{r}]\\
\notag &= \sum_{r=0}^{\ell - 1} \sum_{k=1}^{\infty} \frac{(-1)^k}{(\ell (k-1) + r)!} x_i^{\ell (k-1) + r} \left(\Delta_{\ell}- \frac{\partial^{\ell}}{\partial x_i^{\ell}}\right)^k [p_{r}].
\end{align}
Reindexing (\ref{eq:DerivOnlyXm}) by substituting $k+1$ for $k$ gives
\begin{align}
\label{eq:DerivOnlyXm2}
\frac{\partial^{\ell}p}{\partial x_i^{\ell}} = -\sum_{r=0}^{\ell - 1} \sum_{k=0}^{\infty} \frac{(-1)^k}{(\ell k + r)!} x_i^{\ell k + r} \left(\Delta_{\ell}- \frac{\partial^{\ell}}{\partial x_i^{\ell}}\right)^{k+1} [p_{r}].
\end{align}
The polynomial $\left(\Delta_{\ell}- \displaystyle\frac{\partial^{\ell}}{\partial x_i^{\ell}}\right)[p]$ is equal to
\begin{align}\label{eq:DerivOnRest}
\left(\Delta_{\ell}- \frac{\partial^{\ell}}{\partial x_i^{\ell}}\right)[p] &= \left(\Delta_{\ell}- \frac{\partial^{\ell}}{\partial x_i^{\ell}}\right)\left[\sum_{r=0}^{\ell - 1} \sum_{k=0}^{\infty} \frac{(-1)^k}{(\ell k + r)!} x_i^{\ell k + r} \left(\Delta_{\ell}- \frac{\partial^{\ell}}{\partial x_i^{\ell}}\right)^k [p_{r}]\right]\\
\notag &= \sum_{r=0}^{\ell - 1} \sum_{k=0}^{\infty} \frac{(-1)^k}{(\ell k + r)!} x_i^{\ell k + r} \left(\Delta_{\ell}- \frac{\partial^{\ell}}{\partial x_i^{\ell}}\right)^{k+1} [p_{r}].
\end{align}
Notice that the differential operator
$\left(\Delta_{\ell}- \displaystyle\frac{\partial^{\ell}}{\partial x_i^{\ell}}\right)$ does not act on the variable $x_i$.
Also note that (\ref{eq:DerivOnlyXm2}) is the negative of (\ref{eq:DerivOnRest}), which implies that
\[\Delta_{\ell}[p] = \frac{\partial^{\ell}p}{\partial x_i^{\ell}} + \left(\Delta_{\ell}- \frac{\partial^{\ell}}{\partial x_i^{\ell}}\right)[p] = 0. \]

\end{proof}

\begin{definition}Let $\ell$ be a positive integer. Let $p \in \FF[x]$ be equal to
\[p = \sum_{r=0}^{\ell-1} x_i^r p_r, \]
where each $p_r$ is either equal to $0$ or has degree $0$ in $x_i$.
Define $\Hm[p,x_i,\ell] \in \FF[x]$\index{Hmpxil@$\Hm[p,x_i,\ell]$} to be
\begin{align}\label{eq:commdefinition}
\Hm[p,x_i,\ell] := \sum_{r=0}^{\ell - 1} \sum_{k=0}^{\left\lfloor \frac{\deg(p)}{\ell} \right\rfloor} \frac{(-1)^k}{(\ell k + r)!} x_i^{\ell k + r} \left(\Delta_{\ell}- \frac{\partial^{\ell}}{\partial x_i^{\ell}}\right)^k [p_{r}].
\end{align}
\end{definition}

Note that Proposition \ref{thm:classificationcomm} implies that $\Hm[p,x_i,\ell]$ is harmonic.
The following Lemma will be useful for generating commutative (and noncommutative) $\ell$-harmonic polynomials with desirable properties.

\begin{lemma}
\label{prop:techOnHsubEll}
Let $r$ be an integer with $0 \leq r < \ell$.  Let $p \in \FFx$ be defined by
\begin{align}
\notag
p &= \frac{(-1)^Q (\ell Q + r)!}{Q!(\ell!)^Q}\sum_{r=0}^{\ell - 1} \sum_{k=0}^{\lfloor \frac{d}{\ell} \rfloor} \frac{(-1)^k}{(\ell k + r)!} x_i^{\ell k + r} \left(\Delta_{\ell}- \frac{\partial^{\ell}}{\partial x_i^{\ell}}\right)^k \left[x_{a_1}^{\ell} \ldots x_{a_Q}^{\ell} \right]\\
\notag &= \frac{(-1)^Q (\ell Q + r)!}{Q!(\ell!)^Q}\Hm\left[x_{a_1}^{\ell} \ldots x_{a_Q}^{\ell},x_i,\ell\right],
\end{align}
where $i, a_1, \ldots, a_Q$ are distinct positive integers.
Then $\deg_i(p) = \ell Q + r$ and the coefficient of $x_i^{\ell Q + r}$ in $p$ is
$1$.
\end{lemma}

\begin{proof}Each polynomial
\[\frac{(-1)^k}{(\ell k + r)!} x_i^{\ell k + r} \left(\Delta_{\ell}- \frac{\partial^{\ell}}{\partial x_i^{\ell}}\right)^k \left[x_{a_1}^{\ell} \ldots x_{a_Q}^{\ell}\right]\]
is either $0$ or degree $\ell Q + r$ with degree $\ell k + r$ in $x_i$.
The only piece of $p$ which can possibly be of degree $\ell Q + r$ in $x_i$ is
\begin{equation}
\label{eq:sufficescoeff}\frac{(-1)^Q (\ell Q + r)!}{Q!(\ell!)^Q}\frac{(-1)^Q}{(\ell Q + r)!} x_i^{\ell Q + r} \left(\Delta_{\ell}- \frac{\partial^{\ell}}{\partial x_i^{\ell}}\right)^Q \left[x_{a_1}^{\ell} \ldots x_{a_Q}^{\ell}\right].
\end{equation}
It therefore suffices to show that (\ref{eq:sufficescoeff}) is equal to $x_i^{\ell Q + r}$.

One sees
\begin{align}
\left(\Delta_{\ell}- \frac{\partial^{\ell}}{\partial x_i^{\ell}}\right)^Q &= \left(\sum_{j_1 \neq i} \frac{\partial^{\ell}}{x_{j_1}^{\ell}}\right)\ldots
\left(\sum_{j_Q \neq i} \frac{\partial^{\ell}}{x_{j_Q}^{\ell}}\right)
\notag\\
\label{eq:brokenoperator}
&= \sum_{j_1 \neq i} \ldots \sum_{j_Q \neq i} \frac{\partial^{\ell}}{x_{j_1}^{\ell}} \ldots \frac{\partial^{\ell}}{x_{j_Q}^{\ell}}.
\end{align}
Since
\[\deg_j(x_{a_1}^{\ell} \ldots x_{a_Q})^{\ell} = \left\{\begin{array}{cc}
\ell& j \in \{a_1, \ldots, a_Q\}\\
0& \mathrm{otherwise}
\end{array} \right.,\]
the only terms of (\ref{eq:brokenoperator}) which don't annihilate $x_{a_1}^{\ell} \ldots x_{a_Q}^{\ell}$ are those where $j_1, \ldots, j_Q$ are equal to $a_1, \ldots, a_Q$ in some order.  Since there are $Q!$ distinct orderings of $a_1, \ldots, a_Q$, one sees
\[\left(\Delta_{\ell}- \frac{\partial^{\ell}}{\partial x_i^{\ell}}\right)^Q x_{a_1}^{\ell} \ldots x_{a_Q}^{\ell} = Q!\frac{\partial^{\ell}}{x_{a_1}^{\ell}} \ldots \frac{\partial^{\ell}}{x_{a_Q}^{\ell}}
x_{a_1}^{\ell}\ldots x_{a_Q}^{\ell} = Q! (\ell!)^Q. \]
Inserting this into (\ref{eq:sufficescoeff}) gives the result.
\end{proof}

\subsection{Simple Cases of Non-Commutative $\ell$-Harmonic Polynomials}

First consider a couple of basic cases.
\subsubsection{Degree Bounded by $\ell$}

It is straightforward to classify the set of $\ell$-harmonic polynomials
which are homogeneous of degree bounded by $\ell$.

\begin{prop}
\label{prop:deg2harm}
Let $x = (x_1, \ldots, x_g)$, and let $\ell$ be
a positive integer.
\begin{enumerate}

\item \label{item:lowdeglharm}All homogeneous NC polynomials of degree $d < \ell$
are $\ell$-harmonic.

\item \label{item:elldeglharm}
The following is a basis over $\FF$ for the set of NC $\ell$-harmonic,
homogeneous degree $\ell$
polynomials:
\begin{equation}
\label{eq:basisd2homharm}
\mathcal{B} = \{x_k^{\ell} - x_1^{\ell}:\ 2 \leq k\leq g \} \cup \{m \in \mathcal{M}:\ |m| = \ell,\ m \neq x_i^{\ell},\ i = 1, \ldots, g\}.
\end{equation}
\end{enumerate}
\end{prop}
Note that (\ref{item:lowdeglharm}) proves Theorem \ref{thm:main} (\ref{item:lowdegmain}).

\begin{proof}
\begin{enumerate}
\item
Any polynomial $p$ such that
$\deg_i(p) < \ell$ satisfies $\dird[p,x_i^{\ell}, h] = 0$.
All polynomials $p$ with $\deg(p) < \ell$
satisfy $\deg_i(p) < \ell$ for each $x_i$, and must therefore be $\ell$-harmonic.

\item The set $\mathcal{B}$ in (\ref{eq:basisd2homharm}) is clearly linearly independent.
Note that all of the elements of
\[\{m \in \mathcal{M}:\ |m| = \ell,\ m \neq x_i^{\ell},\ i = 1, \ldots, g\}\]
must satisfy $\lap_{\ell}[p, h] = 0$ since all $m$ in this set have
 $\deg_i(m) < \ell$ for each $x_i$.
 Also, note that for each $k$
\[\lap_{\ell}[x_k^{\ell} - x_1^{\ell}, h] = \ell! h^{\ell} - \ell! h^{\ell} = 0. \]
Therefore, the span $\mathcal{B}$ is contained in the
set of NC $\ell$-harmonic homogeneous degree $\ell$ polynomials.

Conversely, let $p \in \FFx$ be an $\ell$-harmonic homogeneous degree $\ell$ polynomial of the form
\[p =  \sum_{|m|=\ell} A_{m}m.\]
Applying $\lap_{\ell}$ to $p$ gives
\[\lap_{\ell}[p, h] = \ell!\left(\sum_{k=1}^g A_{x_k^{\ell}}\right)h^{\ell} = 0.\]
This implies that $\displaystyle\sum_{k=1}^g A_{x_k^{\ell}} = 0$.
Subtracting $\left(\displaystyle\sum_{k=1}^g A_{x_k^{\ell}}\right)x_1^{\ell}$, i.e. $0$, from $p$ gives
\[p = \sum_{m \not\in \{x_1^{\ell}, \ldots, x_g^{\ell}\}} A_m m + \sum_{k=2}^g A_{x_k^{\ell}} (x_k^{\ell} - x_1^{\ell})\]
which is in the span of $\mathcal{B}$.
\end{enumerate}
\end{proof}

\subsubsection{1-Harmonic Polynomials}

In the case that $\ell = 1$, the $\ell$-Laplacian is equal to
\[\lap_{\ell}[p,h] = \dird\left[p, \sum_{i=1}^g x_i, h \right]. \]

\begin{definition}
Let $p\in \FFx$.  The expression $\dird[p,x_i,x_j]$ is defined to be
\[\dird[p,x_i,x_j]:= \frac{d}{dt} p(x_1, \ldots, x_i + tx_j, \ldots, x_g)|_{t=0}. \]
Essentially, $\dird[p,x_i,x_j]$ is $\dird[p,x_i,h]$ with each $h$ replaced with $x_j$.
\end{definition}

\begin{exa}For example, if $p(x_1, x_2) = x_1 x_1 x_2 x_2 x_1$
then \[ \dird[p, x_1^2, h] = 2hhx_2x_2x_1 + 2hx_1x_2x_2h + 2x_1hx_2x_2h.\]
Substituting $x_1$ in for $h$ gives
\[\dird[p, x_1^2, x_1] = 6x_1x_1x_2x_2x_1.\]
Differentiation and substitution in this example essentially
multiplied the polynomial $p$ by 6.\qed
\end{exa}

\begin{prop}
\label{prop:ellequalsone}
Every NC, homogeneous degree $d$, $1$-harmonic polynomial in $\FFx$ is a sum of polynomials of the form
\begin{equation}
\label{eq:1mainthm}
\prod_{i=1}^d \sum_{j=1}^g a_{ij}x_j,
\end{equation}
such that $\sum_{j=1}^g a_{i j} = 0$ for each $i$.
\end{prop}

\begin{proof}
Let $A \in \FF^{g \times g}$ be an orthogonal matrix such that
\[A\begin{pmatrix}1\\0\\\vdots\\0\end{pmatrix} =  \frac{1}{\sqrt{g}}\begin{pmatrix}1\\1\\\vdots\\1\end{pmatrix}.\]
By Corollary \ref{cor:chainrule}, if $p \in \FFx$ is $1$-harmonic and homogeneous of degree $d$, then
\[\dird\left[p, \sum_{i=1}^g x_i, h \right](x) = \dird[p\left(A^Tx\right), \sqrt{g} x_1, h](Ax) = 0.\]
This implies that $\deg_1\left(p\left(A^Tx\right) \right) = 0$ for the following reason:

Assume $\deg_1\left(p\left(A^Tx\right) \right) = d_1 > 0$. Let $p\left(A^Tx\right)$ be equal to
\[p\left(A^Tx\right) = p' + r, \]
where $p'$ is homogeneous in $x_1$ of degree $\deg(p(A^Tx))$ and where $\deg_1(p') > \deg_1(r)$.
Thus,
\[\dird[p,x_1,h] = \dird[p',x_1,h] + \dird[r,x_1,h] = 0. \]
Since $r$ has a smaller degree in $x_1$ than $p'$, which is homogeneous in $x_1$,
this implies that $\dird[p',x_1,h] = 0.$ Let $p'$ be equal to
\[p' = \sum_{\sigma \in S_d} \sigma\left[x_1^{d_1} q_{\sigma} \right]\]
for some polynomials $q_{\sigma}$ such that $\deg_1\left(q_{\sigma}\right) = 0$.
Therefore,
\begin{align}
\notag
\dird[p',x_1,h] &=
\dird\left[\sum_{\sigma \in S_d} \sigma\left[x_1^{d_1} q_{\sigma}\right] \right]\\
\notag &=
\sum_{\sigma \in S_d} \sigma\left[\dird\left[x_1^{d_1}, x_1, h\right] q_{\sigma} \right]=0.
\end{align}
Note that
\begin{align}
\notag
\dird\left[x_1^{d_1},x_1,h\right] &= \dird\left[\Symm\left[x_1^{d_1}, x_1, h\right]\right]\\
\notag &=\Symm\left[h\frac{\partial}{\partial x_1} x_1^{d_1}, x_1, h\right]\\
\notag&= d_1\Symm\left[hx_1^{d_1-1} \right].
\end{align}
Also,
\[\dird\left[x_1^{d_1},x_1,x_1\right] = d_1\Symm\left[x_1x_1^{d_1-1} \right] = d_1x_1^{d_1}\]
Therefore
\begin{align}
\dird[p\left(A^Tx\right),x_1,x_1] &=
\sum_{\sigma \in S_d} \sigma\left[\dird\left[x_1^{d_1}, x_1, x_1\right] q_{\sigma} \right]
\notag\\
\notag
&= d_1 p' = 0.
\end{align}
This implies that $p' = 0$, which is a contradiction. Therefore $\deg_1\left(p\left(A^Tx\right) \right) = 0$.

Each term of $p\left(A^Tx\right)$ contains no $x_i$. Therefore, each term of $p$ is of the form
\[\prod_{i=1}^d \sum_{j=1}^g a_{ij}x_j, \]
where each vector $\left(a_{i1}, \ldots, a_{ig}\right)^T$ corresponds to a column of $A$ which is not equal to the first column.  Since $A$ is orthogonal, and since the first column of $A$ equals $\displaystyle\frac{1}{\sqrt{g}}(1, \ldots, 1)^T$, this implies that
\[\sum_{j=1}^g a_{ij} = 0.\]
\end{proof}

\subsubsection{The Two-Variable Case}

The $\ell = 2$ case
is the case of NC harmonic polynomials in two variables, which
were classified in \cite{HMH09}. This classification is
repeated in Proposition \ref{prop:2dimensional}.  For more general $\ell$, the
following proposition classifies all NC $\ell$-harmonic polynomials in two
variables.

\begin{prop}
\label{prop:twovargenell}
Let $x = (x_1, x_2)$, and fix positive integers $d$ and $\ell$.
The following is a basis for the space of NC, homogeneous degree $d$, $\ell$-harmonic polynomials in $\CCx$:
\begin{align}
\label{eq:basis2varcase}
\mathcal{A}_d &=
\{m \in \cM:\  |m| = d,\ \deg_1(m), \deg_2(m) < \ell\} \\
\notag
&\cup \{\Symm[\Hm[x_1^rx_2^{d - r}, x_1, \ell]]:\ d - r \geq \ell,\ 0 \leq r < \ell \}.
\end{align}
\end{prop}

\begin{proof}The first set in $\mathcal{A}_d$ is either empty or contains
distinct monomials with $\deg_1(m), \deg_2(m) < \ell$.
Next consider the elements of the second set in $\mathcal{A}_d$.
For each $r$, by construction
the polynomial $\Hm\left[x_1^rx_2^{\ell - r}, x_1, \ell\right]$ is a sum of
terms whose degree in $x_1$ is equivalent to $r$ modulo $\ell$.  Further, the
terms of $\Hm\left[x_1^rx_2^{\ell - r}, x_1, \ell\right]$ either have degree in
$x_2$ greater than or equal to $\ell$ or degree in $x_1$ greater than or equal
to $\ell$.  Therefore $A_d$ is a linearly independent set.
Further, Theorem \ref{thm:parta} and Proposition \ref{prop:deg2harm}
(\ref{item:lowdeglharm}) imply that $\mathcal{A}_d$ is contained in the set of
NC, homogeneous degree $d$, $\ell$-harmonic polynomials in $\CCx$.

It now suffices to show that $\mathcal{A}_d$ spans the set of NC, homogeneous
degree $d$, $\ell$-harmonic polynomials in $\CCx$. Consider three cases.

\begin{enumerate}
\item For $d \leq \ell$, this follows by Proposition \ref{prop:deg2harm}.

\item Assume the proposition for all degrees less than or equal to $d = \ell + k$, with $0 \leq k < \ell - 1$. Let $p \in \CCx$ be a homogeneous degree $d+1$, $\ell$-harmonic polynomial.  Then $p$ is equal to
\[p = x_1 p_1(x) + x_2 p_2(x),\]
for some NC, homogeneous degree $d$ polynomials $p_1, p_2 \in \CCx$.  By
repeated application of the Product Rule, one sees
\begin{align}
\notag
\lap_{\ell}[p,h] &=  \sum_{i=1}^2 \sum_{n=0}^{\ell} \binom{\ell}{n} \dird\left[x_1, x_i^{n}, h \right]\dird\left[p_1, x_i^{\ell-n}, h \right]\\
 \notag &+ \sum_{i=1}^2 \sum_{n=0}^{\ell} \binom{\ell}{n}\dird\left[x_2, x_i^{n}, h \right]\dird\left[p_2, x_i^{\ell-n}, h \right]\\
\notag &=x_1 \lap_{\ell}[p_1,h] + x_2\lap_{\ell}[p_2,h] \\
\notag &+ \ell h\dird\left[p_1, x_1^{\ell-1},h\right] + \ell h\dird\left[p_2, x_2^{\ell - 1},h\right].
\end{align}
Since $\lap_{\ell}[p,h] = 0$, this implies that $p_1$ and $p_2$ are $\ell$-harmonic and that
\begin{equation}
\label{eq:laplIdentity}
 \dird\left[p_1, x_1^{\ell-1},h\right] + \dird\left[p_2, x_2^{\ell -
1},h\right] =0.
\end{equation}
By induction, $p_1$ and $p_2$ are in the span of $\mathcal{A}_{d}$.
If $p_1$ contains terms $\alpha m_1$ with $\deg_1(m_1) < \ell - 1$ and $\deg_2(m_1) < \ell$, then
$\deg_1(x_1 m_1), \deg_2(x_1m_1) < \ell$, in which case $x_1m_1$ is in the span
of $\mathcal{A}_d$.
Therefore, assume without loss of generality that $p_1$ contains no such terms. Similarly, assume that $p_2$ contains no terms $m_2$ with $\deg_1(m_2)<\ell$ and $\deg_2(m_2) < \ell - 1$.
Express $p_1$ and $p_2$ as
\begin{align}
p_1 &= \sum_{\substack{\deg_1(m_1) = \ell - 1\\\deg_2(m_1) < \ell}} \alpha_{m,1}
m_1 + \sum_{r=0}^{k} \beta_{r,1} \Symm\left[\Hm\left[x_1^rx_2^{d - r}, x_1,
\ell\right]\right]\notag\\
\notag
p_2 &= \sum_{\substack{\deg_2(m_2) = \ell - 1\\\deg_1(m_1) < \ell}} \alpha_{m,2}
m_2 + \sum_{r=0}^{k} \beta_{r,2} \Symm\left[\Hm\left[x_1^rx_2^{d - r}, x_1,
\ell\right]\right].
\end{align}
Given $r \leq k < \ell - 1$, by the definition of $\Hm[p,x_1,\ell]$,
\[\Hm\left[x_1^rx_2^{d - r}, x_1, \ell\right]= \frac{1}{r!}x_1^rx_2^{d-r} - \frac{1}{(r+\ell)!} x_1^{r+\ell}x_2^{d-r-\ell}. \]
Therefore,
\begin{align}
\label{eq:lapOfP1Part} \frac{\partial^{\ell-1}}{\partial x_1^{\ell-1}}
\Hm\left[x_1^rx_2^{d - r}, x_1, \ell\right] &= - \frac{1}{(r+1)!}
x_1^{r+1}x_2^{d-r-\ell} \\
\frac{\partial^{\ell-1}}{\partial x_2^{\ell-1}} \Hm\left[x_1^rx_2^{d - r}, x_1, \ell\right] &=  \frac{(d-r)!}{r!(d-r-\ell+1)!} x_1^{r}x_2^{d-r-\ell+1}.
\label{eq:lapOfP2Part}
\end{align}
Further, for each monomial $m_1$, with $|m_1| = d$ and $\deg_1(m_1) = \ell - 1$, there exists a permutation $\sigma \in S_d$ such that $m_1 = \sigma\left[x_1^{\ell-1}x_2^{d - \ell + 1} \right]$.  Therefore
\begin{align}
\notag
\dird\left[m_1, x_1^{\ell-1},h \right] &= \sigma\left[\dird\left[x_1^{\ell-1}x_2^{d - \ell + 1}, x_1^{\ell-1},h \right] \right]\\
\notag
&=  (\ell-1)! \sigma\left[h^{\ell-1}x_2^{d - \ell + 1} \right]\\
\notag
&= (\ell-1)!m_1(h,x_2).
\end{align}
Similarly, if $\deg_2(m_2) = \ell - 1$, then $\dird[m_2,x_2^{\ell-1},h] =
(\ell-1)!m_2(x_1,h)$.
Therefore plugging these expressions as well as (\ref{eq:lapOfP1Part})
and (\ref{eq:lapOfP2Part}) into (\ref{eq:laplIdentity}) gives
\begin{align}
\dird\left[p_1, x_1^{\ell-1},h\right] &+ \dird\left[p_2, x_2^{\ell - 1},h\right] =
\sum_{\substack{\deg_1(m_1) = \ell - 1\\\deg_2(m_1) < \ell}}
(\ell-1)!\alpha_{m,1} m_1(h,x_2) \notag\\
&- \sum_{r=0}^{k} \beta_{r,1} \frac{1}{(r+1)!} \Symm\left[h^{\ell-1}x_1^{r+1}x_2^{d-r-\ell}\right]
\notag\\
 &+ \sum_{\substack{\deg_2(m_2) = \ell - 1\\\deg_1(m_2)< \ell}}
(\ell-1)!\alpha_{m,2} m_2(x_1,h)
\notag\\
&+ \sum_{r=0}^{k} \beta_{r,2} \frac{(d-r)!}{r!(d-r-\ell+1)!} \Symm\left[h^{\ell-1}x_1^{r}x_2^{d-r-\ell+1}\right].
\notag
\end{align}

Grouping pieces with the same degree in $x_1$ and in $x_2$ gives
\begin{align}
\sum_{m_1} (\ell-1)!\alpha_{m,1} m_1(h,x_2) +
\beta_{0,2} \frac{d!}{(d+1-\ell)!} \Symm[h^{\ell-1}x_2^{k+1}]&=0
\label{eq:eqs}\\
-\beta_{0,1}
\Symm[h^{\ell-1}x_1x_2^{k}]
+
\beta_{1,2} \frac{(d-1)!}{(d-\ell)!} \Symm[h^{\ell-1}x_1x_2^{k}]&=0
\label{eq:starteqs}
\\
\vdots
\notag
\\
-\beta_{k-1,1} \frac{1}{(k-1)!}
\Symm[h^{\ell-1}x_1^kx_2]
+
\beta_{k,2} \frac{\ell!}{k!} \Symm[h^{\ell-1}x_1^kx_2]&=0
\label{eq:endeqs}\\
-\beta_{k,1} \frac{1}{k!}
\Symm[h^{\ell-1}x_1^{k+1}]
+ \sum_{m_2} (\ell-1)!\alpha_{m,2} m_2(x_1,h)&=0.
\label{eq:lasteq}
\end{align}
This first line (\ref{eq:eqs}) implies that the $\alpha_{m_1,1}$ are all equal and determined by $\beta_{0,2}$ since the coefficients of each monomial in
\[\beta_{0,2} \frac{d!}{(d+1-\ell)!} \Symm[h^{\ell-1}x_2^{k+1}]\]
are all equal. Similarly, the last line (\ref{eq:lasteq}) implies that the $\alpha_{m_2,2}$ are all equal
and determined by $\beta_{k,1}$ since the coefficients of each monomial in
\[-\beta_{k,1} \frac{1}{k!}
\Symm[h^{\ell-1}x_1^{k+1}]\]
are all equal.
Further, the set of $p_1$ and $p_2$ which produce solutions to the system
of equations defined by line (\ref{eq:starteqs}) through (\ref{eq:endeqs}), as
determined by the $\alpha_{m,n}$ and $\beta_{i,j}$, is $(k+1)$-dimensional.  The
span of the second set of
$\mathcal{A}_{d+1}$ as defined in (\ref{eq:basis2varcase}) is also
$(k+1)$-dimensional (in this case with $d+1$ instead of $d$) and is contained in
the set of possible $\ell$-harmonic polynomials of the form $x_1p_1 + x_2p_2$.
Therefore $p$ must be in the span of (\ref{eq:basis2varcase}), which implies
that $\mathcal{A}_{d+1}$ spans the set of NC $\ell$-harmonic, homogeneous degree
$d+1$ polynomials.

\item Now assume the proposition for all degrees less than or equal to $d =
\ell + k$, with $k \geq \ell$.  Let $p \in \FFx$ be a NC, homogeneous degree
$d+1$, $\ell$-harmonic polynomial.  There does not exist a monomial $m$
with $\deg_1(m) < \ell$ and $\deg_2(m) < \ell$ since $d > 2(\ell-1)$, hence $p$
contains no terms in the first set of (\ref{eq:basis2varcase}). As in step 2,
    \begin{align}
    \notag p &= x_1 p_1 + x_2 p_2\\
    &= x_1 \sum_{r=0}^{\ell-1} \beta_{r,1} \Symm\left[\Hm\left[x_1^rx_2^{d-r}, x_1, \ell\right]\right]
    \notag\\
    &+ x_2 \sum_{r=0}^{\ell-1} \beta_{r,2}
\Symm\left[\Hm\left[x_1^rx_2^{d-r},x_1, \ell\right]\right],
    \notag
    \end{align}
with the additional requirement that
\begin{align}
 \label{eq:additionalReq}
&\dird\left[\sum_{r=0}^{\ell-1} \beta_{r,1} \Symm\left[\Hm\left[x_1^rx_2^{d-r},
x_1, \ell\right]\right],x_1^{\ell-1},h\right]\\
\notag
&+
\dird\left[\sum_{r=0}^{\ell-1} \beta_{r,2}
\Symm\left[\Hm\left[x_1^rx_2^{d-r},x_1, \ell\right]\right] , x_2^{\ell-1},h
\right] = 0.
\end{align}

    One sees that for each $r$,
    \[\dird\left[\Symm\left[\Hm\left[x_1^rx_2^{d-r},x_1, \ell\right]\right],x_1^{\ell-1},h\right] \]
    must be $\ell$-harmonic and  have terms with degree in $x_1$ equivalent to
    $r+1$ modulo $\ell$, and
    \[\dird\left[\Symm\left[\Hm\left[x_1^rx_2^{d-r},x_1, \ell\right]\right],x_2^{\ell-1},h\right] \]
    must be $\ell$-harmonic and have terms with
    degree in $x_1$ equivalent to $r$ modulo $\ell$.
    It is also clear by the definition of $\Hm\left[x_1^rx_2^{d-r},x_1,
\ell\right]$ that these derivatives are nonzero.
Since these derivatives are symmetrized and $\ell$-harmonic in $x_1$ and $x_2$
with $\deg(x_1), \deg(x_2)$ specified modulo $\ell$, Theorem \ref{thm:parta} and
Proposition \ref{thm:classificationcomm} imply what they are equal to, up to a
nonzero scalar multiple. Specifically,
    for $0 \leq r < \ell -1$,
    \begin{align}
    \dird\left[\Symm\left[\Hm\left[x_1^rx_2^{d-r},x_1, \ell\right]\right],x_1^{\ell-1},h\right]
    \\
    \notag=C_{r,1}\Symm\left[h^{\ell-1}\Hm\left[x_1^{r+1}x_2^{d-\ell-r},x_1, \ell\right]\right],
    \notag\end{align}
    and
    \begin{align}
    \dird\left[\Symm\left[\Hm\left[x_1^rx_2^{d-r},x_1, \ell\right]\right],x_2^{\ell-1},h\right]
    \\
    \notag=C_{r,2}\Symm\left[h^{\ell-1}\Hm\left[x_1^{r}x_2^{d+1-\ell-r},x_1, \ell\right]\right]
    \notag
    \end{align}
    for some nonzero scalars $C_{r,1}$ and $C_{r,2}$.  For $r = \ell -1$,
    \begin{align}
    \notag
    \dird\left[\Symm\left[\Hm\left[x_1^{\ell-1}x_2^{d+1-\ell},x_1, \ell\right]\right],x_1^{\ell-1},h\right]
    \\
    =C_{\ell-1,1}\Symm\left[h^{\ell-1}\Hm\left[x_2^{d+1-2\ell},x_1, \ell\right]\right],
    \notag
    \end{align}
    and
    \begin{align}
    \dird\left[\Symm\left[\Hm\left[x_1^{\ell-1}x_2^{d+1-\ell},x_1, \ell\right]\right],x_2^{\ell-1},h\right]
    \notag\\
    =C_{\ell-1,2}\Symm\left[h^{\ell-1}\Hm\left[x_1^{\ell-1}x_2^{d+2-2\ell},x_1, \ell\right]\right],
    \notag
    \end{align}
    some nonzero $C_{\ell-1,1}$ and $C_{\ell-1,2}$.
    Grouping terms with the same degree in $x_1$ and $x_2$ modulo
$\ell$ in (\ref{eq:additionalReq}) gives
    \begin{align}
    \notag
    C_{\ell-1,1}\beta_{\ell-1,1}\Symm[h^{\ell-1}\Hm[x_2^{d+1-\ell},x_1, \ell]]
    + C_{0,2}\beta_{0,2}\Symm[h^{\ell-1}\Hm[x_2^{d+1-\ell},x_1, \ell]] &=0\\
    C_{0,1}\beta_{0,1}\Symm[h^{\ell-1}\Hm[x_1x_2^{d-\ell},x_1, \ell]]
    + C_{1,2}\beta_{1,2}\Symm[h^{\ell-1}\Hm[x_1x_2^{d-\ell},x_1, \ell]] &=0
    \notag\\
    \vdots
    \notag\\
    \notag
C_{\ell-2,1}\beta_{\ell-2,1}\Symm[h^{\ell-1}\Hm[x_1^{\ell-1}x_2^{d+2-2\ell},x_1,
\ell]]&\\
    \notag
    +
C_{\ell-1,2}\beta_{\ell-1,2}\Symm[h^{\ell-1}\Hm[x_1^{\ell-1}x_2^{d+2-2\ell},x_1,
\ell]] &=0
    \end{align}
    Therefore the space of suitable $p_1$ and $p_2$ is $(\ell-1)$-dimensional,
    which is the same as the dimension of the span of $\mathcal{A}_{d+1}$.  Therefore $\mathcal{A}_{d+1}$ spans the set of NC $\ell$-harmonic, homogeneous degree $d+1$ polynomials.
\end{enumerate}

\end{proof}

\subsection{Breaking Up $\ell$-Harmonic Polynomials}
\label{sub:breakingup}

This subsection shows how to express NC $\ell$-harmonic polynomials
as a sum of manageable pieces.
The results of this section will be useful in proving Theorem \ref{thm:main}.

\subsubsection{Breaking up by Degree}
\label{subsub:oddsandevens}

Let $p \in \FFx$ be a NC polynomial. Fix a variable $x_i$.
One may express $p$ uniquely as
\[p = \sum_{j=0}^{finite} p_{j} \]
where each $p_j$ is either homogeneous of degree $j$ in $x_i$ or equal to $0$.
Given a positive integer $\ell$, one may group the $p_j$ together by
equivalence classes to get
\[p_{[k]} = \sum_{j \equiv {k}\bmod{\ell}} p_{j},\]
so that
\[p =  p_{[0]} + p_{[1]} + \ldots + p_{[\ell - 1]}.\]

\begin{prop}
\label{prop:oddsandevens}
Let $x = (x_1, \ldots, x_g)$, let $\ell$ be a
positive integer, and fix a variable $x_i$.
If $p \in \FFx$ is $\ell$-harmonic, then for each $k$, the polynomial $p_{[k]}$ is $\ell$-harmonic.
\end{prop}

\begin{proof}
Let $p \in \FFx$.
Given a NC polynomial which is homogeneous in $x_i$,
taking the derivative of a with respect to $x_i^{\ell}$ either annihilates it or decreases its degree in $x_i$ by $\ell$, whereas taking the derivative of that polynomial with respect to $x_j^{\ell}$, for $j \neq i$, either annihilates it or maintains its degree in $x_i$.
Therefore, $\lap_{\ell}[p_{[k]},h]$ is either zero or a sum of terms whose
degree in $x_i$ equals $k$ modulo $\ell$.
Since $\lap_{\ell}[p,h] = \lap_{\ell}\left[\sum_{k=0}^{\ell-1} p_{[k]},h\right]$, for each $k$ the only possible terms of $\lap[p,h]$ which have degree in $x_i$ equivalent to $k$ mod $\ell$ are the terms of $\lap_{\ell}[p_{[k]},h]$.
If $p$ is $\ell$-harmonic, then $\lap_{\ell}[p_{[k]}, h] = 0$ for each $k$.
\end{proof}

\subsubsection{Factoring out $x_i$}

\begin{lemma}
\label{lemma:neighbor}Fix a variable $x_i$ and fix positive integers $d$ and $d_i$, with $d_i \leq d$.
There exist $\sigma_1, \ldots, \sigma_N \in S_d$ (not necessarily unique) such that
the set
\begin{equation}
\label{eq:basisLemmaNeighbor}
\mathcal{B} = \{\sigma_j[x_i^{d_i}m]:\ j \in \{1, \ldots, N\},\ m \in \mathcal{M},\ |m| = d - d_i,\ \deg_i(m) = 0 \}
\end{equation}
is a basis for the set of NC, homogeneous degree $d$ polynomials which are
homogeneous of degree $d_i$ in $x_i$.
\end{lemma}

\begin{proof}Let $I_1, \ldots, I_N$ be the distinct subsets of
of $\{1, \ldots, d\}$ of size $d_i$. Choose $\sigma_1, \ldots, \sigma_N \in S_d$
such that for each $j$, the set
$\{\sigma_j(1), \ldots, \sigma_j(d_i)\}$ is equal to $I_j.$
Let $m \in \mathcal{M}$ be the monomial
\[m = x_{a_1} \ldots x_{a_d},\]
with $\deg_i(m) = d_i$. Let the set
$\{k:\ a_k = i\}$ be equal to $I_j$ for some $j$.
Then $\sigma_j[m]$ is equal to
\begin{align}
\sigma_j[m] &= x_{a_{\sigma_j(1)}}\ldots x_{a_{\sigma_j(d)}}\\
\notag &= x_i^{d_i}x_{a_{\sigma_j(d_i+1)}}\ldots x_{a_{\sigma_j(d)}},
\end{align}
since
\[\{\sigma_j(1), \ldots, \sigma_j(d_i)\} = I_j = \{k:\ a_k = i\}.\]
Let $\overline{m}$ be equal to
\[\overline{m} =  x_{a_{\sigma(d_i+1)}}\ldots x_{a_{\sigma(d)}}\]
so that
\[m = \sigma_j[x_i^{d_i}\overline{m}].\]
Therefore, the set $\mathcal{B}$ from (\ref{eq:basisLemmaNeighbor})
spans the set of homogeneous degree $d$ polynomials which are
homogeneous of degree $d_i$ in $x_i$.

To prove linear independence, it suffices to show that each monomial
$\sigma_j[x_i^{d_i}m]$ is distinct.
Suppose
\[\sigma_{j_1}[x_i^{d_i}m_1] = \sigma_{j_2}[x_i^{d_i}m_2]. \]
The monomial $\sigma_{j_1}[x_i^{d_i}m_1]$ has an $x_i$ at
each entry corresponding to an element of $I_{j_1}$, and the monomial $\sigma_{j_2}[x_i^{d_i}m_2]$
has an $x_i$ at each entry corresponding to an element of $I_{j_2}$. Therefore $j_1 = j_2$.
Further,
\[ x_i^{d_i}m_1 = x_i^{d_i}m_2,\]
which implies that $m_1 = m_2$.
\end{proof}

\begin{prop}
\label{thm:neighbor}
Fix a variable $x_i$, and fix positive integers $d$ and $d_i$, with $d_i \leq d$.
Let $\sigma_1, \ldots, \sigma_N \in S_d$ be
a set of permutations satisfying the conclusions of Lemma \ref{lemma:neighbor}.

\begin{enumerate}
\item
Fix a variable $x_i$.
Let $p\in \FFx$ be homogeneous of
degree $d$ with $\deg_i(p) = d_i$.
The polynomial $p$ may be expressed uniquely as
\begin{equation}
\label{eq:neighboreq}
p = \sum_{j=1}^N \sigma_j[x_i^{d_i}p_j] + r,
\end{equation}
such that the following hold.
\begin{itemize}
\item each polynomial $p_j$ is homogeneous of degree $d - d_i$ with $\deg_i(p_j) = 0$;

\item the polynomial $r$ is homogeneous of degree $d$ with $deg_i(r) < d_i$.
\end{itemize}
\item
If $p$ is $\ell$-harmonic, then each $p_j$ in
(\ref{eq:neighboreq}) is also $\ell$-harmonic.
\end{enumerate}

\end{prop}

\begin{proof}

\begin{enumerate}
\item
If $p$ is degree $d_i$ in $x_i$ then let $p = q + r$,
where $q$ is homogeneous of degree $d_i$ in $x_i$, and where $\deg_i(r) < d$.
By Lemma \ref{lemma:neighbor},
$q$ is uniquely expressed as
\[q = \sum_{j=1}^N \sum_{m \in \mathcal{M}} A_{m,j}\sigma_j[x_i^{d_i}m], \]
for some scalars $A_{m,j} \in \FF$.
For each $j$, let $p_j = \displaystyle\sum_{m \in \mathcal{M}} A_{m,j}m$. Then
(\ref{eq:neighboreq}) holds.

\item
Suppose $\lap_{\ell}[p,h] = 0$.
Applying $\lap_{\ell}$ to (\ref{eq:neighboreq})
yields
\begin{align}
\label{eq:lapofneighborsum}\lap_{\ell}[p,h]  &= \lap_{\ell}[r,h]  + \dird\left[\sum_{j=1}^N \sigma_j[x_i^{d_i}p_{j}], x_i^{\ell}, h]\right]
\\
\label{eq:thirdterm}
&+ \dird\left[\sum_{j=1}^N \sigma_j [x_i^{d_i}p_j], \sum_{k \neq i} x_k^{\ell}, h\right] = 0.
\end{align}
The first two terms, line (\ref{eq:lapofneighborsum}), have degree less than $d_i$ in $x_i$
whereas the third term, line
(\ref{eq:thirdterm}),
is either $0$ or homogeneous of degree $d_i$ in $x_i$. Therefore (\ref{eq:thirdterm}) must be $0$.

Examining (\ref{eq:thirdterm}) more closely shows that for each $j$,
\begin{align}
\label{eq:EachPieceDiff}
\dird\left[\sigma_j [x_i^{d_i}p_j], \sum_{k \neq i} x_k^{\ell}, h\right] &= \sigma_j\left[x_i^{d_i}\dird\left[p_j, \sum_{k \neq i} x_k^{\ell}, h\right]\right]\\
& = \sigma_j\left[x_i^{d_i}\lap_{\ell}[p_j, h]\right].\notag
\end{align}
Therefore,
\begin{equation}
\label{eq:SimplifNUnique}\sum_{j=1}^N \sigma_j\left[x_i^{d_i}\lap_{\ell}[p_j, h]\right] = 0.
\end{equation}
By Lemma \ref{lemma:neighbor}, the equation
(\ref{eq:SimplifNUnique}) is a unique representation of the zero polynomial. Therefore,
for each $j$, the expression $\lap_{\ell}[p_j, h]$ is equal to zero, or in other words,
each $p_j$ is $\ell$-harmonic.
\end{enumerate}
\end{proof}

\subsection{"Fully Degree $\ell$" Polynomials}
\label{sect:fullydeg2}

A key step to proving Theorem \ref{thm:main} is to prove it
in the special case of polynomials $p \in \FFx$ which are sums of terms which,
for each variable $x_i$,
are homogeneous of either degree $0$
or degree $\ell$ in $x_i$.
Such polynomials are said to be \textbf{fully degree $\ell$} \index{fully degree
$\ell$}.
An equivalent definition is to say that a polynomial is fully degree $\ell$ if it is a linear combination
of permutations of monomials of the form $x_{a_1}^{\ell} \ldots x_{a_d}^{\ell}$,
where $a_1, \ldots, a_d$ are distinct positive integers.

\begin{definition}
Fix positive integers $\ell$ and $d$.
 Let $P_{\ell,d} \subset S_{\ell d}$ \index{Pld@$P_{\ell,d}$} be the
set of permutations defined by
\[ P_{\ell, d} := \left\{ \sigma \in S_{\ell d}:\ \exists
\tau \in S_d,\ \forall (a_1, \ldots, a_d) \in
\mathbb{Z}^d,\ \sigma\left[x_{a_1}^{\ell} \ldots
x_{a_d}^{\ell}\right] = x_{a_{\tau(1)}}^{\ell} \ldots
x_{a_{\tau(d)}}^{\ell}\right\}. \]
\end{definition}

\begin{lemma}
Fix positive integers $\ell$ and $d$.
 The set $P_{\ell,d}$ is a subgroup of $S_{\ell d}$.
\end{lemma}

\begin{proof}
If $\sigma, \omega \in P_{\ell, d}$,
then consider $\sigma^{-1} \omega$.
We have for all $(a_1, \ldots, a_d) \in \mathbb{Z}^d$
\[ \omega\left[x_{a_1}^{\ell} \ldots
x_{a_d}^{\ell}\right] = x_{a_{\nu(1)}}^{\ell} \ldots x_{a_{\nu(d)}}^{\ell}\]
for some fixed $\nu \in S_d$ and 
  \[ \sigma\left[x_{a_1}^{\ell} \ldots
x_{a_d}^{\ell}\right] = x_{a_{\tau(1)}}^{\ell} \ldots x_{a_{\tau(d)}}^{\ell}\]
for some fixed $\tau \in S_d$.

This implies that for each $i$, $\sigma$ sends $1 + \ell(\tau(i) - 1), \ldots,
\ell + \ell(\tau(i) - 1)$ to  $1 + \ell(i - 1), \ldots,
\ell + \ell(i- 1)$ in some order.

Therefore,  for each $i$, $\sigma^{-1}$ sends $1 + \ell(i - 1), \ldots,
\ell + \ell(i- 1)$ to $1 + \ell(\tau(i) - 1),
\ldots,
\ell + \ell(\tau(i) - 1)$ in some order.

Therefore,
 for each $i$, $\sigma^{-1}$ sends $1 + \ell(\tau^{-1}(i) - 1), \ldots,
\ell + \ell(\tau^{-1}(i)- 1)$ to $1 + \ell(i - 1),
\ldots,
\ell + \ell(i - 1)$ in some order.

Therefore, for each $i$, $\sigma^{-1}\omega $ sends $1 + \ell(\nu(\tau^{-1}(i))
- 1), \ldots,
\ell + \ell(\nu(\tau^{-1}(i))- 1)$ to $1 + \ell(i - 1),
\ldots,
\ell + \ell(i - 1)$ in some order.

Therefore,
\[\sigma^{-1} \omega\left[x_{a_1}^{\ell} \ldots
x_{a_d}^{\ell}\right] = x_{a_{\nu(\tau^{-1}(1))}}^{\ell} \ldots
x_{a_{\nu(\tau^{-1}(d))}}^{\ell}. \]
Therefore $\sigma^{-1} \omega \in P_{\ell,d}$.
\end{proof}

\begin{definition}
Let $d,g$ be positive integers with $d \leq g$.
 Let $S_{d,g} \subset S_{g}$ \index{Sdg@$S_{d,g}$} be the set
\[ S_{d,g}:= \{ \sigma \in S_g:\ \sigma(i) = i\ \forall\ 1 \leq i \leq d\}.\]
\end{definition}

\begin{lemma}
Let $d,g$ be positive integers with $d \leq g$.
Then $S_{d,g}$ is a subgroup of $S_g$.
\end{lemma}

\begin{proof}
 Let $\sigma, \tau \in S_{d,g}$.
Then $\sigma(i) = \tau(i) = i$ for $1 \leq i \leq d$.
Consequently, $\sigma^{-1}(i) = i$ for $1 \leq i \leq d$.
Therefore, for $1 \leq i \leq d$,
\[ \sigma^{-1} \tau(i) = i.\]
Therefore $\sigma^{-1} \tau \in S_{d,g}$.
\end{proof}

\begin{lemma}
 \label{lemma:fully}
Let $x = (x_1, \ldots, x_g)$, and fix integers $\ell$ and $d$ with $d \leq g$.
Let $\sigma_1 P_{\ell, d}, \ldots, \sigma_M P_{\ell, d}$ be the distinct left
cosets of $P_{\ell, d}$ in $S_{\ell d}$, and let $\tau_1S_{d, g}, \ldots,
\tau_NS_{d, g}$ be the
distinct left cosets of $S_{d, g}$ in $S_g$.
Let $p \in \FFx$ be a homogeneous degree $\ell d$, fully degree $\ell$
polynomial.
Then $p$ can be uniquely expressed as
\begin{equation}
 \label{eq:formOfFullyDegEll}
p = \sum_{i=1}^M \sigma_i \left[  \sum_{j=1}^N A_{i,j}
x_{\tau_j(1)}^{\ell} \ldots x_{\tau_j(d)}^{\ell} \right].
\end{equation}
for some scalars $A_{i,j}$. Further, if $p$ is $\ell$-harmonic, then for each
$i$ the polynomial
\[  \sum_{j=1}^N A_{i,j}x_{\tau_j(1)}^{\ell}
\ldots x_{\tau_j(d)}^{\ell}\]
is also $\ell$-harmonic.
\end{lemma}

\begin{proof}

\noindent \textbf{Existence of representation}. By definition, a homogeneous
degree $\ell d$, fully degree $\ell$ polynomial is
a linear combination
of monomials of the form
\[ \sigma\left[ x_{\tau(1)}^{\ell} \ldots x_{\tau(d)}^{\ell} \right]\]
for some $\sigma \in S_{\ell d}$ and $\tau \in S_g$.
Such a monomial is therefore of the form
\[ \sigma_i \tilde{\sigma} \left[ x_{\tau (1)}^{\ell} \ldots
x_{\tau (d)}^{\ell} \right]\]
for some $\sigma_i$ and some $\tilde{\sigma} \in P_{\ell,d}$.
Since $\tilde{\sigma} \in P_{\ell, d}$, it must be that
\[\tilde{\sigma} \left[ x_{\tau (1)}^{\ell} \ldots
x_{\tau (d)}^{\ell} \right] = x_{\tau^\prime (1)}^{\ell}
\ldots
x_{\tau^\prime(d)}^{\ell}\]
for some $\tau^\prime \in S_d$.
Further, we can take $\tau^\prime$ to be equal to
$\tau_j \tilde{\tau}$, for some $\tau_j$ and some $\tilde{\tau}
\in S_{d,g}$. Since $\tilde{\tau}(i) = i$ for $1 \leq i \leq d$, we have
\[\sigma\left[ x_{\tau(1)}^{\ell} \ldots x_{\tau(d)}^{\ell} \right]
= \sigma_i\left[ x_{\tau_j(1)}^{\ell} \ldots x_{\tau_j(d)}^{\ell} \right]. \]
Therefore every homogeneous degree $\ell d$, fully degree $\ell$ polynomial is
of the form (\ref{eq:formOfFullyDegEll}).

\noindent \textbf{Uniqueness of representation}. To see uniqueness, suppose
\[ \sigma_{i_1}\left[ x_{\tau_{j_1}(1)}^{\ell} \ldots x_{\tau_{j_1}(d)}^{\ell}
\right] = \sigma_{i_2}\left[ x_{\tau_{j_2}(1)}^{\ell} \ldots
x_{\tau_{j_2}(d)}^{\ell}
\right].\]
Then we have
\[  x_{\tau_{j_1}(1)}^{\ell} \ldots x_{\tau_{j_1}(d)}^{\ell}
= \sigma_{i_1}^{-1} \sigma_{i_2}\left[ x_{\tau_{j_2}(1)}^{\ell} \ldots
x_{\tau_{j_2}(d)}^{\ell}
\right].\]
In this case, the indices $\tau_{j_1}(1), \ldots, \tau_{j_1}(d)$ must be some
permutation of the indicies $\tau_{j_2}(1), \ldots, \tau_{j_2}(d)$ by
definition of the permutation operation on NC monomials.
This implies that $\sigma_{i_1}^{-1} \sigma_{i_2} \in P_{\ell,d}$, which
implies
that $\sigma_{i_1}$ and $\sigma_{i_2}$ are in the same left coset of $
P_{\ell,d}$, which implies, since the $\sigma_i$ were chosen to be in distinct
cosets, that $i_1 = i_2$.
So now we have
\[ x_{\tau_{j_1}(1)}^{\ell} \ldots x_{\tau_{j_1}(d)}^{\ell} =
x_{\tau_{j_2}(1)}^{\ell} \ldots
x_{\tau_{j_2}(d)}. \]
This implies that $\tau_{j_1}(i) = \tau_{j_2}(i)$ for $1 \leq i \leq d$.
Therefore $\tau_{j_1}^{-1} \tau_{j_2}(i) = i$ for  $1 \leq i \leq d$.
Therefore $\tau_{j_1}$ and $\tau_{j_2}$ are in the same left coset of $S_{d,g}$,
so $j_1 = j_2$.
Therefore for each distinct $i$ and $j$ the terms of
(\ref{eq:formOfFullyDegEll}) are distinct monomials, which implies
that the representation given in (\ref{eq:formOfFullyDegEll}) is unique.

\noindent \textbf{$\ell$-harmonicity}.
Finally, suppose $p$ is $\ell$-harmonic.
Then
 \begin{equation}
\label{eq:lapOfFullyEll}
  \lap_{\ell}[p] = \ell! \sum_{i=1}^M \sigma_i \left[ 
\sum_{j=1}^N
A_{i,j}
\left( h^{\ell} x_{\tau_j(2)}^{\ell} \ldots x_{\tau_j(d)}^{\ell} + \ldots
+  x_{\tau_j(1)}^{\ell} \ldots x_{\tau_j(d-1)}^{\ell}h^{\ell}\right)\right] = 0
 \end{equation}
Consider the representation of $\lap_{\ell}[p,h]$ given by
(\ref{eq:formOfFullyDegEll}), with $h$ in the place of $x_{g+1}$.  We get said
representation by grouping together terms with the same $\sigma_i$.
We see that the monomials
\[ \sigma_i[h^{\ell} x_{\tau_j(2)}^{\ell} \ldots x_{\tau_j(d)}^{\ell}],
\ldots,
 \sigma_i[x_{\tau_j(1)}^{\ell} \ldots x_{\tau_j(d-1)}^{\ell}h^{\ell}]\]
over all $\tau_j$
are the only monomials arising from $\sigma_i$ in the equation
(\ref{eq:formOfFullyDegEll}).
Therefore
\begin{align}
\notag
 \lap_{\ell} \left[ \sum_{j=1}^N A_{i,j}x_{\tau_j(1)}^{\ell}
\ldots x_{\tau_j(d)}^{\ell}\right] &=
\\
 \notag
=\ell! \sum_{j=1}^N
A_{i,j}
\left( h^{\ell} x_{\tau_j(2)}^{\ell} \ldots x_{\tau_j(d)}^{\ell} + \ldots
+  x_{\tau_j(1)}^{\ell} \ldots x_{\tau_j(d-1)}^{\ell}h^{\ell}\right) &= 0,
\end{align}
which implies that
\[ \sum_{j=1}^N A_{i,j}x_{\tau_j(1)}^{\ell}
\ldots x_{\tau_j(d)}^{\ell} \]
is $\ell$-harmonic.
\end{proof}

\begin{prop}
\label{prop:keytechnicalstep}Let $x = (x_1, \ldots, x_g)$, let $\ell$ and $d$ be
positive integers, with $\ell \geq 2$ and $g \geq 2 d$. The set
$\mathcal{B}_{g,\ell,d}$ defined by
\[\mathcal{B}_{g,\ell,d} :=
\left\{\sigma\left[\left(x_{\tau(1)}^{\ell} -
x_{\tau(2)}^{\ell}\right)\ldots\left(x_{\tau(2d-1)}^{\ell} -
x_{\tau(2d)}^{\ell}\right) \right]:\ \sigma \in S_{\ell d}, \tau \in S_{g}
\right\} \]
spans the set of NC, $\ell$-harmonic, fully degree $\ell$, homogeneous degree
$\ell d$ polynomials.
\end{prop}

\begin{proof}
First note that the elements of each $\mathcal{B}_{g,d,\ell}$ are
$\ell$-harmonic by Proposition \ref{prop:ProdSuff}, since each element is a
permuation of an independent product of
symmetrized $\ell$-harmonic polynomials.  Also note that if $b \in
\mathcal{B}_{g,d-1,\ell}$ and $\deg_i(b) = \deg_j(b) = 0$, where $i \neq j$,
then
$(x_i^{\ell} - x_j^{\ell})b \in \mathcal{B}_{g,d,\ell}$.

Proceed by induction on $d$.  For $d=1$, by
Proposition \ref{prop:deg2harm} the only $\ell$-harmonic, fully degree $\ell$
polynomials
 which are homogeneous of degree $\ell$
 are linear combinations of $x_k^{\ell} - x_1^{\ell}$, which are in
$\mathcal{B}_{g,\ell,1}$.
Therefore $\mathcal{B}_{g,\ell,1}$
spans the set of NC, $\ell$-harmonic, fully degree $\ell$, homogeneous degree
$\ell d$ polynomials.

Next, assume for $D< d$ and any $G$ that $\mathcal{B}_{G,D,\ell}$ spans the
set of NC, $\ell$-harmonic, fully degree $\ell$, homogeneous degree
$\ell D$ polynomials in the variables $x_1, \ldots, x_G$.
Suppose $p \in \FFx$ is $\ell$-harmonic, fully degree $\ell$, and homogeneous
of degree $\ell d$.  Since $p$ is fully degree $\ell$, $\deg_g(p)$ is
either $0$ or $\ell$.  Consider both of these cases.

\begin{enumerate}
 \item[Case 1:]  Suppose $\deg_g(p) = 0$.  By Lemma \ref{lemma:fully}, it
suffices to
consider the case where $p$ is of the form
\[ p = \sum_{\tau \in S_{g-1}} A_{\tau} x_{\tau(1)}^{\ell} \ldots
x_{\tau(d)}^{\ell}.\]
Express $p$ as
\[ p = \sum_{i=1}^{g-1} x_i^{\ell} p_i,\]
where the $p_i$ are defined by
\[p_i = \sum_{\tau(1) = i} A_{\tau} x_{\tau(2)}^{\ell} \ldots
x_{\tau(d)}^{\ell}.\]
Note that $\deg_i(p_i) = 0$ for each $i$, and hence $x_i^{\ell} p_i$ is an
independent product.
Taking the $\ell$-Laplacian gives
\begin{equation}
\label{eq:lapEllofP}
  \lap_{\ell}[p,h] = \ell! \sum_{i=1}^{g-1} h^{\ell} p_i +
\sum_{i=1}^{g-1} x_i^{\ell} \lap_{\ell}[p_i,h] = 0.
\end{equation}
For each $i$, the terms of (\ref{eq:lapEllofP}) whose leftmost variables are
$x_i^{\ell}$ are $x_i^{\ell}\lap_{\ell}[p_i,h]$, which must therefore equal $0$.
Therefore each $p_i$ is $\ell$-harmonic, fully degree $\ell$, and
homogeneous
of degree $\ell (d-1)$ and degree $0$ in $x_g$.  Express each $p_i$ as
\[ p_i = \sum_{k=1}^{M} \beta_{i,k} b_k,\]
where $b_1, \ldots, b_k$ are in $\mathcal{B}_{g-1,d-1,\ell}$ and the
$\beta_{i,k}$ are some scalars. Since
$\deg_i(p_i) = 0$, choose the $\beta_{i,k}$ so that if $\deg_i(b_k) > 0$ then
$\beta_{i,k} = 0$.
Also note that the terms of (\ref{eq:lapEllofP}) with $h^{\ell}$ as the
leftmost variables are
\[ \sum_{i=1}^{g-1}  h^{\ell} p_i \]
which must therefore equal zero.
This implies that
\[ \sum_{i=1}^{g-1} x_g^{\ell} p_i  = 0.\]
Therefore
\begin{align}
 \notag
p&= p- \sum_{i=1}^{g-1} x_g^{\ell} p_i \\
\label{eq:sumOfBgdl}
&= \sum_{i=1}^{g-1} \sum_{k=1}^M \beta_{i,k} (x_i^{\ell} - x_g^{\ell}) b_{k}.
\end{align}
Since $\beta_{i,k}$ is nonzero only if $\deg_i(b_k) = 0$, and since
$\deg_g(b_k) = 0$ for all $k$, we see that (\ref{eq:sumOfBgdl}) is in the span
of $\mathcal{B}_{g,d,\ell}$.

 \item[Case 2:]
Suppose $\deg_g(p) = \ell$. Express $p$ as in Proposition
\ref{thm:neighbor} to get
\[ p = \sum_{j=1}^N \sigma_j\left[ x_g^{\ell} p_j \right] + r,\]
where $\deg_g(p_j) = 0$ for each $j$ and $\deg_g(r) < \ell$.  Since $p$ is a
sum of fully degree $\ell$ monomials, we can take $p_j$ and $r$ to be a sum of
fully degree $\ell$ terms.  In particular, this implies that since $\deg_g(r) <
\ell$, that $\deg_g(r) = 0$.  By Proposition \ref{thm:neighbor}, since $p$ is
$\ell$-harmonic, so is each $p_j$.  Since each $p_j$ is $\ell$-harmonic and
degree $0$ in $x_g$, by induction, we may express each $p_j$ as a sum
\[ p_j = \sum_{k=1}^{M} \beta_{j,k} b_k,\]
where $b_1, \ldots, b_k$ are in $\mathcal{B}_{g-1,d-1,\ell}$.
Each element of $\mathcal{B}_{g-1,d-1,\ell}$ has only $2(d-1)$ variables.
Since $g \geq 2d$, for each $b_k$ there exists a variable $x_{\nu_k}$, with $1
\leq \nu_k < g$ such that $\deg_{\nu_k}(b_k) = 0$.  Therefore the polynomial
\[ (x_g^{\ell} - x_{\nu_k}^{\ell})b_k\]
is in $\mathcal{B}_{g,d,\ell}$.
  Therefore the polynomial
\begin{align}
\notag
 p - \sum_{j=1}^N \sigma_j\left[ \sum_{k=1}^M
\beta_{j,k} (x_g^{\ell}-x_{\nu_j}^{\ell}) b_k \right] &=
\sum_{j=1}^N
\sigma_j\left[ \sum_{k=1}^M
\beta_{j,k} x_g^{\ell} b_k \right] + r\\
\notag & -\sum_{j=1}^N \sigma_j\left[
\sum_{k=1}^M
\beta_{j,k} (x_g^{\ell}-x_{\nu_k}^{\ell}) b_k \right] \\
\label{eq:reducedEllHarm}
&=\sum_{j=1}^N \sigma_j\left[ \sum_{k=1}^M
\beta_{j,k} x_{\nu_k}^{\ell} b_k \right] + r
\end{align}
is $\ell$-harmonic and is in the span of $\mathcal{B}_{g,d,\ell}$ if and only
if $p$ is. Also note that (\ref{eq:reducedEllHarm}) has degree $0$ in $x_g$ and
is fully degree $\ell$ and homogeneous of degree $\ell d$.  By Case 1,
(\ref{eq:reducedEllHarm}) is in the span of $\mathcal{B}_{g,d,\ell}$, which
implies that $p$ is in the span of  $\mathcal{B}_{g,d,\ell}$.
\end{enumerate}

\end{proof}

\subsection{A Classification of the Set of Noncommutative $\ell$-Harmonic Polynomials}
\label{sect:genclassification}

The following proposition gives a classification of the set of NC
$\ell$-harmonic polynomials.  From this proposition will follow Theorem
\ref{thm:main}.

\begin{prop}
\label{prop:main}
Let $x = (x_1, \ldots, x_g)$.
Let
$p \in \FFx$ be a NC, homogeneous degree $d$ polynomial, let $\ell \geq 2$ be a positive integer,
and suppose $g \geq d$.
The polynomial $p$ 
is $\ell$-harmonic
if and only if
it is a finite sum
\begin{align}
\label{eq:mainprop}
p = \sum_{k}^{finite} \sigma_k [p_{k,1} \ldots p_{k,\num_k}]
\end{align}
for some $\sigma_k \in S_d$, possibly repeating,
and for some NC, symmetrized, $\ell$-harmonic, homogeneous polynomials
 $p_{k,i} \in \FF\langle x_1, \ldots, x_g\rangle$
 such that for each $k$, the product $p_{k,1} \ldots p_{k,\num_k}$
is homogeneous of degree $\ell d$ and is an independent product.

If desired, one may additionally impose the following technical conditions on the $p_{k, i}$
in (\ref{eq:mainprop}).
\begin{enumerate}
\item For each $i,j,k$,
\[\deg_j\left(p_{k, i}\right) \leq \max\{\deg_j(p), \ell\}.\]

\item
Each polynomial
$p_{k,1} \ldots p_{k, M_k}$
depends on at most $d$ variables.
\end{enumerate}
\end{prop}

\begin{proof}
Consider a polynomial $p \in \FFx$.  Proceed by induction on $\deg(p)$
to shoe that if $p$ is $\ell$-harmonic then it is of the form of
(\ref{eq:mainprop}).

  For $d \leq \ell$, the elements of the basis
described in Proposition
\ref{prop:deg2harm} are also all of the form of (\ref{eq:mainprop}).  Note
that a polynomial $x_i^k$, with $k < \ell$, is a NC, symmetrized,
$\ell$-harmonic polynomial.

Suppose now that the proposition holds for all degrees less than some
$d > \ell$.
Let $p \in \FFx$ be a NC, $\ell$-harmonic, homogeneous degree $d$ polynomial.
The following steps will prove algorithmically that $p$ is of the form of
(\ref{eq:mainprop}).

\begin{enumerate}
 \item First we reduce the problem of proving $p$ is of the form of
(\ref{eq:mainprop}) to proving that some other polynomial $r$, with $\deg_i(r)
\leq \ell$ for all $i$, is of the form of (\ref{eq:mainprop}).

\begin{enumerate}
\item \label{item:step1} Let $d_* := d$.

\item Let $i:=1$.
\label{step:A0}

\item
\label{step:A1}
At this step, the following conditions hold:
 \begin{align}
 d_* & > \ell \notag\\
 \deg_j(p) & \leq d_*,\ j \geq i\notag\\
 \deg_j(p) & < d_*,\ j < i \notag.
 \end{align}
 Choose $\sigma_1, \ldots, \sigma_N$ which satisfy the conclusions of Lemma \ref{lemma:neighbor}
and express $p$ as
\[p = \sum_{j=1}^N \sigma_j \left[x_i^{d_*}q_j\right] + r,\]
with $\deg_i(r) < d_*$, $\deg_i(q_j)=0$ for each $j$,
and, by Proposition \ref{thm:neighbor}, each $q_j$ being $\ell$-harmonic.
By the induction hypothesis, since $\deg(q_j) = d - d_* < d$, each $q_j$ may be
expressed as a finite sum
\[q_j =  \sum_{k} q_{j,k},\]
where each $q_{j,k}$ is of the form (\ref{eq:mainprop}).  In particular,
each $q_{j,k}$ depends on at most $d-d_*$ variables.
By the division algorithm, let $d_* = \ell Q + r$, where $r$ and $Q$ are integers with
$0 \leq r < \ell$.
For each $q_{j,k}$, choose $Q$ distinct
variables $x_{a_1}, \ldots, x_{a_{Q}}$, each not
equal to $x_i$, such that $\deg_{a_m}(q_{j,k}) = 0$
for each $a_m$.   Define $s_{j,k}$ to
be
\[ s_{j,k} = \frac{(-1)^q (\ell Q + r)!}{Q!
\ell^Q}\Symm\left[\Hm\left[x_i^{r}x_{a_1}^{\ell} \ldots x_{a_Q}^{\ell}, x_i,
\ell\right]\right].\]
This $s_{j,k}$ is, by construction, a symmetrized $\ell$-harmonic polynomial.
By Lemma \ref{prop:techOnHsubEll}, $s_{j,k}$
has $\deg_i(s_{j,k}) = d_*$ with $x_i^{d_*}$ having a coefficient of $1$.
Further note that
since $s_{j,k}$ and $q_{j,k}$ don't have any variables in common,
$s_{j,k}q_{j,k}$ is an independent product of a symmetrized $\ell$-harmonic and
a polynomial of the form (\ref{eq:mainprop}).  Therefore $s_{j,k}q_{j,k}$ is
itself of the form  (\ref{eq:mainprop}).

The polynomial $p$ is equal to
\begin{align}
p &= \sum_{j=1}^N \sum_{k} \sigma_{j}\left[x_i^{d_*}q_{j,k}\right]\\
\notag &= \sum_{j=1}^N \sum_{k} \sigma_{j}\left[\left(x_i^{d_*} - s_{j,k}\right)q_{j,k}\right] + \sum_{j=1}^N \sum_{k} \sigma_{j}\left[s_{j,k}q_{j,k}\right].
\end{align}
Since the polynomial
\[ \sum_{j=1}^N \sum_{k} \sigma_{j}\left[s_{j,k}q_{j,k}\right]\]
is of the form of (\ref{eq:mainprop}),
it follows that $p$ also is if and only if
\[\sum_{j=1}^N \sum_{k} \sigma_{j}\left[\left(x_i^{d_*} - s_{j,k}\right)q_{j,k}\right]\]
is.
Redefine $p$ to be
\[p :=  \sum_{j=1}^N \sum_{k} \sigma_{j}\left[\left(x_i^{d_*} - s_{j,k}\right)q_{j,k}\right].\]
This new $p$ has $\deg_i(p) < d_*$.

\item If $i < g$, redefine $i$ to be
$i:= i+1$. Go to step \ref{step:A1}.

\item If $i = g$ and $d_* > \ell + 1$, redefine $d_* := d_* - 1$ and go to step \ref{step:A0}.

\end{enumerate}

\item
\label{step:A2}
At this step, $\deg_j(p) \leq \ell$ for all $j$.  We will reduce the problem of
showing that $p$ is of the form of (\ref{eq:mainprop}) to showing that some
fully degree $\ell$ polynomial is of the form of (\ref{eq:mainprop}).

\begin{enumerate}
\item Let $i := 1$.

\item At this step, for $j < i$, the terms of $p$ have either degree $0$
or degree $\ell$ in each $x_j$.
Express $p$ as in Proposition \ref{prop:oddsandevens} as
    \[p = \sum_{k=0}^{\ell -1} p_{[k]},\]
    where the terms of each $p_{[k]}$ have degree in $x_i$ equivalent to $k$ modulo $\ell$.
    Since $\deg_i(p) \leq \ell$,
    when $1 \leq k < \ell$, the polynomial $p_{[k]}$ must be homogeneous of degree $k$ in $x_i$. Further,
    by Proposition \ref{prop:oddsandevens},
    each $p_{[k]}$ is $\ell$-harmonic.

For $k = 1, \ldots, \ell -1$, choose $\sigma_1^{(k)}, \ldots, \sigma_N^{(k)}$ which satisfy the conclusions Lemma \ref{lemma:neighbor} and express $p_{[k]}$ as
        \[p_{[k]} = \sum_{j=1}^N
        \sigma_{j}^{(k)}\left[x_i^k f_{k,j}\right], \]
        where, by Proposition \ref{thm:neighbor}, each $f_{k,j}$ is $\ell$-harmonic and degree $0$ in $x_i$.
        By the induction assumption, each $f_{k,j}$
        is of the form (\ref{eq:mainprop}).  Further,
        since $k < \ell$, the polynomial $x_i^k$ is symmetrized and $\ell$-harmonic,
        which implies that $p_{[k]}$
        is of the form of (\ref{eq:mainprop}).
Therefore
        $p$ is of the form of (\ref{eq:mainprop}) if and only if $p_{[0]}$ is.
        Redefine $p := p_{[0]}$.

        \item If $i < g$, redefine $i:= i + 1$ and go to
            step \ref{step:A2}.
\end{enumerate}

            \item At this step,
            $p$ is a sum of terms
            with degree in $x_j$
            equal to $0$ or $\ell$
            for each $j$.
            By definition, $p$ is fully degree $\ell$.
            By Proposition
            \ref{prop:keytechnicalstep},
            $p$ is of the form of
            (\ref{eq:mainprop}).
\end{enumerate}
\end{proof}

\subsubsection{Proof of Theorem \ref{thm:main} (\ref{item:middledegmain})}

\begin{proof}
For $\ell = 1$, the result follows by Proposition \ref{prop:ellequalsone}.
Otherwise, suppose $\ell \geq 2$.
It must be shown that given a homogeneous degree $d$,
$\ell$-harmonic
polynomial $p \in \CCx$, with $g \geq d$, then
$p$ is a sum of polynomials of the form
\begin{equation}
 \label{eq:mainthmMod}
 p = \prod_{i=1}^d \sum_{j=1}^g a_{ij}x_j,
\end{equation}
where for each $1 \leq k_1 < k_2 < \ldots < k_{\ell} \leq d$
we have
\[ \sum_{j=1}^g a_{k_1 j} \ldots a_{k_{\ell} j} = 0.\]
Consider a few cases.

\begin{enumerate}
 \item\label{item:case1} Suppose $p$ is of the form
\[ p = (b_1 x_1 + \ldots + b_gx_g)^d,\]
where $d \geq \ell$ and $b_1^{\ell} + \ldots + b_{g}^{\ell} = 0$.
Then
\[ p = \prod_{i=1}^d \sum_{j=1}^g a_{ij}x_j,\]
with $a_{ij} = b_j$ for all $i,j$ shows that $p$ is of the form of
(\ref{eq:mainthmMod}) since for each $1 \leq k_1 < k_2 < \ldots < k_{\ell} \leq
d$ we have
\[ \sum_{j=1}^g a_{k_1 j} \ldots a_{k_{\ell} j} = \sum_{j=1}^g b_j^{\ell} = 0.\]
\item
\label{item:degEllHigherGen}
If $p$ is a symmetrized $\ell$-harmonic of degree $d \geq \ell$, then
by (\ref{eq:prevprovedcomm}) and by Theorem \ref{thm:parta}, $p$ is a sum of
polynomials of the form
\[ \Symm[(b_1 x_1 + \ldots + b_gx_g)^d] = (b_1 x_1 + \ldots + b_gx_g)^d \in
\CCx.\] By Case \ref{item:case1}, this is a sum of polynomials of the form
(\ref{eq:mainthmMod}).
\item\label{item:oneCanConsider}
A monomial $m = x_{b_1} \ldots x_{b_d}$, with $d < \ell$ is of the form
\[  m = x_{b_1} \ldots x_{b_d} = \prod_{i=1}^d \sum_{j=1}^g a_{ij}x_j\]
by setting $a_{ij} = 1$ if $j = b_i$ and $0$
otherwise.
If $p \in \CCx$ is homogeneous of degree $d < \ell$, 
then this implies that $p$
is a sum
of polynomials of the form
\[  \prod_{i=1}^d \sum_{j=1}^g a_{ij}x_j\]
for some $a_{ij}$.
In this case, since $d < \ell$, there exists no $k_1, \ldots, k_{\ell}$
with $1 \leq k_1 < \ldots < k_{\ell} \leq d$, so therefore a homogeneous degree
$d$ polynomial is of the form (\ref{eq:mainthmMod}) automatically.
\item\label{item:prodOfBig}
Suppose $p_1, p_2\in \CCx$ are homogeneous of degrees $d_1$ and $d_2$
respectively are such that $p_1p_2$ is an independent product, and suppose
\[  p_1 = \prod_{i=1}^{d_1} \sum_{j=1}^g a_{ij}x_j\]
and
\[  p_2 = \prod_{i=1}^{d_2} \sum_{j=1}^g b_{ij}x_j\]
are of the form of (\ref{eq:mainthmMod}).
For $1 \leq i \leq d_2$, define $a_{(i+d_1)j} = b_{ij}$
so that
\[p_1p_2 = \prod_{i=1}^{d_1 + d_2} \sum_{j=1}^g a_{ij}x_j\]

Supppose $1 \leq k_1 < k_2 < \ldots < k_{\ell} \leq d_1 + d_2$.
If $k_{\ell} \leq d_1$, then by assumption
\[ \sum_{j=1}^g a_{k_1 j} \ldots a_{k_{\ell} j} = 0.\]
If $k_1 > d_1$, then
\[ \sum_{j=1}^g a_{k_1 j} \ldots a_{k_{\ell} j} = \sum_{j=1}^g b_{(k_1-d_1) j}
\ldots b_{(k_{\ell}-d_{1}) j} = 0.\]
If $k_1 \leq d_1 < k_{\ell}$, then
consider a product $a_{k_1 j}a_{k_d j}$.  If $a_{k_1 j} \neq 0$, then
$\deg_j(p_1) \neq 0$, so since $p_1p_2$ is an independent product, $\deg_j(p_2)
= 0$ and $a_{k_d j} = b_{(k_d-d_1)j} = 0$.
Similarly, if $a_{k_d j} = b_{(k_d-d_1)j} \neq 0$, then $a_{k_1 j} = 0$.
Therefore, in all cases
\[ \sum_{j=1}^g a_{k_1 j} \ldots a_{k_{\ell} j} = 0.\]
\item\label{item:permOfThing}
If $\sigma \in S_d$ and
$p$ is of the form of (\ref{eq:mainthmMod}), then Proposition
\ref{prop:permOfProdLin} implies that $\sigma[p]$ is as well.
\item
If $p_1, p_2$ are sums of polynomials of the form of (\ref{eq:mainthmMod}), then
by linearity and by Case \ref{item:prodOfBig}, so is $p_1p_2$.
\end{enumerate}

Now, given these cases, by Proposition \ref{prop:main}, every homogeneous
degree $d$, $\ell$-harmonic polynomial $p \in \CCx$ is a sum of polynomials
which are
permutations of independent products
\[ p_1 \ldots p_k\]
where each $p_i$ is a symmetrized $\ell$-harmonic polynomial.
By linearity and by Case \ref{item:permOfThing}, it suffices to consider $p =
p_1 \ldots p_k$.  First, if $\deg(p_i) \geq \ell$, then $p_i$ is of the form of
(\ref{eq:mainthmMod}) by Case \ref{item:degEllHigherGen}.
If $\deg(p_i) < \ell$ then by Case \ref{item:oneCanConsider}, $p_i$ is of the
form of
(\ref{eq:mainthmMod}).  By repeated application of Case \ref{item:prodOfBig},
the product $p_1 \ldots p_k$ is of the form of (\ref{eq:mainthmMod}),
which proves Theorem \ref{thm:main} (\ref{item:middledegmain}).
\end{proof}

\subsection{The Case where $d > g$}
Theorem \ref{thm:main} (\ref{item:highdegmain}) follows as a corollary.

\begin{proof}[Proof of Theorem \ref{thm:main} (\ref{item:highdegmain})]
Suppose $p \in \CC\langle x_1, \ldots, x_g \rangle$ is $\ell$-harmonic, and
suppose $d > g$.
The space $\CC \langle x_1, \ldots, x_g \rangle$ may be viewed as a subspace of
$\CC\langle x_1, \ldots, x_d \rangle$.
Therefore, $p \in \CC\langle x_1, \ldots, x_d \rangle$ is a homogeneous degree
$d$ polynomial, and so Theorem \ref{thm:main} (\ref{item:middledegmain})
applies as long as $p$ is $\ell$-harmonic.

We see
\begin{align}
\lap_{\ell}[p,h] &=
\dird\left[p, \sum_{i=1}^d x_i^{\ell}, h \right] \notag \\
\label{eq:lookattwothings}
&= \dird\left[p, \sum_{i=1}^g x_i^{\ell}, h \right] + \dird\left[p, \sum_{i=g+1}^d x_i^{\ell},h\right].
\end{align}
By assumption, the first term of (\ref{eq:lookattwothings}) is zero.  Also,
since $\deg_i(p) = 0$ for each $i > g$, it follows that the second term of
(\ref{eq:lookattwothings}) is zero.  Therefore $p \in \CC\langle x_1, \ldots,
x_d \rangle$ is $\ell$-harmonic.
The result follows from Theorem \ref{thm:main} (\ref{item:middledegmain}).
\end{proof}

\subsection{Other Selected Cases of Non-Commutative Differential Equations}

\subsubsection{Generalization of Proposition \ref{prop:main} to $q =
\sum_{i=a+1}^g x_i^{\ell}$}

\begin{prop}
\label{prop:genWithZeroes}
Let $p \in \FFx$ be a homogeneous degree $d$ polynomial. Then $p$ satisfies
\begin{equation}
\label{eq:genWithZeroes}
\dird\left[p, \sum_{i=a+1}^{g} x_i^{\ell}, h\right] = 0,
\end{equation}
where $0<a < g$, if and only if $p$ is a finite sum of the form
\begin{equation}
\label{eq:formOfGenZeroes}
\sum_{k} \sigma_k\left[p_{1,k}(x_1, \ldots, x_a)p_{2,k}(x_{a+1}, \ldots, x_g)\right]
\end{equation}
where each $p_{1,k}p_{2,k}$ is homogeneous of degree $d$, where the $\sigma_k \in S_d$ are some permutations, possibly repeating, and where each $p_{2,k}(x_{a+1}, \ldots, x_g)$ is $\ell$-harmonic.
\end{prop}

\begin{proof}
Let $p$ be equal to
\[p = \sum_{i=0}^d p_i, \]
where each $p_i$ is either $0$ or homogeneous of degree $i$ in $x_1$.
For each $i$, fix permutations $\sigma_1^{(i)}, \ldots, \sigma_{N_i}^{(i)}$ as in Lemma \ref{lemma:neighbor} and express each $p_i$ as
\[p_i = \sum_{j=1}^{N_i} \sigma_j^{(i)}\left[x_1^i p_j^{(i)} \right], \]
where each $p_j^{(i)}$ is degree $0$ in $x_1$.
If $\dird\left[x_1^i, \displaystyle\sum_{k=a+1}^g x_k^{\ell}, h\right]  = 0$, then
\begin{align}
\notag \dird\left[p, \sum_{k=a+1}^g x_k^{\ell}, h\right] &= \sum_{i=1}^d \sum_{j=1}^{N_i} \sigma_j^{(i)} \left[\dird\left[x_1^ip_j^{(i)}, \sum_{k=a+1}^g x_k^{\ell}, h \right]\right]\\
\notag &=\sum_{i=1}^d \sum_{j=1}^{N_i} \sigma_j^{(i)} \left[x_1^i\dird\left[p_j^{(i)},\sum_{k=a+1}^g x_k^{\ell}, h\right] \right] = 0.
\end{align}
By Proposition \ref{thm:neighbor}, since $\dird[p_i, \sum_{k=a+1}^g x_k^{\ell},
h]$ can be uniquely represented with respect to the $\sigma_j^{(i)}$
chosen, this implies that for each $(i,j)$,
\[\dird\left[p_j^{(i)},
\sum_{k=a+1}^g x_k^{\ell}, h\right] = 0.\]
Repeat this process to get
\[p = \sum_{\tau \in S_d} \sum_{\gamma} \tau\left[x_1^{\gamma_1} \ldots x_a^{\gamma_a}p_{\tau,\gamma} \right], \]
where each $\gamma = (\gamma_1, \ldots, \gamma_a)$ is a $a$-tuple of nonnegative integers and where each
$p_{\tau,\gamma}$ doesn't depend on $x_1, \ldots, x_a$ and satisfies
\[\dird\left[p_{\tau,\gamma}, \sum_{k=a+1}^g x_k^{\ell}, h\right] = 0.\]
Further, since $p_{\tau,\gamma}$ has $\deg_i(p_{\tau,\gamma}) = 0$ for $i \leq
a$,
\[\dird\left[p_{\tau,\gamma}, \sum_{k=a+1}^g x_k^{\ell}, h\right] = \dird\left[p_{\tau,\gamma}, \sum_{k=1}^g x_k^{\ell}, h\right] = \lap_{\ell}\left[p_{\tau,\gamma}, h\right] = 0. \]
\end{proof}

\subsubsection{Change of Variables}

For other many other partial differential equations of the form $q(\nabla)p(x) = 0$,
there exists an invertible linear transformation $A$ such that
$q(Ax) = x_1^{\ell} + \ldots + x_g^{\ell}.$

\begin{prop}Let $B = (b_{ij})_{i,j}
\in \FF^{g \times g}$ be an invertible matrix, and let
$q \in \FF[x]$ be equal to
\[q = \sum_{i=1}^g \left(\sum_{j=1}^g b_{ij}x_j\right)^{\ell}.\]
A degree
$d$ homogeneous polynomial $p \in \FFx$
is a solution to
\begin{align}
\dird[p,q,h] = 0
\end{align}
if and only if $p\left(B^Tx\right)$ is $\ell$-harmonic.
\end{prop}

\begin{proof}First of all, when $Bx$ is substituted for $x$ in the polynomial
$\displaystyle\sum_{i=1}^g x_i^{\ell}$, one gets
\[\sum_{i=1}^g \left(\sum_{j=1}^g b_{ij}x_j\right)^{\ell} = q(x). \]
Therefore,
\[q(B^{-1}x) =  \sum_{i=1}^g x_i^{\ell}.\]
By Corollary \ref{cor:chainrule},
\begin{align}
\lap_{\ell}\left[p\left(B^T\right),h\right] &=
\dird\left[p\left(\left(B^T\right)x\right), q\left(B^{-1}x\right), h\right]\\
 \notag &= \dird\left[p, q\left(B^{-1}Bx\right), h\right]\left(B^Tx\right)
\\ \notag &=\dird[p,q,h]\left(B^Tx\right).
\end{align}
Since $B$ is invertible,
$\dird[p,q,h]\left(B^Tx\right)$ is identically zero
if and only if $\dird[p,q,h]$ is. Therefore, $p\left(B^Th\right)$ is $\ell$-harmonic
if and only if $\dird[p,q,h] = 0$.
\end{proof}

\section{Noncommutative Polynomials in Nonsymmetric Variables}
\label{sect:ncvars}

This section considers NC polynomials in nonsymmetric variables.
A NC polynomial $p$ in \textbf{nonsymmetric variables}\index{variable!nonsymmetric} is a
NC polynomial with the additional condition that $x_i$ and $x_i^T$ are
not equal. The set of NC polynomials in nonsymmetric variables
with coefficients in $\FF$ is denoted as $\FFNx$\index{Faxs@$\FFNx$}.

\subsection{Definitions}

\begin{definition}
Consider the set $\{1, T\}^d$.\index{1Td@$\{1,T\}^d$}.  If $\alpha = (a_1,
\ldots, a_d) \in \{1, T\}^d$ and $m = x_{b_1} \ldots x_{b_d} \in \FFx$ is a
monomial, let $m^{\alpha} \in \FFNx$ be defined to be the monomial
\begin{equation}
\label{eq:defOfMAlpha}
 m^{\alpha} := x_{b_1}^{a_1} x_{b_2}^{a_2} \ldots x_{b_d}^{a_d},
\end{equation}
where for any given $x_i$, the variable $x_i^1 := x_i$ and $x_i^T$ is the
transpose of $x_i$.

If $p = \sum_{m \in \cM} A_m m$ is a general homogeneous degree $d$ polynomial
in $\FFx$, and $\alpha = \in \{1, T\}^d$, define $p^{\alpha}$
\index{palpha@$p^{\alpha}\ \left(\alpha \in \{1,T\}^d\right)$} to be
\[ p^{\alpha}:= \sum_{m \in \cM} A_m m^{\alpha}.\]
\end{definition}

\begin{exa}
 Let $\alpha = (1, T, 1, T)$ and let $p = x_1x_1x_1x_1 + 3 x_2x_1x_1x_2$.
Then $p^{\alpha}$ is equal to
\[ p^{\alpha} = x_1x_1^Tx_1x_1^T + 3 x_2x_1^Tx_1x_2^T.\]
\end{exa}

\begin{definition}Let $\alpha = (a_1, \ldots, a_d) \in \{1, T\}^d$.
Define the set $\NS_{\alpha} \subseteq \FFNx$ to be\index{NSalpha@$\NS_{\alpha}$}
\[\NS_{\alpha} := \mathrm{span}\{ m^{\alpha}:\ m \in \mathcal{M},\ |m| = d\}. \]
Let $\Proj_{\alpha}$ \index{Palpha@$\Proj_{\alpha}$} be defined to be the
projection map from $\FFNx$ onto $\NS_{\alpha}$.
\end{definition}

\begin{exa}
 Let $p \in \FFNx$ be equal to
\[p = x_1x_1x_2^Tx_1 -7 x_2x_2x_1^Tx_1 +2x_1^Tx_2x_2x_2 -2 x_1x_1x_1x_1 +
x_1x_1^Tx_1^Tx_2. \]
Let $\alpha_1, \ldots, \alpha_4 \in \{1, T\}^4$ be defined by
\begin{align}
\notag
\alpha_1 & = (1, 1, T, 1), & \alpha_2 &= (T,1,1,1), &\alpha_3 &=
(1, 1, 1, 1), & \alpha_4 &= (1, T, T, 1). \end{align}
Then
\begin{align}
\notag \Proj_{\alpha_1}[p] &= x_1x_1x_2^Tx_1 -7 x_2x_2x_1^Tx_1,&
\Proj_{\alpha_2}[p]& = 2x_1^Tx_2x_2x_2, \\
\notag \Proj_{\alpha_3}[p] &=-2 x_1x_1x_1x_1,&
\Proj_{\alpha_4}[p] &= x_1x_1^Tx_1^Tx_2. \end{align}
\end{exa}

\begin{exa}
Let $p = x_1x_2x_1x_2 \in \FFx$ and let $\alpha = (T, T, 1, 1), \beta = (1, T,
T, T)$. Then
\[p^{\alpha} = x_1x_2x_1^Tx_2^T\] and
\[p^{\beta} = x_1^Tx_2x_1x_2.\]
Notice that
\[\Proj_{\alpha}[p^{\beta}] = 0 = \Proj_{\beta}[p^{\alpha}],\]
since $p^{\alpha}$ and $p^{\beta}$ have transposes on different entries.\qed
\end{exa}

\begin{definition}
Let $\bigoplus$ denote the direct sum of vector spaces.  That is, if $V_1, \ldots, V_N$ are vector spaces, then\index{direct sum}\index{$\oplus$|see{direct sum}}
\[\bigoplus_{i=1}^N V_i = \sum_{i=1}^N V_i, \]
with the additional condition that each element of $\bigoplus_{i=1}^N V_i$ can be uniquely represented as
a sum $v_1 + \ldots + v_N$, where each $v_i \in V_i$.
\end{definition}

\begin{prop}
\label{prop:directsumalpha}
 \begin{align}
 \FFNx = \bigoplus_{\alpha \in \{1,T\}^d} \NS_{\alpha}.
 \end{align}
 \end{prop}

 \begin{proof}Straightforward
 \end{proof}

 Define a third map $\mathrm{Sx}$ \index{Sx@$\mathrm{Sx}$} from $\FFNx$
 to $\FFx$
 by \begin{equation}
 \mathrm{Sx}[p(x, x^T)] = p(x, x).
 \end{equation}
That is, $\mathrm{Sx}$ sets $x_i^T$ equal to $x_i$ in $p(x, x^T)$ so that $p$
becomes an element of $\FFx$.

\begin{exa}
 Let $p \in \FFNx$ be equal to
\[ p = x_1^Tx_2x_2^Tx_1 - x_1x_2x_2x_1^T + 3x_1^Tx_1x_1x_1.\]
Then $\mathrm{Sx}[p] = 3x_1^4$.

Further, if $\alpha = (T,1,T,1)$, then
$\Proj_{\alpha}[p]$ is equal to
\[ \Proj_{\alpha}[p] = x_1^Tx_2x_2^Tx_1\]
and
\[ \mathrm{Sx}[\Proj_{\alpha}[p]]^{\alpha} = x_1^Tx_2x_2^Tx_1 =
\Proj_{\alpha}[p].\]

In general, for any $\alpha \in \{1,T\}^d$ and homogeneous degree $d$
polynomial $p \in \FFNx$ we have
\[\mathrm{Sx}[\Proj_{\alpha}[p]]^{\alpha}  =
\Proj_{\alpha}[p] \]
since $\Proj_{\alpha}[p] $ only has terms with transposes corresponding to the
transposes of $\alpha$.
\end{exa}

\subsection{Differentiation}

Differentiation in $\FFNx$ is similar to differentiation in the case of symmetric variables.
If $p \in \FFNx$, define $\dird[p(x_1,\ldots,x_g),x_i,h]$
to be
\begin{equation}
\dird[p(x_1,\ldots,x_g), x_i,h]:= \frac{d}{dt}[p(x_1, \ldots, (x_i+th),
\ldots, x_g)]_{|_{t=0}}.
\end{equation}\index{directional derviative}

\begin{exa}Take as an example $p(x_1, x_2,x_3) = x_1^Tx_1x_3 - x_2^Tx_2x_3$. One
sees,
\begin{align}\dird[p(x_1, x_2,x_3), x_1, h] = \frac{d}{dt}((x_1 + th)^T(x_1 +
th)x_3 - x_2^Tx_2x_3)
= h^Tx_1x_3 + x_1^Thx_3
\end{align}
In this specific case, also note that
\[\lap[p(x_1, x_2,x_3), h] = 2h^Thx_3 - 2h^Thx_3 = 0\] meaning $p(x_1, x_2,x_3)$
is harmonic.\qed
\end{exa}

As this example shows, the directional derivative of $p$ with respect to $x_i$
in the direction $h$ is the sum of all possible terms produced from $p$ by replacing
either one instance of $x_i$ with $h$ or one instance of $x_i^T$
with $h^T$.
Given $\alpha$, the previous definitions of $p^{\alpha}$, $P_{\alpha}$,
$\NS_{\alpha}$ and $\mathrm{Sx}$ are extended to the space $\FF\langle x, h
\rangle$ by treating $h$ as if it were another variable $x_{g+1}$.
This gives the following proposition.

 \begin{prop}Let $\alpha \in \{1,T\}^d$, let $p \in \FFNx$, and let $q \in \FF[x]$. Then
 \begin{equation}
\Proj_{\alpha}[\dird[p, q, h]] = \dird[\Proj_{\alpha}[p],q, h].
\end{equation}
 \end{prop}

 \begin{proof}
 As discussed, $\dird[\Proj_{\alpha}[p], q, h]$ has transposes on the same entries as $\Proj_{\alpha}[p]$ does.  Therefore
 \[\Proj_{\alpha}[\dird[p,q,h] ] = \Proj_{\alpha}\left[ \sum_{\beta \in
\{1,T\}^d} \dird[\Proj_{\beta}[p], q, h]\right] =
\Proj_{\alpha}\dird[\Proj_{\alpha}[p], q, h] = \dird[\Proj_{\alpha}[p],q,h].\]
 \end{proof}

\subsection{Partial Differential Equations in Nonsymmetric Variables}

The results of $\S$ \ref{sect:classification} can be extended to $\FFNx$ by the 
following proposition.  In particular, there is an extension of Theorem
\ref{thm:main} to $\FFNx$.

\begin{prop}
\label{prop:symm2NS}
Let $p \in \FFNx$ be a NC, homogeneous
degree $d$ polynomial, and let $q \in \FF[x]$ be a commutative, homogeneous, degree $\ell$
polynomial.
\begin{enumerate}
\item \label{item:sumOfAlphaSoln}
The polynomial $p$ satisfies
$\dird[p, q, h] = 0$ if and only if
$\dird[\Proj_{\alpha}[p], q, h] = 0$ for each $\alpha \in \{1,T\}^d$.

\item \label{item:alphaPartIsSymmSoln}
For a given $\alpha \in \{1,T\}^d$, the polynomial $p$ satisfies
$\dird[\Proj_{\alpha}[p], q, h] = 0$ if and only if $\dird[S[P_{\alpha}[p]],q,h] = 0$.
\end{enumerate}
\end{prop}

\begin{proof}
\begin{enumerate}
\item
Given $\alpha \in \{1,T\}^d$, each derivative of the form
\[\dird[\Proj_{\alpha}[p], q, h] = \Proj_{\alpha}[\dird[p, q, h]]\]
lies in the space
$\NS_{\alpha}.$
Since $p$ is equal to
\[p = \sum_{\alpha \in \{1,T\}^d} \Proj_{\alpha}[p],\]
by Proposition \ref{prop:directsumalpha},
the condition that
$\dird[p, q, h] = 0$
is equivalent to
$\dird[\Proj_{\alpha}[p], q, h] = 0$ for all $\alpha \in \{1,T\}^d$.

\item It is straightforward to check that
 \[\dird[\mathrm{Sx}[p], q, h] = \mathrm{Sx}[\dird[p, q, h]],\]
 and that on the set $\NS_{\alpha}$
 the map $\mathrm{Sx}$ is injective.
Therefore
\[ \dird[\Proj_{\alpha}[p],q,h] = \Proj_{\alpha}[\dird[p, q, h]] = 0\] if and only if
\[ \mathrm{Sx} \circ \Proj_{\alpha}[\dird[p,q, h]] = \dird[\mathrm{Sx} \circ \Proj_{\alpha}[p],q, h] = 0.\]
\end{enumerate}
\end{proof}

Here is an extension of Theorem \ref{thm:main} to $\FFNx$.

\begin{corollary}
\label{cor:extMainNS}
Let $x = (x_1, \ldots, x_g)$ and let $\ell$ be a positive integer.
\begin{enumerate}
\item
\label{item:nslowdegmain}
Every polynomial in $\CCNx$ of degree less than $\ell$ is $\ell$-harmonic.
\item \label{item:nsmiddledegmain}If $g \geq d $, then every NC,
homogeneous degree $d$, $\ell$-harmonic polynomial in $\CCNx$ is a sum of
polynomials of the form
\begin{equation}
\label{eq:corNSmainthm}
\prod_{i=1}^d \sum_{j=1}^g a_{ij}x_j^{\alpha_i},
\end{equation}
such that $\sum_{j=1}^g a_{k_1 j} \ldots a_{k_{\ell} j} = 0$ for each $1 \leq
k_1 < \ldots < k_{\ell} \leq d$, and such that $\alpha_i \in \{1,T\}$.

\item \label{item:nshighdegmain} If $d > \max\{\ell,g\}$, then
the space $\CCNx$ can be viewed as a subspace of $\CC \langle x_1, \ldots, x_d,
x_1^T, \ldots, x_d^T \rangle$.  Therefore,
every NC, homogeneous degree $d$, $\ell$-harmonic polynomial in $\CCx \subset
\CC\langle x_1, \ldots, x_d, x_1^T, \ldots, x_d^T \rangle$ is a sum of
polynomials of the form
\[\prod_{i=1}^d \sum_{j=1}^d a_{ij}x_j^{\alpha_i}, \]
such that $\sum_{j=1}^d a_{k_1 j} \ldots a_{k_{\ell} j} = 0$ for each $1 \leq
k_1 \leq \ldots \leq k_{\ell} \leq d$, and such that $\alpha_i \in \{1,T\}$.
\end{enumerate}
\end{corollary}

\begin{proof}
Let $p \in \CCNx$ be $\ell$-harmonic.
 By Proposition \ref{prop:symm2NS} part \ref{item:sumOfAlphaSoln}, for each
$\alpha$ the polynomial $\Proj_{\alpha}[p]$ is $\ell$-harmonic.
By Proposition \ref{prop:symm2NS} part \ref{item:alphaPartIsSymmSoln}, the
polynomial $\mathrm{Sx}[\Proj_{\alpha}[p]] \in \CCx$ is $\ell$-harmonic for
each $\alpha$.
Since $p$ is equal to
\[ p = \sum_{\alpha \in \{1, T\}^d} \mathrm{Sx}[\Proj_{\alpha}[p]]^\alpha\]
we see $p$ is simply a sum of polynomials of the form $q^{\alpha}$, where $q
\in \CCx$ is $\ell$-harmonic.
Taking
polynomials of the form
(\ref{eq:mainthm}) and exponentiating them
with $\alpha$ gives the corollary.
\end{proof}

\section{Noncommutative Subharmonic Polynomials}
\label{sect:HarmAndSub}

This section discusses NC subharmonic polynomials.
The notion of a NC Laplacian, NC
harmonic, and NC subharmonic polynomials was first introduced
in \cite{HMH09}.
As mentioned previously,
  the Laplacian of a NC polynomial $p$ is
 defined as:
\begin{align}
\lap[p, h] &:= \dird\left[p, \summ{i=1}{g} x_i^2, h\right] \label{eq:dirdysub}\\
&=
\summ{i=1}{g}\frac{d^2}{dt^2}
[p(x_1,\ldots,(x_i+th),\ldots,x_g)]_{|_{t=0}}.
\end{align}
When the variables commute,
$\lap[p,h]$ is
$h^2
\Delta [p]
$
where $\Delta [p]$ is the standard Laplacian on $\RR^n$
\[
\Delta [p] := \summ{i=1}{g}\frac{\partial^{\ell} p}{\partial x_i^{\ell}}.
\]

A NC  polynomial is called {\bf harmonic} \index{harmonic}
if its Laplacian is zero, i.e., if it is $\ell$-harmonic
for $\ell = 2$. A NC polynomial is called
 {\bf subharmonic} \index{subharmonic}
if its Laplacian\index{Laplacian} is matrix positive---or equivalently, if
its Laplacian is a finite sum of squares $\sum_i P_i^T P_i$ (see \cite{H02}). A subharmonic
polynomial is called
{\bf purely subharmonic} \index{subharmonic!purely subharmonic}
 if it is not harmonic, i.e.,
if its Laplacian is nonzero and
matrix positive.

\subsection{Basic Properties and Results}

\subsubsection{Classification of NC Harmonic
Polynomials}

Theorem \ref{thm:main} with $\ell=2$
classifies all NC harmonic polynomials.

\begin{corollary}
\label{cor:main}
Let $x = (x_1, \ldots, x_g)$ and let $\ell$ be a positive integer.
\begin{enumerate}
\item
\label{item:lowdegmain2}
Every polynomial in $\CCx$ of degree less than $2$ is harmonic.
\item \label{item:middledegmain2}If $g \geq d \geq 2$, then every NC, homogeneous degree $d$, harmonic polynomial in $\CCx$ is a sum of polynomials of the form
\begin{equation}
\label{eq:mainthmell2}
\prod_{i=1}^d \sum_{j=1}^g a_{ij}x_j,
\end{equation}
such that $\sum_{j=1}^g a_{k_1 j}a_{k_2 j} = 0$ for each $1 \leq k_1 <  k_2 \leq d$.

\item \label{item:highdegmain2}If $d > \max\{2,g\}$, then
the space $\CCx$ can be viewed as a subspace of $\CC \langle x_1, \ldots, x_d \rangle$.  Therefore,
every NC, homogeneous degree $d$, harmonic polynomial in $\CCx \subset \CC\langle x_1, \ldots, x_d\rangle$ is a sum of polynomials of the form
\[\prod_{i=1}^d \sum_{j=1}^d a_{ij}x_j, \]
such that $\sum_{j=1}^g a_{k_1 j}a_{k_2 j} = 0$ for each $1 \leq k_1 <  k_2 \leq d$.
\end{enumerate}
\end{corollary}

The two-dimensional case, as stated below, is proven in Theorem 1 of \cite{HMH09}, and is equivalent to Proposition \ref{prop:twovargenell} with $\ell = 2$.

\begin{prop}
\label{prop:2dimensional}
Let $x = (x_1, x_2)$.
The NC harmonic polynomials $p \in \CCx$ of homogeneous degree $d>2$
are exactly the linear combinations of
\begin{equation}
\re((x + i y)^d) \quad \text{and} \quad \im((x + i y)^d),
\eeq\end{equation}
\end{prop}

\subsubsection{The Laplacian of a Product}

\begin{lemma}[Lemma 2.2 in
\cite{HMH09}]
\label{lem:prodRulap}
The product rule for the Laplacian of NC polynomials is
\begin{equation}
\label{eq:laplacianproductrule}
\lap[p_1 \, p_2,h] = \lap[p_1,h]\, p_2 + p_1\ncm \lap[p_2,h]
     + 2\summ{i=1}{g}\bigl( \dird[p_1,x_i,h]\ncm\dird[p_2,x_i,h] \bigr).
\eeq\end{equation}
As a consequence if $p$ is harmonic, then
\begin{equation}
\label{eq:lapSq}
\lap[p^T \, p,h] =
      2\summ{i=1}{g}\bigl( \dird[p,x_i,h]^T \ncm\dird[p,x_i,h] \bigr).
\eeq\end{equation}
\end{lemma}

\proof
\begin{align*}
\label{eq:LapPros}
\lap[p_1 \, p_2,h]
& = \summ{i=1}{g} \dird[\dird[p_1\,p_2,x_i,h],x_i,h] \\
&= \summ{i=1}{g}  \dird[p_1\,\dird[p_2,x_i,h] +
      \dird[p_1,x_i,h]\,p_2, \ x_i,h] \\
&= \summ{i=1}{g} \bigl(p_1\,\dird[\dird[p_2,x_i,h],x_i,h] +
     \dird[\dird[p_1,x_i,h],x_i,h]\,p_2 \\
& \qquad \qquad + 2\dird[p_1,x_i,h],\dird[p_2,x_i,h]\bigr)\\
&= \lap[p_1,h]\, p_2 + p_1\ncm \lap[p_2,h]
     + 2\summ{i=1}{g}\bigl( \dird[p_1,x_i,h]\ncm\dird[p_2,x_i,h] \bigr).
\end{align*}
\qed

\subsection{Classifying Subharmonic Polynomials}

\subsubsection{Preliminary Results}

The results of this subsection extend results
presented in \cite{HMH09}.

\begin{lemma}
\label{lemma:symm}
Let $x = (x_1, \ldots, x_g)$. If $p \in \RRx$ or $\RRNx$
is subharmonic, then $\displaystyle\frac{p + p^T}{2}$ is symmetric and subharmonic and
$\displaystyle\frac{p - p^T}{2}$ is harmonic.
\end{lemma}

\begin{proof}First of all, it is straightforward to show that for any $p \in \RRx$ or $\RRNx$ and any variable $x_i$ that $\dird[p^T, x_i, h] = \dird[p,x_i,h]^T$. This implies in particular that
$\lap[p^T, h] = \lap[p, h]^T.$
The polynomial $p$ being subharmonic implies that
$\lap[p, h]$ is matrix positive, or equivalently, that
$\lap[p, h]$ is a sum of hermitian squares.  In particular,
this implies that $\lap[p, h] = \lap[p,h]^T$.
Therefore,
\begin{align}
\notag
\lap\left[\frac{p + p^T}{2},h \right] &= \frac{\lap[p,h] + \lap[p,h]^T}{2} = \lap[p,h] \succeq 0\\
\notag
\lap\left[\frac{p - p^T}{2},h \right] &= \frac{\lap[p,h] - \lap[p,h]^T}{2} = 0.
\end{align}
\end{proof}

\begin{prop}
\label{SosHarmPlusHarm}
Let $x = (x_1, \ldots, x_g)$. Suppose $p \in \RRx$ or $\RRNx$ is a polynomial of the form
\[p = \sum_i^{finite} R_i^T R_i + H \]
where the polynomials $R_i$ and $H$ are harmonic.
Then $p$ is subharmonic.
\end{prop}

\begin{proof}
This observation is proven in \cite{HMH09}.

Applying the Laplacian to $p$ gives
\begin{align}
\lap\left[\sum_i^{finite} R_i^T R_i + H, h\right] &= \sum_i^{finite}\lap[R_i^T , h]R_i + \sum_i^{finite}R_i^T \lap[R_i, h]
\label{eq:sossubh}\\
 &+\lap[H,h] + 2\sum_i^{finite}\sum_{j=1}^g \dird[R_i, x_j,h]^T \dird[R_i, x_j,h].\notag
\end{align} Since the $R_i$ and $H$ are harmonic, (\ref{eq:sossubh}) simplifies to
\begin{equation}
\lap\left[\sum_i^{finite} R_i^T R_i + H, h\right] = 2\sum_i^{finite}\sum_{j=1}^g \dird[R_i, x_j,h]^T \dird[R_i, x_j,h],
\end{equation}  which is a sum of squares.
\end{proof}

\begin{prop}
\label{prop:noneofodddegree}There are no pure subharmonic, homogeneous polynomials of odd degree.
\end{prop}

\begin{proof}This also appears in \cite{HMH09}. If $p$ is of odd degree and $\lap[p, h]$
is nonzero, then $\lap[p, h]$ has odd degree, so it cannot be a sum
of squares.
\end{proof}

\subsubsection{A Subharmonic Polynomial Has Harmonic Neighbors}

We use the following proposition, which restates Theorem 2 of \cite{HMH09}.

\begin{prop}[Theorem 2 of \cite{HMH09}]
\label{thm:mainsub}
Let $\mathcal{A} = \RRx$ or $\RRNx$
Let $x = (x_1, \ldots, x_g)$,
and let $\{\h_1, \ldots, \h_k\}$ be a basis for the
set of homogeneous degree $d$ harmonic polynomials in $\mathcal{A}$.
If $p \in \mathcal{A}$ is subharmonic and homogeneous of degree $2d$, then
it has the form
\begin{equation}\label{eq:repp}
p = H(x)^T A H(x),
\end{equation}
where $H(x)$ is the vector
\[H(x) = \begin{pmatrix} \h_1\\ \vdots \\ \h_k \end{pmatrix} \]
and where $A \in \RR^{k \times k}$.
\end{prop}

To prove this proposition, consider the following lemma.

\begin{lemma}
\label{lemma:goodtechnical}
Let $x = (x_1, \ldots, x_g)$.
Let $V(x) = (v_1, \ldots, v_r)^T$
be a vector of linearly independent
homogeneous degree $d_r$ NC polynomials, and let  $W(x) = (w_1, \ldots, w_s)^T$ be a vector of linearly independent
homogeneous degree $d_s$ NC polynomials. If
 $A \in \RR^{r \times s}$, then
\[V(x)^T A W(x) = 0\]
if and only if $A = 0$.
\end{lemma}

\begin{proof}First consider the case where $V = M_{d_r}$ is a vector whose entries are all the monomials of degree $d_r$,
and where $W = M_{d_s}$ is a vector whose entries are all the monomials of degree $d_s$.  Let $A \in \RR^{r \times s}$ be such that
$M_r(x)^T A M_s(x) = 0$.  Each NC monomial $m$ of degree $d_r + d_s$ can be uniquely expressed as a product
\[m = m_r m_s, \]
where $m_r$ is a monomial consisting of the first $d_r$ entries of $m$ and
$m_s$ is a monomial consisting of the last $d_s$ entries of $m$.
Let $A_{m_r^T,m_s}$ be the entry of $A$ corresponding to $m_r^T$ on the left and $m_s$ on the right.
Then,
\[M_r(x)^T A M_s(x) = \sum_{|m_r| = d_r} \sum_{|m_s| = d_s} A_{m_r^T,m_s} m_rm_s = 0. \]
This implies that each $A_{m_r^T,m_s} = 0$, or in other words, that $A = 0$.

Next consider the case where the entries of $V$ span the set of all homogeneous degree $d_r$ polynomials, and where the entries of $W$ span the set of all homogeneous degree $d_s$ polynomials.  The elements of $V$ may be expressed uniquely as a linear combination of the elements of $M_r$.  Therefore there exists an invertible matrix $R$ such that
\[RV(x) = M_r(x). \]
Similarly, there exists an invertible matrix $S$ such that
\[SW(x) = M_s(x). \]
Therefore
\[ V(x)^T A W(x) = M_r^T R^T A S M_s = 0 \]
which implies that $R^TAS = 0$. Since $R$ and $S$ are invertible, it follows that $A = 0$.

Finally, consider the generic case.  Let $V'$ be a vector whose entries, together with the entries of $V$, form a basis for the set of homogeneous degree $d_r$ polynomials, and let $W'$ be a vector whose entries, together with the entries of $W$, form a basis for the set of homogeneous degree $d_s$ polynomials. Suppose $V(x)^TAW = 0$.  Then
\[V^T A W = \begin{pmatrix} V\\V' \end{pmatrix}^T \left( \begin{array}{cc}
A&0\\
0&0
\end{array}
\right)
\begin{pmatrix} W\\W' \end{pmatrix} =0,\]
which implies that $A = 0$.
\end{proof}

\begin{definition}Let $V(x) = (v_1, \ldots, v_k)^T$ be a vector of NC polynomials.
Let $\dird[V,x_i,h]$ denote
\[\dird[V,x_i,h]:= (\dird[v_1,x_i,h], \ldots, \dird[v_k,x_i,h])^T, \]
and let $\lap[V,h]$ denote
\[\lap[V,h] = (\lap[v_1,h], \ldots, \lap[v_k,h])^T. \]
\end{definition}

\begin{proof}[Proof of Proposition \ref{thm:mainsub}] \ \
This proof is based on the proof of Theorem 2 of \cite{HMH09}.
Let $Q(x)$ be a vector whose entries together with the entries of $H(x)$ form
a basis for the space of homogeneous degree $d$ polynomials.
Let $p$ be equal to
\[p = \begin{pmatrix}H(x)\\ Q(x) \end{pmatrix}^T
\left(\begin{array}{cc}
A&B\\
C&D
\end{array}
\right)
\begin{pmatrix}H(x)\\ Q(x) \end{pmatrix}. \]
It suffices to prove that $B,C,$ and $D = 0$.

By the product rule for Laplacian of Lemma \ref{lem:prodRulap},
\begin{align}
\label{eq:firstlineneighbor}
\lap[p,h] &= 2\sum_{i=1}^g  \begin{pmatrix}\dird[H(x),x_i,h]\\ \dird[Q(x),x_i,h] \end{pmatrix}^T
\left(\begin{array}{cc}
A&B\\
C&D
\end{array}
\right)
\begin{pmatrix}\dird[H(x),x_i,h]\\ \dird[Q(x),x_i,h] \end{pmatrix}\\
\label{eq:secondlineneighbor}
&+  \begin{pmatrix}\lap[H(x),h]\\ \lap[Q(x),h] \end{pmatrix}^T
\left(\begin{array}{cc}
A&B\\
C&D
\end{array}
\right)
\begin{pmatrix}H(x)\\ Q(x) \end{pmatrix}\\
\label{eq:thirdlineneighbor}
&+   \begin{pmatrix}H(x)\\ Q(x) \end{pmatrix}^T
\left(\begin{array}{cc}
A&B\\
C&D
\end{array}
\right)
\begin{pmatrix}\lap[H(x),h]\\ \lap[Q(x),h] \end{pmatrix}.
\end{align}
Since $p$ is subharmonic, $\lap[p,h]$ must be a sum of squares.  Since $\lap[p,h]$ is homogeneous of
degree $2d$ and homogeneous of degree $2$ in $h$, it must be that $\lap[p,h]$ is a sum of squares of homogeneous degree $d$ polynomials which are homogeneous of degree $1$ in $h$.  Therefore $\lap[p,h]$ is of the form
\begin{equation}
\label{eq:formofsubharmlap}\lap[p,h] = V(x,h)^T Z V(x,h),
\end{equation}
where $V(x,h)$ is a vector of homogeneous degree $d$ polynomials which are homogeneous of degree $1$ in $h$, and where $Z \succeq 0$.

One sees the following:
\begin{itemize}
\item The expression (\ref{eq:firstlineneighbor}) is spanned by terms with one $h$ within the first $d$
entries and one $h$ within the last $d$ entries.
\item The expression (\ref{eq:secondlineneighbor}) is spanned by terms with two $h$ within the first $d$
entries and no $h$ within the last $d$ entries.
\item The expression (\ref{eq:thirdlineneighbor}) is spanned by terms with no $h$ within the first $d$
entries and two $h$ within the last $d$ entries.
\end{itemize}
Since the terms of (\ref{eq:formofsubharmlap}) have one $h$ in the first $d$ entries, the expressions (\ref{eq:secondlineneighbor}) and (\ref{eq:thirdlineneighbor}) are each equal to $0$.
Consider (\ref{eq:secondlineneighbor}). The vector $\lap[H(x),h]$ must be zero since $H(x)$ is spanned by harmonic polynomials.
On the other hand, the entries of $\lap[Q(x),h]$ are linearly independent for the following reason:

Suppose $q_1, \ldots, q_n$ are the entries of $Q(x)$, which are assumed to be linearly independent. Let $\alpha_1, \ldots, \alpha_n$ be scalars such that
\[\alpha_1\lap[q_1,h] + \ldots + \alpha_n \lap[q_n,h] = \lap[\alpha_1 q_1 + \ldots + \alpha_n q_n, h]= 0. \]
This implies that $\alpha_1q_1 + \ldots + \alpha_nq_n$ is harmonic.  However, the elements of $Q(x)$ are linearly independent from the entries of $H(x)$, which span the set of NC harmonic, homogeneous degree $d$ polynomials.  Therefore $\alpha_1q_1 + \ldots + \alpha_nq_n = 0$, which implies, by linear independence, that each $\alpha_i = 0$.

Therefore,
\[\lap[Q(x),h]^T \left(\begin{array}{cc} C&D \end{array}\right)\begin{pmatrix}H(x)\\Q(x)\end{pmatrix}
=0. \]
By Lemma \ref{lemma:goodtechnical}, this implies that $C$ and $D$ are equal to zero.
A similar argument on (\ref{eq:thirdlineneighbor}) proves that $B$ also is equal to zero.
\end{proof}

\subsubsection{Classification of Select Classes of Subharmonic NC Polynomials}

In degree two, the NC subharmonic polynomials have a simple classification.

\begin{prop}
\label{prop:deg2subharm}
Let $x = (x_1, \ldots, x_g)$.
\begin{enumerate}
\item The degree $2$ subharmonic polynomials in $\RRx$ are
precisely those of the form
\begin{equation}
\label{eq:formsubdegree2}
A x_1^2 + H
\end{equation}
for a constant $A \geq 0$ and a NC harmonic polynomial $H \in \RRx$.
\item The degree $2$ subharmonic polynomials in $\RRNx$ are
precisely those of the form
\begin{equation}
\label{eq:formsubdegree2NS}
A x_1^Tx_1 + B(x_1x_1 +x_1^Tx_1^T) + C x_1x_1^T + H
\end{equation}
for a constant matrix
\[\left(
\begin{array}{cc}
A&B\\
B&C
\end{array}\right) \] which is positive semidefinite and a NC harmonic polynomial $H \in \RRNx$.
\end{enumerate}
\end{prop}

\begin{proof}\begin{enumerate}
\item This is proven in part (2a) of Theorem 1 of \cite{HMH09}.

Let $p \in \RRx$ be a degree $2$ polynomial. Express $p$ as
\[p = \sum_{i,j} A_{ij} x_ix_j + L(x),\]
where $L(x)$ is affine linear.
Then $p$ equals
\begin{equation}
\label{eq:deg2subharm}
p = L(x) + \sum_{i \neq j} A_{ij}x_ix_j + \sum_{i=2}^g A_{ii}(x_i^2 - x_1^2) +
\left(\sum_{i=1}^g A_{ii}\right)x_1^2.
\end{equation}
Since the first three terms of (\ref{eq:deg2subharm}) are harmonic,
the Laplacian of $p$ is \[\lap[p, h] = \left(\sum_{i=1}^g A_{ii}\right)h^2\]
which is matrix positive if and only if $A:= \sum_{i=1}^g A_{ii} \geq 0$.

\item
Let $p \in \RRNx$ be a degree $2$ polynomial. Express $p$ as
\[p = \sum_{\alpha \in \{1,T\}^2} \sum_{i,j} A_{ij\alpha} x_ix_j^{\alpha} + L(x),\]
where $L(x)$ is affine linear.
Then $p$ equals
\begin{align}
\label{eq:deg2subharmNS}
p &= L(x) + \sum_{\alpha \in \{1,T\}^2} \sum_{i \neq j} A_{ij\alpha}x_ix_j^{\alpha} + \sum_{\alpha \in \{1,T\}^2}\sum_{i=2}^g A_{ii\alpha}(x_i^2 - x_1^2)^{\alpha} \\
\notag &+ \sum_{\alpha \in \{1,T\}^2} \left(\sum_{i=1}^g A_{ii\alpha}\right)(x_1^2)^{\alpha}.
\end{align}
Since the first three terms of (\ref{eq:deg2subharmNS}) are harmonic,
the Laplacian of $p$ is \[\lap[p] = 2\begin{pmatrix}
h\\h^T
\end{pmatrix}^T
\left(\begin{array}{cc}
A&B\\
B&C
\end{array}\right)
\begin{pmatrix}
h\\h^T
\end{pmatrix}
, \]
where
\[
\left(
\begin{array}{cc}
A&B\\
B&C
\end{array}\right)
:= \left(
\begin{array}{cc}
\sum_{i=1}^g A_{ii(T,1)}&\sum_{i=1}^g A_{ii(1,1)}\\
\sum_{i=1}^g A_{ii(T,T)}&\sum_{i=1}^g A_{ii(1,T)}
\end{array}\right)
 \] is positive semidefinite if and only if $p$ is subharmonic.
\end{enumerate}
\end{proof}

Homogeneous subharmonic polynomials $p \in \RR\langle x_1, x_2 \rangle$
all have a convenient form.

\begin{prop}
\label{prop:2varsub}The homogeneous degree $2d$ subharmonics in
$\RR\langle x_1, x_2 \rangle$
are precisely those of the form
\begin{equation}
\label{eq:sosplusharm}
 p=  \sum_i^{finite}  R_i^T R_i +  H
\end{equation}
where the $R_i$ are harmonic homogeneous degree $d$ polynomials
and where $H$ is a NC harmonic, homogeneous degree $2d$ polynomial.
\end{prop}

\begin{proof}
This proposition is Theorem 1 of \cite{HMH09} for $d \neq 2$.  It will be proven
here for all $d$.  The $d \neq 2$ case will be shown here for the convenience of
the reader.

If $p$ is a degree $2$
subharmonic then Proposition \ref{prop:deg2subharm}
is of the form (\ref{eq:sosplusharm}).

If $p$ is of degree $2d$, $d > 2$, then
by Propositions \ref{prop:2dimensional} and \ref{thm:mainsub} it is of the form
\[A_{11}\re((x + i y)^d)\re((x+iy)^d) + A_{12}\re((x + i y)^d)\im((x + i y)^d) \]
\[+ A_{21}\im((x + i y)^d)\re((x + i y)^d) + A_{22} \im((x + i y)^d)\im((x + i y)^d).\]
The polynomials
\[\re((x + i y)^d)\im((x + i y)^d),\quad \im((x + i y)^d)\re((x + i y)^d),\]
\[ and \quad \im((x + i y)^d)\im((x + i y)^d) - \re((x + i y)^d)\re((x+iy)^d)\]
are harmonic, which means $p$ is of the form
\[p = (A_{11} - A_{22})\re((x + i y)^d)\re((x+iy)^d) + H \]
with $H \in \RRx$ harmonic. Therefore
\begin{align}
\lap[p, h] &= (A_{11} - A_{22})\lap[\re((x + i y)^d)\re((x+iy)^d), h]\\
& = 2(A_{11} - A_{22})\sum_{j=1}^g \dird[\re((x + i y)^d), x_j, h]^2
\end{align}
which is matrix positive if and only if $A_{11} - A_{22} \geq 0$.

The $d = 2$ case is more difficult.
A polynomial $p$ is a sum of squares if it can
be expressed in the form $p = V^TAV$ for some vector of
polynomials $V$ and some positive semidefinite constant
matrix $A$.
By Proposition \ref{prop:2dimensional}, the NC harmonic polynomials of homogeneous degree 4 are
linear combinations of
\[\h_1 = \Symm[x_1^4 - 6x_1^2x_2^2 + x_2^4] \quad and \quad \h_2 = \Symm[x_1^3x_2 - x_2^3x_1]\]
by Proposition \ref{prop:2dimensional}.
Express the set of homogeneous degree 4
harmonics $H$ as
\begin{equation}
\label{eq:harmonicsdegree4} H = \begin{pmatrix}
x_1^2 - x_2^2\\ x_1 x_2 + x_2 x_1\\x_1x_2
\end{pmatrix}^T \left( \begin{array}{ccc}
\alpha&\beta&0\\
\beta&-\alpha&0\\
0&0&0 \end{array} \right)\begin{pmatrix}
x_1^2 - x_2^2\\ x_1 x_2 + x_2 x_1\\x_1x_2
\end{pmatrix}
\end{equation}
for arbitrary scalars $\alpha, \beta$ corresponding
to $\h_1$ and $\h_2$ respectively.

Let $p \in \RRx$ be a
NC, degree 4, subharmonic
polynomial. By Lemma \ref{lemma:symm},
$p$ may be expressed as a symmetric polynomial plus
a harmonic polynomial.
Given this and given (\ref{eq:harmonicsdegree4}),
$p$ is equal to
\[p = H_1 + \begin{pmatrix}
x_1^2 - x_2^2\\ x_1 x_2 + x_2 x_1\\x_1x_2
\end{pmatrix}^T \left( \begin{array}{ccc}
0&0&a\\
0&b&c\\
a&c&d \end{array} \right)\begin{pmatrix}
x_1^2 - x_2^2\\ x_1 x_2 + x_2 x_1\\x_1x_2
\end{pmatrix}\]
for some harmonic polynomial $H_1$.
In this case, $\lap[p, h]$ is equal to
\begin{equation}
\label{eq:subtheabove}
\lap[p, h] = \begin{pmatrix}
x_1 h + h x_1\\x_2h + hx_2\\x_1h\\hx_2
\end{pmatrix}^T \left( \begin{array}{cccc}
b&0&c&a\\
0&b&-a&c\\
c&-a&d&0\\
a&c&0&d \end{array} \right)\begin{pmatrix}
x_1 h + h x_1\\x_2h + hx_2\\x_1h\\hx_2
\end{pmatrix}.
\end{equation}
Since $p$ is subharmonic, the polynomial (\ref{eq:subtheabove}) must be a sum of squares,
which implies that the
constant matrix in (\ref{eq:subtheabove}) is positive semidefinite.
Since the diagonal entries of a positive semidefinite matrix are non-negative, $b, d \geq 0$.
Further, if a diagonal entry of a positive semidefinite matrix is $0$, then the row and column entries corresponding to that diagonal entry are also $0$.
Therefore, if either $b$ or $d$ equals zero, then both $a$ and
$c$ are zero, in which case the result is trivial.  Otherwise
\[ \left( \begin{array}{cccc}
\frac{1}{\sqrt{b}}&0&0&0\\
0&\frac{1}{\sqrt{b}}&0&0\\
0&0&\frac{1}{\sqrt{d}}&0\\
0&0&0&\frac{1}{\sqrt{d}}\end{array} \right)\left( \begin{array}{cccc}
b&0&c&a\\
0&b&-a&c\\
c&-a&d&0\\
a&c&0&d \end{array} \right)\left( \begin{array}{cccc}
\frac{1}{\sqrt{b}}&0&0&0\\
0&\frac{1}{\sqrt{b}}&0&0\\
0&0&\frac{1}{\sqrt{d}}&0\\
0&0&0&\frac{1}{\sqrt{d}}\end{array} \right)\]
\begin{equation}
\label{eq:conjmatrix}
=\left( \begin{array}{cccc}
1&0&r \sin(\theta)&r \cos(\theta)\\
0&1&-r \cos(\theta)&r \sin(\theta)\\
r \sin(\theta)&-r \cos(\theta)&1&0\\
r \cos(\theta)&r \sin(\theta)&0&1 \end{array} \right) \succeq 0,
\end{equation}
where
\[r \cos(\theta) = \frac{a}{\sqrt{bd}}
\quad
and
\quad
r \sin(\theta) = \frac{c}{\sqrt{bd}},\]
with $r \geq 0$.
All of the principal minors of (\ref{eq:conjmatrix}) must be positive, so
\[\left|\begin{array}{ccc}
1&0&r \sin(\theta)\\
0&1&-r \cos(\theta)\\
r \sin(\theta)&-r \cos(\theta)&1
 \end{array}\right| = 1 - r^2 \geq 0. \]
Therefore
$0 \leq r \leq 1$.

Therefore $p$ equals
\begin{equation}
\label{eq:involvesMatrixA}p = H + \begin{pmatrix}
x_1^2 - x_2^2\\ x_1 x_2 + x_2 x_1\\x_1x_2
\end{pmatrix}^T
\left( \begin{array}{ccc}
b\cos^2(\theta)&b\cos(\theta)\sin(\theta)&\sqrt{bd}r\cos(\theta)\\
b\cos(\theta)\sin(\theta)&b\sin^2(\theta)&\sqrt{bd}r\sin(\theta)\\
\sqrt{bd}r\cos(\theta)&\sqrt{bd}r\sin(\theta)&d
\end{array} \right)
\begin{pmatrix}
x_1^2 - x_2^2\\ x_1 x_2 + x_2 x_1\\x_1x_2
\end{pmatrix}
\end{equation}
for some NC harmonic polynomial $H \in \RRx$.
The matrix in (\ref{eq:involvesMatrixA}) is equal to
\begin{align} r\left( \begin{array}{ccc}
\sqrt{b}\cos(\theta)&0&0\\
0&\sqrt{b}\sin(\theta)&0\\
0&0&\sqrt{d}\end{array}\right)
\left( \begin{array}{ccc}
1&1&1\\
1&1&1\\
1&1&1 \end{array} \right)
\left( \begin{array}{ccc}
\sqrt{b}\cos(\theta)&0&0\\
0&\sqrt{b}\sin(\theta)&0\\
0&0&\sqrt{d}\end{array}\right)
+\notag\\ (1-r)\left( \begin{array}{ccc}
\sqrt{b}\cos(\theta)&0&0\\
0&\sqrt{b}\sin(\theta)&0\\
0&0&\sqrt{d}\end{array}\right)
\left( \begin{array}{ccc}
1&1&0\\
1&1&0\\
0&0&1
\end{array} \right)
\left( \begin{array}{ccc}
\sqrt{b}\cos(\theta)&0&0\\
0&\sqrt{b}\sin(\theta)&0\\
0&0&\sqrt{d}\end{array}\right)\notag
\end{align}
which is positive semidefinite.
\end{proof}

\begin{exa}
\label{exa:countersubSOS}
In higher dimensions, not all homogeneous subharmonic polynomials are of the form of Proposition \ref{SosHarmPlusHarm}.
For example, consider the polynomial
$p(x)$ defined by
 \[p(x) = (x_1x_2x_2x_1 + x_2x_1x_1x_2) + (x_1x_3x_3x_1 + x_3x_1x_1x_3) - (x_2x_3x_3x_2 + x_3x_2x_2x_3).\]
The Laplacian of $p$ is equal to
\[\lap[p,h] = 4(x_1hhx_1 + hx_1x_1h)\]
which is a sum of squares.

Fix a positive integer $g$. Let $V_g(x)$ be the vector
\[V_g(x) = (x_1^2 - x_2^2, \ldots, x_1^2 - x_g^2, x_1x_2, \ldots, x_gx_{g-1})^T.\]
The entries of $V_g(x)$ form a basis for the set
of degree $2$ homogeneous harmonic polynomials in $\RR\langle x_1, \ldots, x_g \rangle$.
Assume by contradiction that there exists a harmonic polynomial $H \in \RRx$ such that
$p + H$ equals
\[p + H = V_g^T A V_g\]
with $A \succeq 0$.
 Let $\alpha_{k}$ be the diagonal entry of $A$
corresponding to $x_1^2 - x_k^2$. Let $C_{ij}$, where $i < j$, be the $2 \times 2$ block on the diagonal
of $A$
corresponding to $(x_ix_j, x_jx_i)$.
Since $A \succeq 0$, each $\alpha_{k} \geq 0$ and
each $C_{ij} \succeq 0$.

Express $\lap[p, h]$ as
\[\lap[p, h] = W_g(x)^T B W_g(x) \]
where $W_g(x) = (x_1h, hx_1, \ldots, x_gh, hx_g)^T$.
Let $\beta_{i}$ be the block on the diagonal of $B$ corresponding to
$(x_ih, hx_i)$.
Since $p$ is subharmonic, $\beta_i \succeq 0$ for each $i$.

By Lemma \ref{lem:prodRulap},
\[\lap[p,h] = 2\sum_{j=1}^g \dird[V_g(x), x_j,h]^T A \dird[V_g(x), x_j,h]. \]
For a given variable $x_i$, with $i > 1$, the terms of $\lap[p,h]$ with degree two in $x_i$ are
\begin{align}
\label{eq:thetermsdegree2firstversion}
2 \sum_{i < j}\begin{pmatrix}\dird[x_ix_j,x_j,h]\\\dird[x_jx_i,x_j,h]
\end{pmatrix}^T C_{ij} \begin{pmatrix}\dird[x_ix_j,x_j,h]\\\dird[x_jx_i,x_j,h]
\end{pmatrix}\\
\notag + 2 \sum_{i > k}
\begin{pmatrix}\dird[x_kx_i,x_i,h]\\\dird[x_ix_k,x_i,h]
\end{pmatrix}^T C_{ki} \begin{pmatrix}\dird[x_kx_i,x_i,h]\\\dird[x_ix_k,x_i,h]
\end{pmatrix}\\
\notag + 2(\dird[x_1^2 - x_i^2, x_i, h])^T\alpha_i (\dird[x_1^2 - x_i^2, x_i, h]).
\end{align}
Alternately, (\ref{eq:thetermsdegree2firstversion}) may be expressed as
\[\begin{pmatrix} x_i h\\ h x_i \end{pmatrix}^T
\beta_i \begin{pmatrix} x_i h\\ h x_i \end{pmatrix}. \]
If $i > 1$, since $\lap[p,h] = 4(x_1hhx_1 + hx_1x_1h)$ has no terms with degree $2$ in $x_i$,
\begin{equation}
\label{eq:betaCijbi}0 =
\beta_i = 2\sum_{i < j} C_{ij} +  \left(
\begin{array}{cc}
0&1\\
1&0
\end{array}
\right)\left(2\sum_{i>k} C_{ki} \right)
\left(
\begin{array}{cc}
0&1\\
1&0
\end{array}
\right) + 2\alpha_i\left(
\begin{array}{cc}
1&1\\
1&1
\end{array}
\right).
\end{equation}
Each of the terms of the righthand side of (\ref{eq:betaCijbi})
is by assumption positive semidefinite, they must all be zero, that is,
$C_{ij} = 0$ for all $j > i$, $C_{ki} = 0$ for all $k < i$, and $\alpha_i = 0$.
Therefore all of the diagonal entries of $A$ are equal to $0$. Since $A \succeq 0$, this implies that $A = 0$.
This implies that $\lap[p,h] = 0$,
which is a contradiction.\qed
\end{exa}

\begin{definition}A NC polynomial $p \in \RRx$ or $\RRNx$ is \textbf{bounded below}\index{bounded below} if there exists a constant $C \in \RR$ such that $p + C$ is matrix positive.
\end{definition}

\subsubsection{Proof of Theorem \ref{thm:bddbelow}}
\label{subsub:pfOfBddBelow}

The following is the proof Theorem \ref{thm:bddbelow}.

\begin{proof} Suppose for some constant $C$
 that $p+C$ is matrix positive, or equivalently, that it is a finite sum of
squares, i.e. that $p + C$ can be expressed as
\begin{equation}
\label{eq:pPlusCSOS}
 p+ C = \sum_{finite} q_i^Tq_i 
\end{equation}
for some $q_i$.
(see \cite{H02}).
Since $p$ is homogeneous of degree $2d$,
this implies that each $q_i$ must be at most degree $d$ and that at least one
$q_i$ actually is of degree $d$.
For each $q_i$ in (\ref{eq:pPlusCSOS}), let $q_i'$ a homogeneous degree $d$
polynomial such that $\deg(q_i - q_i') < d$.
Then
\begin{align}
 \label{eq:separatePPlusC} p + C = \sum_{finite} (q_i')^T(q_i') + \sum_{finite}
\left([q_i-q_i']^T[q_i-q_i'] + [q_i']^T[q_i-q_i'] + [q_i']^T[q_i-q_i']\right).
\end{align}
Since the only terms on the right-hand side of (\ref{eq:separatePPlusC}) which
are degree $2d$ are those of $\sum_{finite} (q_i')^T(q_i') $, and since $p$ is
homogeneous of degree $2d$ and $C$ is constant,
this implies that $p = \sum_{finite} (q_i')^T(q_i')$.

If $V(x)$ is a vector whose entries form a basis for the set of homogeneous
degree $d$ polynomials, then $q'_i = \alpha_i^T V(x)$ for some scalar vectors
$\alpha_i$ and
\[ p = V(x)^T\left( \sum_{finite} \alpha_i \alpha_i^T \right) V(x).\]
By Lemma \ref{lemma:goodtechnical}, the matrix
\begin{equation}
 \label{eq:theUniqueRepMatrix}
\sum_{finite} \alpha_i \alpha_i^T \succeq 0
\end{equation}
is a unique matrix for the vector $V(x)$ and $p$.

Since $p$ is subharmonic,
by Proposition \ref{thm:mainsub} it may be expressed as
\[p = H(x)^T A H(x), \]
where $H(x)$ is a vector of linearly independent, homogeneous degree $d$,
harmonic polynomials.  Lemma \ref{lemma:goodtechnical} implies that the matrix
$A$ is unique.
 By
(\ref{eq:theUniqueRepMatrix}), $A$ must be positive semidefinite. 
Therefore
\[p = \left(\sqrt{A}H(x) \right)^T\left(\sqrt{A}H(x)\right), \]
is a finite sum of squares of homogeneous degree $d$ harmonic polynomials.
\end{proof}

\subsection{Subharmonic Polynomials in Nonsymmetric Variables}

\begin{definition}
For a $d$-tuple $\alpha = (\alpha_1, \ldots, \alpha_{d}) \in \{1,T\}^d$ define \index{transpose!of a $d$-tuple in $\{1,T\}^d$}
$\alpha^{T} \in \{1,T\}^d$ as
\begin{equation}
\alpha^{T} = (\alpha_{d}T, \alpha_{d-1}T, \ldots, \alpha_1 T),
\end{equation}
where $1T:=T$ and $TT = T^2:=1$.
If $m \in \RRx$ is homogeneous of degree $d$ 
and $\alpha \in \{1,T\}^d$,
then $(m^{\alpha})^T = (m^T)^{\alpha^T}$.
(Recall that $m^{\alpha}$ is defined in (\ref{eq:defOfMAlpha}).)
If $\alpha = \alpha^{T}$ call $\alpha \in \{1,T\}^d$
\textbf{symmetric}\index{symmetric!$d$-tuple in $\{1,T\}^d$}.
\end{definition}

\begin{exa}Let $\alpha = (0,0,T,T)$ so that $\alpha \in \{1,T\}^d$ is symmetric.
The polynomial $(x_1^2 - x_2^2)(x_1^2 - x_2^2) \in \RRx$ is symmetric and
subharmonic.
Also, \[(x_1^2 - x_2^2)(x_1^2 - x_2^2)^{\alpha} = (x_1^2 - x_2^2)(x_1^2 - x_2^2)^T\]
is also symmetric and subharmonic.\qed
\end{exa}

\begin{prop}If $p \in \RRx$ is a homogeneous degree $2d$
polynomial, and $\alpha \in \{1,T\}^d$ is symmetric, then $p^{\alpha} \in \RRNx$ is
subharmonic if and only if $p$ is.
\end{prop}

\begin{proof}
Suppose $\lap[p, h]$ is equal to
\[W^TBW
\]
for some vector $W$ of linearly
independent  homogeneous degree
$d$ polynomials and for some constant matrix $B$.
By Lemma \ref{lemma:goodtechnical}, the matrix $B$ is uniquely determined,
so $\lap[p, h]$ is a sum of squares if and only if $B\succeq 0$.
Further
\[\lap[p, h]^{\alpha} = \lap[p^{\alpha}, h] = (W^{\beta})^T B (W^{\beta})\]
where
\[\beta = (\alpha_{d+1}, \alpha_{d+2}, \ldots, \alpha_{2d})\]
is the last half of
$\alpha$. Therefore $p^{\alpha}$ is subharmonic if and only if
$B\succeq 0$, or equivalently, if and only if $p$ is subharmonic.
\end{proof}

 \subsection{Properties of Harmonic and Subharmonic Functions}

This subsection aims to extend results about harmonic and subharmonic
functions in commuting variables to NC harmonic and subharmonic polynomials.

Let $(X_1, \ldots, X_g) \in (\RR^{n\times n}_{sym})^g$
denote a $g$-tuple of symmetric matrices.
Each such tuple may be considered alternatively as an element of
$\RR^{\frac{gn(n+1)}{2}}$ with coordinates
$(X_k)_{ij} = x_{ijk}.$  Let $(Z_1, \ldots, Z_g) \in (\RR^{n \times n})^g$
denote a $g$-tuple of matrices.
Each such tuple may be considered alternatively as an element of
$\RR^{gn^2}$ with coordinates
$(Z_k)_{ij} = z_{ijk}.$

\begin{prop}
\label{prop:firstconnection}
Let $p \in \RRx$ and let $q \in \RRNx$.  Fix a dimension $n$.

\begin{enumerate}
\item If $p$ is harmonic, then the function
\[y_1^T p(X) y_2: (\RR^{n\times n}_{sym})^g = \RR^{\frac{gn(n+1)}{2}} \rightarrow \RR\]
is harmonic for each $y_1, y_2 \in \RR^n$.

\item If $p$ is subharmonic, then the function
\[y^T p(X) y: (\RR^{n\times n}_{sym})^g = \RR^{\frac{gn(n+1)}{2}} \rightarrow \RR\] is subharmonic
for each $y \in \RR^n$.

\item If $q$ is harmonic, then the function
\[y_1^T q(Z) y_2: (\RR^{n\times n})^g = \RR^{gn^2} \rightarrow \RR\]
is harmonic for each $y_1, y_2 \in \RR^n$.

\item If $q$ is subharmonic, then the function
\[y^T q(Z) y: (\RR^{n\times n})^g = \RR^{gn^2} \rightarrow \RR\] is subharmonic
for each $y \in \RR^n$.
\end{enumerate}
\end{prop}

\begin{proof}
First consider the symmetric variable case.
A partial derivative
$\displaystyle\frac{\partial p}{\partial x_{ijk}}$
is, in terms of the NC directional derivative notation, equal to
\[\dird[p, x_k, A_{ij}](X)\]
where $A_{ij}$ has a $1$ for the $ij^{th}$ and $ji^{th}$ entry, and $0$ for all other entries. If $i=j$, this means $A_{ij}$ has one lone nonzero entry at $ii$.
Applying $\Delta$ to $p$ gives
\begin{align}
\notag
\Delta\left[y_1^T p y_2\right] &= \sum_{i=1}^n \sum_{j = 1}^n \sum_{k=1}^g \frac{\partial^2}{\partial x_{ijk}^2} y_1^T p y_2
\\
\notag &= \sum_{i=1}^n \sum_{j = 1}^n \sum_{k = 1}^g \dird\left[\dird\left[y_1^T p y_2,  x_k, A_{ij}\right], x_k, A_{ij}\right](X)
 \\
 \notag &=\sum_{i=1}^n \sum_{j = 1}^n y_1^T \lap[p, A_{ij}](X) y_2^T.\end{align}
If $p$ is harmonic, then $\lap[p, A_{ij}](X) = 0$ for each $A_{ij}$.
If $p$ is subharmonic, then $\lap[p, A_{ij}](X) \succeq 0$ for each $A_{ij}$, which implies that
$y^T \lap[p, A_{ij}](X) y \geq 0$ for all vectors $y \in \RR^n$.

In the nonsymmetric variable case, a partial derivative
$\displaystyle\frac{\partial q}{\partial z_{ijk}}$
is, in terms of the NC directional derivative notation, equal to
\[\dird[q, x_k, E_{ij}](Z),\]
where $E_{ij}$ has a $1$ for the $ij^{th}$ entry, and $0$ for all other entries.
The proof follows similarly to the symmetric case.
\end{proof}

\begin{prop}Let $p \in \RRx$, let $q \in \RRNx$, let $y, y_1, y_2 \in \RR^n$ be nonzero vectors, and let
$\Omega \in (\RR^{n\times n}_{sym})^g$ be a connected
domain for some $n$.
\begin{enumerate}
\item Suppose $p$ is harmonic on $\Omega$.
If $y_1^T p(X) y_2$ has
a maximum or a minimum for $X \in (\RR^{n\times n}_{sym})^g$,
then $p$ is constant.

\item Suppose $p$ is subharmonic on $\Omega$. If $y^T p(X) y$ has
a maximum for $X \in (\RR^{n\times n}_{sym})^g$,
then $p$ is constant.

\item Suppose $q(Z)$ is harmonic on $\Omega$.
If $y_1^T q(Z) y_2$ has
a maximum or a minimum for $Z \in (\RR^{n\times n})^g$,
then $q$ is constant.

\item Suppose $q(Z)$ is subharmonic on $\Omega$. If $y^T q(Z) y$ has
a maximum for $Z \in (\RR^{n\times n})^g$,
then $q$ is constant.

\end{enumerate}
\end{prop}

\begin{proof}These items are essentially the Maximum Principle
 applied to $y_1^T p y_2$ and $y_1^T q y_2$.
\end{proof}

\begin{prop}Let $\Omega \in (\RR^{n\times n}_{sym})^g$ or $(\RR^{n\times n})^g$ be a bounded domain.
Let $p, q \in \RRx$ or $\RRNx$ respectively be such that $p$ is harmonic and $q$ is subharmonic.
If
\[q(X) \preceq p(X)\]
for all $X \in \partial \Omega$,  then
\[q(X) \preceq p(X) \]
for all $X \in \Omega$.
\end{prop}

\begin{proof}For each $y \in \RR^n$, $y^T p y$ is harmonic
and $y^T q y$ is subharmonic.
By assumption
\[y^Tq(X)y \leq y^Tp(X)y\]
for all $X \in \partial \Omega$.
Since $y^T p y$
and $y^T q y$ are simply functions
in commuting variables $x_{ijk}$, by the Maximum Principle,
\begin{equation}
\label{eq:ncharmandsub}
y^Tq(X)y \leq y^Tp(X)y
\end{equation}
for all $X \in \Omega$.
Since (\ref{eq:ncharmandsub})
is true for all $y \in \RR^n$,
\[q(X) \preceq p(X) \]
for all $X \in \Omega$.
\end{proof}

\section{Acknowledgments}

Author was supported by J.W. Helton's National Science
Foundation grants: DMS 07758 and DMS
0757212.  Additional thanks goes to J.W. Helton for
suggesting the topic of the paper and for all his work in helping the author
revise this paper.

The material in this paper is drawn from the author's  Ph.D. Thesis
at
University of California San Diego.

\newpage

\setcounter{tocdepth}{4}
\tableofcontents

\newpage
\printindex

\end{document}